\documentclass[11pt]{amsart}

\usepackage{mathpazo}
\usepackage{fullpage}

\usepackage{amsmath,amsfonts,amssymb}
\usepackage{amsrefs}
\usepackage{amsthm}
\usepackage{latexsym,amsmath,amssymb,amsfonts}
\usepackage{rotating}
\usepackage{mathrsfs}
\usepackage{xypic} \xyoption{all}
\usepackage{amscd}
\usepackage{hyperref} 
\usepackage{euscript}
\usepackage{hhline}
\usepackage{graphicx,epstopdf}
\usepackage{epsfig}
\usepackage{xcolor}
\usepackage[all,color]{xy}
\usepackage{mathtools}

\newlength{\fighskip} \fighskip=2pt
\newlength{\figvskip} \figvskip=3pt

\usepackage{hyperref}
\newcommand*{\figbox}[2]{{
  \def\figscale{#1}
  \def\arraystretch{0.8}
  \arraycolsep=0pt
  \begin{array}{c}
    \vbox{\vskip\figscale\figvskip
      \hbox{\hskip\figscale\fighskip
        \includegraphics[scale=\figscale]{#2}}}
  \end{array}}}

\numberwithin{equation}{section}

\newcommand{\C}{\mathbb{C}}

\newcommand{\R}{\mathbb{R}}

\newcommand{\Z}{\mathbb{Z}}

\newcommand{\E}{{\mathcal E}}

\renewcommand{\H}{\mathbb{H}}

\newcommand{\g}{\mathfrak{g}}

\newcommand{\abracket}[1]{\left\langle#1\right\rangle}
\newcommand{\bbracket}[1]{\left[#1\right]}
\newcommand{\fbracket}[1]{\left\{#1\right\}}
\newcommand{\bracket}[1]{\left(#1\right)}

\newcommand{\cinfty}{C^{\infty}}
\newcommand{\pa}{\partial}

\newcommand{\OO}{{\mathcal O}}
\newcommand{\cG}{{\mathcal G}}

\newcommand{\BV}{Batalin-Vilkovisky }
\newcommand{\CE}{Chevalley-Eilenberg }

\newcommand{\Ol}{\mathcal O_{loc}}

\newcommand{\iso}{\cong}

\newcommand{\W}{\mathcal{W}}

\DeclareMathOperator{\Sym}{Sym}
\DeclareMathOperator{\Hom}{Hom}

\DeclareMathOperator{\Tr}{Tr}

\newcommand{\A}{\mathcal A}

\renewcommand{\L}{\mathcal L}

\DeclareMathOperator{\Mult}{Mult}

\theoremstyle{plain}
\newtheorem{thm}{Theorem}[section]
\newtheorem{thm-defn}{Theorem/Definition}[section]
\newtheorem{lem}[thm]{Lemma}
\newtheorem{lem-defn}[thm]{Lemma/Definition}
\newtheorem{prop}[thm]{Proposition}
\newtheorem{cor}[thm]{Corollary}

\theoremstyle{definition}
\newtheorem{defn}[thm]{Definition}
\newtheorem{notn}[thm]{Notation}
\newtheorem{eg}[thm]{Example}

\theoremstyle{remark}
\newtheorem{rmk}[thm]{Remark}

\begin{document}
\title{Batalin-Vilkovisky quantization and the algebraic index}

\author{Ryan E. Grady}
\address{Montana State University\\Bozeman 59717\\USA}
\email{ryan.grady1@montana.edu} 

\author{Qin Li}
\address{Southern University of Science and Technology\\ Shenzhen\\ China}
\email{liqin@sustc.edu.cn}

\author{Si Li}
\address{YMSC\\Tsinghua University\\Beijing 100084\\China}
\email{sili@mail.tsinghua.edu.cn}

\thanks{R Grady and Si Li were partially supported by the National Science Foundation under Award DMS-1309118. Qin Li was supported by a grant from  National Natural Science Foundation of China (Project No.11501537).}

\begin{abstract}

Into a geometric setting,  we import the physical interpretation of index theorems via semi-classical analysis in topological quantum field theory. We develop a direct relationship between Fedosov's deformation quantization of a symplectic manifold $X$ and the \BV (BV) quantization of a one-dimensional sigma model with target $X$. This model is a quantum field theory of AKSZ type and is quantized rigorously using Costello's homotopic theory of effective renormalization. We show that Fedosov's Abelian connections on the Weyl bundle produce solutions to the effective quantum master equation. Moreover, BV integration produces a natural trace map on the deformation quantized algebra.  This formulation allows us to exploit a (rigorous) localization argument in quantum field theory to deduce the algebraic index theorem via  semi-classical analysis, i.e.,  one-loop Feynman diagram computations. 


\end{abstract}

\maketitle

\tableofcontents

\section{Introduction} 

Deformation quantization is one approach to encapsulating the algebraic aspects of observables in a quantum mechanical system.  More precisely, the Poisson algebra of classical observables is deformed to an associative algebra of quantum observables; the idea goes back to the work of Bayen, Flato, Fronsdal, Lichnerowicz, and Sternheimer \cite{5author}. By constructing a trace map on the algebra of quantum observables, correlation functions are defined.  In the early 1990's--using this paradigm--an algebraic analogue of the Atiyah-Singer index theorem was established by Fedosov \cite{Fedbook} and jointly by Nest and Tsygan \cite{Nest-Tsygan}. Beyond an analogy, Nest and Tsygan recover the Atiyah-Singer index theorem (as well as the index theorems of Connes and Moscovici) from the algebraic index theorem \cite{Nest-Tsygan2}. 

On the other hand, the Atiyah-Singer type index theorem has a  simple form in physics.  The index of an elliptic operator $D$ is expressed as an infinite dimensional path integral of a related action functional $S_D$
\[
\mathrm{index}(D) = \lim_{\hbar \to \infty}  \int e^{S_D/\hbar} .
\]
This path integral is independent of $\hbar$ for topological quantum field theories. Therefore it can be computed semi-classically, i.e.,  in the limit $\hbar\to 0$, which amounts to a one-loop Feynman diagram computation and produces exaclty the topological index \cite{Alvarez, FW,Witten, Witten-index}. Though it lacks some mathematical rigor, this physical approach provides a conceptually clean understanding of the origins of index theorems. 

We are interested in the relationship between the above two approaches to the Atiyah-Singer index theorem. 

The existence of deformation quantizations is highly nontrivial. In the general setting, where the classical observables are given by functions on a Poisson manifold, existence of a quantization was established by Kontsevich via the celebrated Formality theorem \cite{Kontsevich-deformation}; Kontsevich  provides an expression of the associative star product in terms of Feynman diagrams. Interestingly, the one dimensional (quantum mechanical) nature of this deformation quantization is realized via the boundary of the two dimensional Poisson sigma model on the disk \cite{Poisson-sigma}. In the symplectic case, the existence was first proved by De Wilde and Lecomte \cite{DL}. Later, Fedosov \cite{Fed} established a beautifully simple differential geometric approach by constructing a flat structure (an Abelian connetion) on the Weyl bundle of a symplectic manifold. This construction is also realized as a BRST quantization in \cite{Grigoriev-Lyakhovich}. In addition to Fedosov's geometric approach, trace maps in deformation quantization are often constructed via explicit (algebraic) computations in Hochschild/cyclic homology; this is the strategy employed in \cite{Nest-Tsygan}, as well as by Feigin, Felder, Shoiket  \cite{FFS} and Dolgushev, Rubtsov\cite{dolgushev}.

The purpose of this paper is to construct a rigorous quantum field theory (which we call one-dimensional Chern-Simons theory) that directly relates the algebraic index theorem with the above physics interpretation. The quantization of one-dimensional Chern-Simons theory requires quantum corrections at all loops to maintain quantum gauge symmetry (see Section \ref{sec: homotopic-BV} for the precise definitions). We show that such quantum corrections are produced at once by Fedosov's Abelian connections. The study of observables in our model leads naturally to a proof of the algebraic index theorem which has a flavor of the semi-classical method in topological quantum mechanics. 

More precisely,  by investigating the geometric content of our one-dimensional Chern-Simons theory, we extend Fedosov's construction to a flat structure on the bundle of \BV (BV) algebras (see Section \ref{sec-BV-bundle}), thereby providing the tools of BV integration \cite{Schwarz}. Fixing a Lagrangian super-submaninfold (see equation \eqref{BV-lagrangian}), yields naturally a BV integration (Definition \ref{defn:BVint}) and a trace map (Definition \ref{defn:trace})
\[
\Tr : C^\infty(M)[[\hbar]] \to \R((\hbar)).
\]
The algebraic index theorem alluded to above (see Section \ref{section-index}), is recovered by computing the partition function, i.e., the trace of the constant function $1$. The main idea of our computation is to consider the equivariant extension of the trace map with respect to the $S^1$-rotation on the circle and show that varying $\hbar$ leads to explicit equivariant exact terms (see equation \eqref{rescaling-eqn}). Then the algebraic index theorem follows from the semi-classical computation where $\hbar\to 0$ (Lemma \ref{Semi-classical limit}), i.e., by Feynman diagrams up to one-loop. This leads to the algebraic index theorem \cite{Fedbook, Nest-Tsygan}
\[
\Tr (1) = \int_M e^{-\omega_\hbar/\hbar} \hat{A} (M).
\]
In this expression $\omega_\hbar \in H^2 (M)[[\hbar]]$ parametrizes the moduli of deformation quantizations of the Poisson algebra $(C^\infty (M), \{-,-\})$. Our approach via BV quantization of one-dimensional Chern-Simons theory illustrates certain geometric aspects of the algebraic computations in \cites{Nest-Tsygan,FFS}.

Our construction is motivated by the AKSZ formulation \cite{AKSZ} of one-dimensional sigma models in a neighborhood of the constant maps \cite{Kevin-HCS}. See also  \cite{AM} for a perspective on one-dimensional Chern-Simons in the vein of Atiyah and Segal. The space of fields in our one-dimensional Chern-Simons theory is 
\[
\mathcal{E}:=\A^\bullet_{S^1}\otimes_{\R} \g[1],
\]
where $\A^\bullet_{S^1}$ is the de Rham complex on the circle $S^1$, $\g$ is a curved $L_\infty$-algebra equipped with a symmetric pairing of cohomology degree $-2$, and $\g[1]$ refers to a shift of degree by $1$ (see Conventions). Integration over $S^1$ induces a symplectic pairing on $\mathcal{E}$ of cohomological degree $-1$ and the classical action functional is then obtained via the AKSZ construction \cite{AKSZ}. To any symplectic manifold $M$ we can associate such an $L_\infty$-algebra $\g_M$ \cite{Owen-Ryan} and we work with this geometric example. In order to quantize this system, we follow a convenient, rigorous formulation of perturbative quantization developed by Costello \cite{Kevin-book}. This requires solving an effective quantum master equation. We briefly review Costello's theory in Section \ref{Costello-renormalization}.

In our setting, solutions of the quantum master equation can be described in terms of Fedosov's Abelian connections.
More precisely, let $\nabla$ be a sympletic connection on the symplectic manifold $M$, which is in general not flat. Fedosov showed that the induced connection $\nabla$ on the Weyl bundle $\W$ over $M$ can be modified to be flat by adding ``quantum corrections''. Such a flat connection is called an ``Abelian connection." In this picture, the quantum observables are simply realized as the flat sections of the Weyl bundle (see Section \ref{section-weyl}). Our main result (Theorem \ref{thm: QME=Fedosov's equation}) establishes the following 
\[
   \boxed{\text{Fedosov's Abelian connections}  \Longrightarrow  \text{solutions of the quantum master equation}}.
\]

One immediate corollary of the correspondence above is our ability to quantize our theory perturbatively  and compute the all-loop  (all orders of $\hbar$) quantum corrections exactly. In the literature there is a simplified situation, described in \cite{Owen-Ryan}, called the cotangent theory; the general notion is defined in \cite{Kevin-HCS}. In the cotangent setting, quantum corrections terminate at one-loop, and it can be viewed as the semi-classical model of our theory. 

In \cite{Kevin-Owen}, Costello and Gwilliam developed an observable theory within the framework of perturbative quantum field theory. In this theory, solutions of the quantum master equation naturally provide quantum deformations of the algebra of classical observables. From the correspondence above, one would expect that Fedosov's quantization via an Abelian connection would agree with the corresponding quantization of observables in our one-dimensional Chern-Simons theory.  This is indeed the case as we explicitly verify  in  Section \ref{section-observable}. 

By realizing the deformation quantization via such an explicit quantum field theory we obtain our geometric constructions outlined above: the observables have the structure of factorization algebra, on which renormalization group flow (in the sense of \cite{Kevin-book}) induces a homotopy equivalence. This allows us to construct correlation functions for local observables via the local-to-global factorization product and produce an effective theory on ``zero modes.'' Additionally, we get for free a trace map via BV integration. This geometric context has its own interest and will be explained in sections \ref{section-Fedosov-BV} and \ref{section-trace}.  

\noindent \textbf{Conventions}
\begin{itemize}
\item We work over characteristic $0$ and our base field $k$ will mainly be $\R$. 
\item Let $V=\bigoplus\limits_{m\in \Z} V_m$ be a $\Z$-graded $k$-vector space. Given $a\in V_m$, we let $\bar a=m$ be its degree.
\begin{itemize}
\item $V[n]$ denotes the degree shifting such that $V[n]_m=V_{n+m}$.
\item $V^*$ denotes its linear dual such that $V^*_m=\Hom_k(V_{-m},k)$.
\item All tensors are in the category of graded vector spaces. 
\item $\Sym^m(V)$ and $\wedge^m(V)$ denote the graded symmetric product and graded skew-symmetric product respectively. The Koszul sign rule is always assumed. We also denote
$$
  \Sym(V):=\bigoplus_{m\geq 0}\Sym^m(V), \quad \widehat{\Sym}(V):=\prod_{m\geq 0} \Sym^m(V). 
$$

\item Given $P\in \Sym^2(V)$, it defines a ``second order operator" $\pa_P$ on $\Sym(V^*)$ or $\widehat{\Sym}(V^*)$ by 
$$
  \pa_P:  \Sym^m(V^*)\to \Sym^{m-2}(V^*), \quad I \to \pa_P I,
$$
where for any $a_1, \cdots, a_{m-2}\in V$, 
$$
  \pa_P I(a_1, \cdots a_{m-2}):= I(P, a_1, \cdots a_{m-2}). 
$$
\item $V[\hbar]$, $V[[\hbar]]$ and $V((\hbar))$ denote polynomial series, formal power series and Laurent series respectively in a variable $\hbar$ valued in $V$.

\item Given $A, B\in \Hom(V, V)$, the commutator $[A, B]\in \Hom(V, V)$ always means the graded commutator: $[A, B]=AB-(-1)^{\bar A \bar B}BA$. 

\end{itemize}

\item We will also work with infinite dimensional functional spaces that carry natural topologies. The above notions for $V$ will be generalized as follows. We refer the reader to \cite{treves} or Appendix 2 of \cite{Kevin-book} for further details.

\begin{itemize}
\item All topological vector spaces we consider will be {\it nuclear} (or at least {\it locally convex Hausdorff}) and we use $\otimes$ to denote the {\it completed projective tensor product}. For example, 
\begin{enumerate}
\item Given two manifolds $M, N$, we have a canonical isomorphism
\[
C^\infty (M) \otimes C^\infty (N) = C^\infty (M \times N);
\]
\item Similarly, if $\mathcal{D} (M)$ denotes compactly supported distributions, then again we have
\[
\mathcal{D} (M) \otimes \mathcal{D} (N) = \mathcal{D} (M \times N).
\]
\end{enumerate}

\item Given two such vector spaces $E$ and $F$, we let $\Hom (E,F)$ denote the space of continuous linear maps and when critical, we will equip $\Hom (E,F)$ with the topology of uniform convergence on bounded sets.
\item We define symmetric powers and the completed symmetric algebra by
\[
\Sym^n (E) := \left (E^{\otimes n} \right )_{S_n} \quad \text{ and } \quad \widehat{\Sym} (E) := \prod_{n \ge 0} \Sym^n (E),
\]
where the subscript $S_n$ denotes coinvariants with respect to the natural symmetric group action. 
\end{itemize}

\item Let $M$ be a manifold. We will denote $\Omega^\bullet_M=\bigoplus_{k}\Omega^k_M$, where $\Omega^k_M$ is the bundle of differential $k$-forms on $M$. Global smooth differential forms on  $M$ will be denoted by 
$$
  \A_M^\bullet=\Gamma(M, \Omega^\bullet_M), \quad \text{where}\ \A_M^k=\Gamma(M, \Omega^k_M).
$$
Given a vector bundle $E$, differential forms valued in a vector bundle $E$ will be denoted by $\A^\bullet_M(E)$:
$$
\A^\bullet_M(E)=\Gamma(M, \Omega^\bullet_M\otimes E). 
$$

\item Let $(M,\omega)$ be a symplectic (smooth) manifold. In local coordinates
\[
\omega=\frac{1}{2}\omega_{ij}dx^i\wedge dx^j,
\] 
where \emph{Einstein summation convention} is implicit throughout this paper. 
We use $(\omega^{ij})$ to denote the inverse of $(\omega_{ij})$.   

\end{itemize}

\section{Fedosov's deformation quantization and BV quantization}\label{sect:2}
In this section we  recall Fedosov's approach to deformation quantization on symplectic manifolds, and then extend Fedosov's constructon of the Weyl bundle to a bundle of \BV algebras (the BV bundle). We present an explicit formula to transfer Fedosov's flat connection on the Weyl bundle (called an Abelian connection) to a flat connection on our BV bundle in terms of Feynman graph integrals. We show that fiberwise BV integration leads to a natural trace map on the deformation quantized algebra. By exploring the $S^1$-equivariant localization within BV integration, we deduce the algebraic index theorem via the semi-classical method, i.e., that of one-loop Feynman diagram computations. Our construction originates from a rigorous quantum field theory that will be explained in Section \ref{sect:3}. 

\subsection{Linear symplectic structure}
Let us begin by considering a linear symplectic space $(V, \omega)$, where $V\iso \R^{2n}$ with a symplectic pairing $\omega \in \wedge^2 V^*$. To the pair $(V, \omega)$ we associate several algebraic structures; we will globalize these structures over a symplectic manifold in later sections. 

\subsubsection{Weyl algebra} 

\begin{defn} Given a vector space $V$, we denote its ring of formal functions by
$$
   \widehat{\OO}(V):=\widehat{\Sym}(V^*)=\prod_{k\geq 0}\Sym^k(V^*). 
$$
Here $V^*$ is the linear dual of $V$.
\end{defn}

Let $\hbar$ be a formal parameter.
There are two natural ring structures on $\widehat{O}(V)[[\hbar]]$.  The first one is the obvious $\hbar$-bilinear commutative product, which we denote by $a\cdot b$ or simply $ab$ for $a,b \in \widehat{O}(V)[[\hbar]]$. The sympletic pairing $\omega$ induces a $\hbar$-bilinear Poisson bracket on $\widehat{O}(V)[[\hbar]]$ which we denote by $\{-,-\}$. 

The second structure on $\widehat{O}(V)[[\hbar]]$ is the Moyal product \cite{Moyal}, which we denote by $a\star b$. Let us choose a basis $\{e_i\}$ for $V$ and let $\{x^i\}\in V^*$ be the dual basis viewed as linear coordinates. The linear symplectic form can be written as
\[
  \omega=\frac{1}{2}\omega_{ij}dx^i\wedge dx^j. 
\]
Then 
\[
   (a\star b)(x, \hbar)=\left. \exp\bracket{\frac{\hbar}{2}\omega^{ij}\frac{\pa}{\pa y^i}\frac{\pa}{\pa z^j}}a(y,\hbar) b(z, \hbar)\right |_{y=z=x},
\]
where $(\omega^{ij})$ is the inverse matrix of $(\omega_{ij})$. This definition does not depend on the choice of basis. It is straightforward to check that the Moyal product is associative and its first-order non-commutativity is measured by the Poisson bracket
\[
  \lim_{\hbar \to 0}\frac{1}{\hbar}(a\star b-b\star a)=\{a,b\}, \quad a,b\in \widehat{O}(V)[[\hbar]].
\]

\begin{defn} We define the (formal) Weyl algebra $\W(V)$ to be the associative algebra $(\widehat{O}(V)[[\hbar]], \star)$.

\end{defn}

\subsubsection{\BV (BV) algebra}  \label{sec-BV}

\begin{defn}\label{defn-BV} A \BV (BV) algebra is a pair $(\A, \Delta)$ where
\begin{itemize}
\item $\A$ is a $\Z$-graded commutative associative unital algebra, and
\item $\Delta: \A \to \A$ is a second-order operator of degree $1$ such that $\Delta^2=0$. 
\end{itemize}
\end{defn}

Here $\Delta$ is called the BV operator. $\Delta$ being ``second-order" means the following: let us define the \emph{BV bracket} $\fbracket{-,-}_\Delta$ as the failure of $\Delta$ to be a derivation
$$
    \fbracket{a,b}_\Delta:=\Delta(ab)-(\Delta a)b- (-1)^{\bar a}a \Delta b. 
$$
Then $\fbracket{-,-}: \A\otimes \A\to \A$ defines a Poisson bracket of degree $1$ satisfying
\begin{itemize}
\item $\fbracket{a,b}=(-1)^{\bar a \bar b}\fbracket{b,a}$,
\item $\fbracket{a, bc}=\fbracket{a,b}c+(-1)^{(\bar a+1)\bar b}b\fbracket{a,c}$, and
\item $\Delta\fbracket{a,b}=-\fbracket{\Delta a, b}-(-1)^{\bar a}\fbracket{a, \Delta b}$. 
\end{itemize}

We will associate to $(V, \omega)$ a canonical BV algebra related to Koszul-Brylinski homology \cite{Brylinski}.

\begin{defn}  

Given a vector space $V$,  we define its ring of formal differential forms by
\[
\widehat\Omega^{-\bullet}_V:=\widehat{\OO}(V)\otimes \wedge^{-\bullet}(V^*),
\]
where the ring structure is induced by the wedge product on forms and the commutative product $\cdot$ on $\widehat{\OO}(V)$. With respect to the natural decomposition 
\[
   \widehat\Omega^{-\bullet}_{V}=\oplus_{p}\widehat\Omega^{-p}_{V}, \quad \widehat\Omega^{-p}_{V}=\widehat{\OO}(V)\otimes \wedge^p (V^*),
\]
our degree assignment is such that $p$-forms  $\widehat\Omega^{-p}_{V}$ sit in degree $-p$. 
\end{defn}

Let $d_{V}$ be the de Rham differential (which is of degree $-1$)
\[
   d_V: \widehat\Omega_{V}^{-\bullet}\to \widehat\Omega_V^{-(\bullet+1)}. 
\]
Let $\Pi=\omega^{-1} \in \wedge^2 V$ be the Poisson tensor. It defines a map of degree $2$ via contraction 
\[
  \iota_{\Pi}: \widehat\Omega_{V}^{-\bullet}\to \widehat\Omega_{V}^{-(\bullet-2)}. 
\]

\begin{defn}
We define the BV operator by the Lie derivative with respect to $\Pi$
\[
  \Delta:=\L_{\Pi}=[d_V, \iota_\Pi]:  \widehat\Omega_{V}^{-\bullet}\to  \widehat\Omega_{V}^{-(\bullet-1)}.
\]
\end{defn}

It is easy to check that $\Delta^2=0$, which is equivalent to the Jacobi identity for the Poisson tensor $\Pi$.  The following proposition is well-known

\begin{prop} The triple $\bracket{\widehat{\Omega}_V^{-\bullet}, \cdot, \Delta}$ defines a BV algebra. 

\end{prop}

One advantage of a BV structure is that it allows us to do integration. We can view $\widehat\Omega^{-\bullet}_V$ as (formal) functions on the graded space $V\otimes H^*(S^1)=V[\epsilon]$, where $\epsilon$ has cohomological degree $1$ and $\epsilon^2=0$. $V[\epsilon]$ is equipped with an odd symplectic pairing
\[
  (v_1+\epsilon v_2, u_1+\epsilon u_2):= \omega(v_1, u_2)+\omega(v_2, u_1), \quad u_i, v_i \in V. 
\]
Hence, $V[\epsilon]$ can be viewed as a BV-manifold (see for example \cite{Schwarz,Getzler}). The analogue of de Rham theorem for BV manifolds \cites{BV,Schwarz} says that there is a natural pairing between BV-cohomology of functions and homology classes of Lagrangian super-submanifolds. There is a natural odd linear lagrangian subspace 
\begin{align}\label{BV-lagrangian}
\epsilon V\subset V[\epsilon]. 
\end{align}

\begin{defn}
We define the BV integration 
\[
    a\to \int_{Ber} a, \quad a\in \widehat\Omega^{-\bullet}_V
\]
by the Berezin integral  over the odd linear space $\epsilon V$. In terms of linear coordinates $x^i, \theta^i$, where $\theta^i$ is the odd variable representing $dx^i$, 
\begin{align}\label{fiber-BV-integration}
   \int_{Ber} a(x, \theta, \hbar):=\left. \frac{1}{n!}\bracket{\frac{1}{2}\omega^{ij}{\pa_{\theta^i}}{\pa_{\theta^j}}}^n a(x, \theta, \hbar)\right|_{x=\theta=0}. 
\end{align}
\end{defn}
In other words, the BV integration $\int_{Ber}$ picks out the top product of odd variables and sets the even variables to zero. 

Our next goal is to globalize the Weyl and BV algebra constructions over a  sympletic manifold. 

\subsection{Weyl bundle and Fedosov's flat connection}\label{section-weyl}
Let $(M,\omega)$ be a $2n$-dimensional symplectic manifold. The $2$-form $\omega$ induces a
linear symplectic structure on each fiber of $TM$. Let $\{x^i\}$ denote local coordinates in a neighborhood of some reference point. We will let
$y^i$ denote the dual of the tangent vector $\partial_{x^i}$. 

\begin{notn}
We will keep this notation in the rest of this section: $x^i$'s always denote coordinates on the base $M$, while $y^i$'s always denote coordinates along the fiber. 
\end{notn}

\subsubsection{Weyl bundle}\label{subsection-Weyl}

\begin{defn}
The Weyl bundle of $(M, \omega)$ is defined to be $\W(M):=\prod_{k\geq 0}{\Sym}^k(T^* M)[[\hbar]].
$ A (smooth) section $a$ of $\W$ is given locally by 
\[
a(x,y)=\sum_{k, l\geq 0}\hbar^k a_{k,i_1\cdots i_l}(x)y^{i_1}\cdots y^{i_l},
\]
where the $a_{k,i_1\cdots i_l}(x)'s$ are smooth functions. We will simply write $\W$ for the Weyl bundle when the manifold $M$ is clear from the context. 
\end{defn}

The fiber of $\W$ over each point $p\in M$ is given by 
$\W(T_pM)
$. In particular, the fiberwise commutative product $\cdot$ and Moyal product $\star$ endow $\W$ with the structure of a bundle of commutative and associative algebras respectively. 

\begin{defn}\label{defn-AW}
We let $\A_M^\bullet(\W)$ denote differential forms on $M$ with values in $\W$ (See Conventions).
\end{defn}
To be precise, a section of $\A_M^\bullet(\W)$ is given locally by a sum of the following terms
\begin{equation}\label{eqn:monomial-p-q}
a_{k,i_1\cdots i_p,j_1\cdots j_q}(x)\hbar^k dx^{j_1}\cdots dx^{j_q}y^{i_1}\cdots y^{i_p},
\end{equation}
where $a_{k,i_1\cdots i_p,j_1\cdots j_q}$'s are smooth tensors symmetric in $i_1,\cdots, i_p$ and antisymmetric in $j_1,\cdots, j_q$. The Moyal product is $\A_M^\bullet$-linearly extended to $\A^\bullet_M(\W)$ and we can define the associated commutator by
\[
\big[a,b\big]_{\star}:=a\star b-(-1)^{q_1q_2}b\star a,\hspace{3mm} \text{for\ }a\in \A_M^{q_1}(\W), b\in\A_M^{q_2}(\W).
\]

\begin{defn}Following \cite{Fed}, we define two operators on $\A^\bullet_M(\W)$:
\[
\delta a:=dx^k\wedge\frac{\partial a}{\partial y^k}\hspace{3mm}, \hspace{3mm}\delta^* a:=y^k\cdot \iota_{\partial_{x^k}} a. 
\]
Here $\iota_{\partial_{x^k}}$ denotes the contraction between forms and the vector field $\frac{\partial}{\partial x^k}$. We define $\delta^{-1}$ by requiring that it acts on the monomial (\ref{eqn:monomial-p-q}) as $\frac{\delta^*}{p+q}$ if $p+q>0$ and zero if $p+q=0$. 
\end{defn}

\subsubsection{Fedosov's flat connection}

It turns out that the bundle $\W$ carries a flat connection. This surprising fact is proved by Fedosov \cite{Fed} and we recall his construction here.  We first need the notion of a symplectic connection.

\begin{defn}
A connection $\nabla$ on the tangent bundle of a symplectic manifold $(M,\omega)$ is called symplectic if it is  torsion-free and compatible with the symplectic form $\omega$ (i.e., $\nabla(\omega)=0$). 
\end{defn}

\begin{rmk}
There always exist symplectic connections on $(M,\omega)$ and they form an affine space modeled on $\Gamma(M,\Sym^3 T^* M)$. Throughout this paper,
we will fix a choice of $\nabla$.
\end{rmk}

The connection $\nabla$ naturally induces a connection on $\W$, which we still denote by $\nabla$. Let 
\begin{equation}\label{curvature}
R_\nabla:=\frac{1}{4}R_{ijkl}y^iy^jdx^k\wedge dx^l, \quad R_{ijkl}=\omega_{im}R^m_{jkl},
\end{equation}
where $\frac{1}{2}R^i_{jkl}dx^k\wedge dx^l$ is the curvature $2$-form of $\nabla$. Note that in contrast to the Riemannian curvature tensor, the symplectic one $R_{ijkl}$ (with index raised and lowered by $\omega$) is skew-symmetric in two indices ($kl$) and symmetric in the other two indices ($ij$). As in \cite{Fed}, the curvature on $\W$ can be represented by 
\begin{equation}\label{eqn:curvature-Weyl-bundle}
\nabla^2(a)=\frac{1}{\hbar}\big[R_\nabla,a\big]_\star \;, \quad a\in \Gamma(M, \W). 
\end{equation}
Note that in the above formula, the factor $\frac{1}{\hbar}$ will cancel with another factor $\hbar$ coming from the commutator associated to the Moyal product. Let us consider connections on the Weyl bundle of the following form:
\[
\nabla +\frac{1}{\hbar}\big[\gamma, -\big]_\star \; ,
\]
where $\gamma\in \A^1_M(\W)$. Connections of the above form are compatible with the Moyal product: 
$$
  D(a \star b)=D(a)\star b+a \star D(b), 
$$
where $D=\nabla +\frac{1}{\hbar}\big[\gamma, -\big]_\star$, and $a,b$ are sections of the Weyl bundle. 

\begin{defn}
The connection $\nabla +\frac{1}{\hbar}\big[\gamma, -\big]_\star$ on the Weyl bundle $\W$ is called an $Abelian$ connection if  its curvature is zero, i.e., $(\nabla +\frac{1}{\hbar}\big[\gamma, -\big]_\star)^2=0$. 
\end{defn}

We define a $\Z$-grading of weights on the Weyl bundle $\W$ by assigning weight $2$ to $\hbar$ and weight $1$ to $y^i$'s. $\W_{(m)}$ will denote the sub bundle of homogenous weight $m$. In particular, we have the weight decomposition 
\[
   \W=\prod_{m\geq 0} \W_{(m)}.
\]

In \cite{Fedbook}, Fedosov proved the following theorem:
\begin{thm}\label{thm: existence-uniqueness-Fedesov-equation}Given a sequence $\{\omega_k\}_{k\geq 1}$ of closed 2-forms on $M$, there exists a unique Abelian connection $\nabla+\frac{1}{\hbar}[\gamma,-]_\star$ such that 
\begin{enumerate}
 \item $\gamma$ is of the form
\[
\gamma=\sum_{i,j}\omega_{ij}y^idx^j+r,
\] 
where $r\in\A_M^1(\W)$ consists of elements of weight at least $3$ which satisfies the gauge fixing condition:
\[
\delta^{-1}(r)=0.
\]
 \item $\gamma$ satisfies the following equation
\begin{equation}\label{Fedosov-eqn}
   \nabla \gamma+\frac{1}{2\hbar }\big[\gamma, \gamma\big]_\star+R_\nabla=\omega_\hbar,
\end{equation}
where $\omega_\hbar=-\omega+\sum_{k\geq 1}\hbar^k \omega_k.$
\end{enumerate}

\end{thm}

Equation \eqref{Fedosov-eqn}  implies flatness. In fact, 
\[
 (\nabla+\frac{1}{\hbar}[\gamma,-]_\star)^2=\frac{1}{\hbar}\big[ \nabla \gamma+\frac{1}{2\hbar }\big[\gamma, \gamma\big]_\star+R_\nabla, -\big]_\star=\frac{1}{\hbar}\big[\omega_\hbar, -\big]_\star=0
\]
since $\omega_\hbar$ is a central element in $\A^\bullet_M(\W)$.

In \cite{Fed-index}, it is shown that $\gamma$ can be solved iteratively in terms of weights. In particular, the  weight 3 component is given by the following lemma. This term is key in our localization computation of the algebraic index later. 

\begin{lem}[\cite{Fed-index}]\label{lemma-curvature}
Modulo $\hbar$, the weight 3 component in $\gamma$ is given by 
$
\frac{1}{8}R_{(ijk)l}dx^l\otimes y^iy^jy^k.
$
Here $R_{(ijk)l}$ is the symmetrization of $R_{ijkl}$ with respect to the first three indices:
$$
R_{(ijk)l}:=\frac{1}{3}\left(R_{ijkl}+R_{jkil}+R_{kijl}\right).
$$
\end{lem}

\subsubsection{Deformation quantization}\label{sec:DQ}
\begin{defn}
A deformation quantization of the algebra of smooth functions on a symplectic manifold $(M,\omega)$ is an associative $\mathbb{R}[[\hbar]]$-bilinear multiplication $\star$ on $C^\infty(M)[[\hbar]]$ of the following form:
\[
f\star g=f\cdot g+\sum_{k\geq 1}\hbar^{k}C_k(f,g),\hspace{3mm}\text{for\ } f,g\in C^\infty(M),
\]
where  the $C_k$'s are bi-differential operators on $C^\infty(M)$ such that 
\[
\{f,g\}=\lim_{\hbar\to 0} \frac{1}{\hbar}(f \star g-g\star f). 
\] 
\end{defn}

In other words, $\star$ deforms $C^\infty(M)$ to an associative algebra and the Poisson bracket measures the leading order degree of noncommutativity.

A choice of an Abelian connection leads to a simple construction of deformation quantization on symplectic manifolds as follows. We  consider the symbol map obtained by sending the $y^i$'s to $0$:
\[
\sigma: \Gamma(M,\W)\rightarrow C^\infty(M)[[\hbar]].
\]
Let $\Gamma^{flat}(M, \W)\subset \Gamma(M,\W)$ be the subspace of flat sections with respect to the chosen Abelian connection. 

\begin{thm}[\cite{Fed}] The symbol map $\sigma$ induces an isomorphism of vector spaces
$$
  \sigma: \Gamma^{flat}(M, \W)\stackrel{\iso}{\to} C^\infty(M)[[\hbar]]. 
$$
\end{thm}

Let
$$
\sigma^{-1}: C^\infty(M)[[\hbar]]\to \Gamma^{flat}(M, \W)
$$
denote the inverse of the above identification. Since Abelian connections are compatible with the Moyal product, $\Gamma^{flat}(M, \W)$ inherits an associative algebra structure.  This induces a star product on $C^\infty(M)[[\hbar]]$ as follows: let $f,g\in C^\infty(M)$, we define  their star product (with the same notation $\star$) as
\[
f\star g:=\sigma(\sigma^{-1}(f)\star \sigma^{-1}(g)).
\]
It is easy to check that this star product defines a deformation quantization on $(M, \omega)$. 

\subsection{BV bundle and quantum master equation} 
\subsubsection{BV bundle} \label{sec-BV-bundle}
Similar to the Weyl bundle, we globalize our BV algebra construction as follows.

\begin{defn} The BV bundle of $(M, \omega)$ is defined to be 
\[
 \widehat \Omega^{-\bullet}_{TM} :=  \widehat{\Sym} (T^* M) \otimes \wedge^{-\bullet}(T^* M), \quad \wedge^{-\bullet}(T^* M):=\bigoplus_{k}\wedge^k (T^* M)[k].
\]
\end{defn}
$ \widehat \Omega^{-\bullet}_{TM} $ can be viewed as the relative de Rham complex for $TM \to M$ formally completed at the zero section . Following our shift conventions, $\wedge^k(T^* M)$ sits in cohomological degree $-k$. 

Let $\Pi=\omega^{-1}\in \Gamma(M, \wedge^2 TM)$ be the Poisson tensor. For each point $p \in M$, $\Pi$ induces a constant Poisson tensor for  the tangent space $T_pM$. The fiberwise operators $d_V$, $\iota_\Pi, \Delta$ defined in Section \ref{sec-BV} carry over to the bundle case, which we denote by 
\[
  d_{TM}: \widehat\Omega^{-\bullet}_{TM}\to \widehat\Omega^{-(\bullet+1)}_{TM}, \quad   \iota_{\Pi}: \widehat\Omega_{TM}^{-\bullet}\to \widehat\Omega_{TM}^{-(\bullet-2)}, \quad   \Delta=\L_{\Pi}:  \widehat\Omega_{TM}^{-\bullet}\to  \widehat\Omega_{TM}^{-(\bullet-1)}.
\]
Similarly, we have a fiberwise BV bracket (which is a morphism of bundles):
\[
  \{-,-\}_\Delta:  \widehat \Omega^{-p}_{TM} \otimes \widehat\Omega^{-q}_{TM}\to \widehat\Omega^{-p-q+1}_{TM}.
\]
The fiberwise BV integration will be denoted by 
\[
  \int_{Ber}: \Gamma(M,\widehat\Omega^{-\bullet}_{TM})\to \cinfty(M). 
\]
which is nonzero only for sections of the component $\widehat\Omega^{-2n}_{TM}$. 

Let $\nabla$ be the symplectic connection as above, it induces a connection 
$
\nabla$ on   $\widehat\Omega^{-\bullet}_{TM}$ since $\widehat\Omega^{-\bullet}_{TM}$ is a tensor bundle constructed from $TM$. Let 
\[
  \A_M^\bullet(\widehat \Omega^{-\bullet}_{TM})=\oplus_p   \A^p_M(\widehat\Omega^{-\bullet}_{TM}), 
\]
where $\A^p_M(\widehat \Omega^{-\bullet}_{TM})$ denotes differential p-forms valued in $\widehat\Omega^{-\bullet}_{TM}$. Our degree assignment is such that $\A^p_M(\Omega^{-q}_\W)$ has cohomological degree $p-q$. As usual, $\nabla$ is naturally extended to an operator on $\widehat\Omega^{-\bullet}_{TM}$-valued forms
$$
 \nabla: \A^p_M(\widehat\Omega^{-\bullet}_{TM})\to \A^{p+1}_M(\widehat\Omega^{-\bullet}_{TM})
$$
which we will denote by the same symbol. 

\begin{lem} \label{lem-compatible}The operators $d_{TM}, \iota_\Pi, \Delta, \{-,-\}_\Delta$ are all parallel with respect to $\nabla$. 
\end{lem}
\begin{proof} 
For the operator $d_{TM}$, it is clearly parallel with respect to $\nabla$. For the  operator $\iota_\Pi$, the statement follows from the fact that $\omega$, hence $\Pi$, is $\nabla$-flat. As to the remaining operators, the statement follows from the first two. 
\end{proof}

\subsubsection{The quantum master equation}

Following Fedosov, we would like to construct a flat structure on the BV bundle $\widehat\Omega^{-\bullet}_{TM}$. As a preliminary, we consider the following differential. 

\begin{lem} The operator $\nabla+\hbar \Delta+ \hbar^{-1} d_{TM}(R_\nabla)$ defines a differential on $\A^\bullet_M(\widehat\Omega^{-\bullet}_{TM})[[\hbar]]$, i.e.,
\[
  \bracket{\nabla+\hbar \Delta+ \hbar^{-1} d_{TM}(R_\nabla)}^2=0. 
\]
\end{lem}
\begin{proof} Observe that the curvature operator
$
  \nabla^2=-\{ d_{TM}(R_\nabla), -\}_{\Delta}
$
on $\A_M^\bullet(\widehat\Omega^{-\bullet}_{TM})$. It follows that
\begin{align*}
2\bracket{\nabla+\hbar \Delta+ \hbar^{-1} d_{TM}(R_\nabla)}^2&=\bbracket{{\nabla+\hbar \Delta+ \hbar^{-1} d_{TM}(R_\nabla)}, {\nabla+\hbar \Delta+\hbar^{-1} d_{TM}(R_\nabla)}}\\
&=\nabla^2+\bbracket{\Delta, d_{TM}(R_\nabla)}\\
&=\nabla^2+\{ d_{TM}(R_\nabla), -\}_{\Delta}=0. 
\end{align*}
Here in the second equality we have used $[\nabla, \Delta]=0$ (Lemma \ref{lem-compatible}) and $[\nabla, d_{TM}(R_\nabla)]=0$ (the Bianchi identity). In the third equality we have used $\Delta(d_{TM}(R_\nabla))\propto R_{ijkl}\omega^{ij}dx^k\wedge dx^l=0$ (since $R_{ijkl}$ is symmetric in $ij$) and the BV algebra relation. 
\end{proof}

\begin{defn}\label{defn-QME}
A degree $0$ element $\Gamma$ of $\A^\bullet_M(\widehat\Omega^{-\bullet}_{TM})[[\hbar]]$ is said to satisfy the \emph{quantum master equation} (QME) if
\begin{align}\label{QME-Gamma}
   \bracket{\nabla+\hbar \Delta +\hbar^{-1} d_{TM}(R_\nabla)}e^{\Gamma/\hbar}=0.
\end{align}
Such a $\Gamma$ is called a nilpotent solution of quantum master equation if $\Gamma^N=0$ for $N\gg 0$.
\end{defn}

The quantum master equation can be viewed as an analogue of Fedosov's equation \eqref{Fedosov-eqn}.

\begin{lem}\label{lem-naive-differential} If $\Gamma$ satisfies the quantum master equation, then the operator $ \nabla+\hbar \Delta+ \{\Gamma,-\}_{\Delta}$ defines a differential on $\A^\bullet_M(\widehat\Omega_{TM}^{-\bullet})$. 
\end{lem}
\begin{proof}
Equation \eqref{QME-Gamma} implies that as a composition of operators the following formula holds
\[
  e^{-\Gamma/\hbar} \bracket{\nabla+\hbar \Delta+ \hbar^{-1} d_{TM}(R_\nabla)}e^{\Gamma/\hbar}= \nabla+\hbar \Delta+ \{\Gamma,-\}_{\Delta}.
  \]
Here the conjugation should be understood as 
$
     e^{-\Gamma/\hbar}(-)  e^{\Gamma/\hbar}:=e^{-ad_{\Gamma/\hbar}}(-)
$ where $ad_{\Gamma/\hbar}(A)=(\Gamma/\hbar \cdot) A-A  (\Gamma/\hbar\cdot)$ for a linear operator $A$ on $\A^\bullet_M(\widehat\Omega_{TM}^{-\bullet})$. The lemma follows easily from this formal expression. 
\end{proof}

\subsubsection{The integration map}\label{sect:integration}
The fiberwise integration $\int_{Ber}$ extends $\A_M^\bullet$-linearly to a map
\[
  \int_{Ber}: \A_M^\bullet(\widehat\Omega_{TM}^{-\bullet})\to \A^\bullet_M
\]
which is of cohomological degree $2n$ as a morphism of graded vector spaces. When we work with various spaces involving the formal variable $\hbar$, it is understood that $\int_{Ber}$ is always $\hbar$-linear. 

\begin{lem}\label{lem-cochain-BV} The map $\int_{Ber}$ is a cochain map from the complex $( \A^\bullet_M(\widehat\Omega_{TM}^{-\bullet})[[\hbar]], \nabla+\hbar \Delta+ \hbar^{-1} d_{TM}(R_\nabla))$ to the complex $(\A^\bullet_M[[\hbar]], d_{M})$, where $d_{M}$ is the de Rham differential on $M$. 
\end{lem}
\begin{proof} Formula \eqref{fiber-BV-integration} implies (by type reason) that 
\[
  \int_{Ber} (\hbar \Delta+ \hbar^{-1} d_{TM}(R_\nabla))a=0, \quad \text{for}\ a \in \A^\bullet_M (\widehat\Omega_{TM}^{-\bullet})
\]
and
\[
  \int_{Ber} \nabla a_+=0, \quad \mbox{for} \; \; a_+\in  \Gamma \left (M,\prod_{p+2n-q>0}\Sym^p (T^* M)\otimes \wedge^{-q}(T^* M) \right ). 
\]
Moreover, when $p+2n-q=0$, i.e., $q=2n, p=0$, we can write the section $a$ as 
\[
a=\alpha\otimes \omega^n, \quad \alpha\in \A^\bullet_M.
\]
Then  $\nabla(a)=d\alpha\otimes\omega^n$ since $\omega$ is $\nabla$-parallel. The lemma follows immediately. 
\end{proof}

\begin{defn}\label{defn:BVint} Given a nilpotent solution $\Gamma$ of the quantum master equation, we define the twisted integration map
\begin{align}\label{defn-BV-integration}
  \int_{\Gamma}: \A^\bullet_M(\widehat\Omega_{TM}^{-\bullet})\to \A^\bullet_M((\hbar)), \quad    \int_{\Gamma} a:= \int_{Ber} e^{\Gamma/\hbar} a. 
\end{align}
\end{defn} 
It follows from Lemma \ref{lem-naive-differential} (its proof) and Lemma \ref{lem-cochain-BV} that $\int_\Gamma$ is a cochain map between complexes 
$$
\int_\Gamma: 
( \A^\bullet_M(\widehat\Omega_{TM}^{-\bullet})[[\hbar]], \nabla+\hbar \Delta+ \{\Gamma,-\}_{\Delta})\to (\A^\bullet_M((\hbar)), d_{M}).
$$

In particular, we have a well-defined evaluation map
\[
  \int_M \int_\Gamma: H^0(\A^\bullet_M(\widehat\Omega_{TM}^{-\bullet})[[\hbar]], \nabla+\hbar \Delta+ \{\Gamma,-\}_{\Delta})\to \R((\hbar)),
\]
which is the composition of BV-integration $\int_\Gamma$ with a  further integration on $M$. Later, in our discussion of quantum field theory in Section \ref{sect:3}, we will see that this evaluation map corresponds to the expectation value on global quantum observables. 

\subsection{Fedosov quantization vs. BV quantization}\label{section-Fedosov-BV}

Let $\gamma$ be a solution of Fedosov's equation \eqref{Fedosov-eqn}, so $\gamma$ describes a flat connection on the Weyl bundle $\W$. In this section, we construct an explicit Feynman integral formula to transfer $\gamma$ to a solution $\gamma_\infty$ of quantum master equation on the BV bundle. The main result is the following.

\begin{thm}\label{main-thm} Let $(M, \omega)$ be a symplectic manifold. Let $\gamma$ be a solution of Fedosov's equation as in Theorem \ref{thm: existence-uniqueness-Fedesov-equation}.  Let (Definition \ref{defn-gamma-infty})
$$
  \gamma_\infty:=\hbar \log \Mult \int_{S^1[*]} e^{\hbar \pa_P+D}e^{\otimes \gamma/\hbar}. 
$$
Then $\gamma_\infty$ is a nilpotent solution of quantum master equation on the BV bundle (Definition \ref{defn-QME}). 
\end{thm}

The meaning of the symbols in the above formula is a bit involved but rather explicit. The rest of this subsection is devoted to the explanation of this formula and the proof of Theorem \ref{main-thm}. To get a feel for the flavor of this formula, $\gamma_\infty$ is obtained as a sum of all connected Feynman diagram integrals in a topological quantum mechanical model. Such a physical interpretation will be given in Section \ref{sect:3}. 
\bigskip

\begin{rmk}\label{rmk-Hochschild} The existence of a flat structure $\gamma_\infty$ from $\gamma$ as in Theorem \ref{main-thm} is a rather formal consequence from noncommutative geometry. In fact, 
$\Omega^{-\bullet}_V$ represents the Hochschild homology of the commutative algebra $(\widehat{\OO}(V), \cdot)$ (Hochschild-Kostant-Rosenberg Theorem). Similarly, the complex $(\Omega^{-\bullet}_V[[\hbar]], \hbar \Delta)$ models the Hochschild homology of the Weyl algebra $\W(V)$, which holds true for any deformation quantization associated to a Poisson tensor.  This can be seen from Tsygan's formality conjecture for Hochschild chains \cite{Tsygan} (proved by Shoikhet \cite{Sh} and Dolgushev \cite{dolgushev-formality}), which asserts a morphism of $L_\infty$-modules of polyvector fields from  Hochschild chains to  differential forms. Then the linearization of this $L_\infty$-morphism at the Maurer-Cartan element of the Poisson tensor $\Pi$ in the Lie algebra of polyvector fields yields the quasi-isomorphism between the Hochschild chain complex of $\W(V)$ and the complex $(\Omega^{-\bullet}_V[[\hbar]], \hbar \Delta)$. A flat family of Weyl algebras (via $\gamma$) leads to a flat family of Hochschild chains, hence a flat family of BV algebras (via $\gamma_\infty$). The authors  thank the referee for pointing this out. Theorem \ref{main-thm} then gives an explicit formula for the transferred data. 

\end{rmk}

\subsubsection{Quantum master equation via Fedosov quantization}

Let us first introduce some notations.  

\begin{defn}
We denote the $k$-th tensor product bundle of $\widehat\Omega_{TM}^{-\bullet}$ by
\[
  (\widehat\Omega_{TM}^{-\bullet})^{\otimes k}:= \widehat\Omega_{TM}^{-\bullet}\otimes \cdots \otimes \widehat\Omega_{TM}^{-\bullet}.
\]
 Further, we define the "multiplication map`` 
\[
    \Mult:  (\widehat\Omega_{TM}^{-\bullet})^{\otimes k}\to \widehat\Omega_{TM}^{-\bullet}, \quad a_1\otimes \cdots \otimes a_k\to a_1\cdots a_k
\]
by the natural graded commutative product. This definition is $\A_M^\bullet$-linearly extended to $\A_M^\bullet(\widehat\Omega_{TM}^{-\bullet})$
\[
\Mult: \A_M^\bullet(\widehat\Omega_{TM}^{-\bullet})^{\otimes k}:=\A_M^\bullet \left( (\widehat\Omega_{TM}^{-\bullet})^{\otimes k}\right) \to \A_M^\bullet(\widehat\Omega_{TM}^{-\bullet}).
\]
\end{defn}

Recall the Poisson tensor $\Pi=\omega^{-1}\in \Gamma(M, \wedge^2 TM)$ and the associated contraction map
\[
\iota_\Pi: \widehat\Omega_{TM}^{-\bullet}\to \widehat\Omega_{TM}^{-(\bullet-2)}. 
\]
The BV operator is $\Delta=\L_{\Pi}=[d_{TM}, \iota_\Pi]$. In terms of local coordinates $\{x^i, y^i\}$ (as in section \ref{section-weyl}), 
\[
\iota_{\Pi}=\frac{1}{2}\omega^{ij}(x)\iota_{\pa_{y^i}}\iota_{\pa_{y^j}} \; \text{ and } \; \Delta=\omega^{ij}(x)\L_{\pa_{y^i}}\iota_{\pa_{y^j}}. 
\]

\begin{defn}We extend the definitions of $d_{TM}, \iota_{\Pi},$ and $\Delta$ to $\A^\bullet_M(\widehat\Omega_{TM}^{-\bullet})^{\otimes k}$ by declaring
\begin{align*} 
     d_{TM} (a_1\otimes \cdots \otimes a_k)&:=\sum_{1\leq \alpha\leq k} \pm  a_1\otimes \cdots \otimes d_{TM} a_\alpha\otimes \cdots \otimes a_k , \\[1ex]
  \iota_\Pi (a_1\otimes \cdots \otimes a_k)&:=\frac{1}{2}\sum_{1\leq \alpha, \beta\leq k}\pm \omega^{ij}(x) a_1\otimes \cdots \iota_{\pa_{y^i}} a_\alpha \otimes \cdots \iota_{\pa_{y^j}} a_\beta\otimes \cdots \otimes a_k, \text{ and } \\[1ex]
  \Delta(a_1\otimes \cdots \otimes a_k)&:=\sum_{1\leq \alpha, \beta\leq k}  \pm \omega^{ij}(x) a_1\otimes \cdots \L_{\pa_{y^i}} a_\alpha \otimes \cdots \iota_{\pa_{y^j}} a_\beta\otimes \cdots \otimes a_k. 
\end{align*}
Here  $\pm$ are Koszul signs generated by passing odd operators through graded objects. 
\end{defn}

It is straightforward to check that these definitions are compatible with the multiplication map
\[
 \quad \Delta \circ \Mult=\Mult \circ \Delta, \quad \iota_{\Pi}\circ \Mult=\Mult \circ \iota_{\Pi}. 
\]

In the formulation and proof of Theorem \ref{main-thm} we will use configuration spaces and integrals over them.  We have recalled some elementary properties in   Appendix \ref{sec: configuration}.

Let $S^1[m]$ denote the compactifed configuration space of $m$ ordered  points on the circle $S^1$, which is constructed via successive real-oriented blow ups of $(S^1)^m$ (see Definition \ref{defn-configuration}); let
\[
  \pi: S^1[m]\to (S^1)^m
\]
 denote the blow down map.  $S^1[m]$ is a manifold with corners.
In particular, 
$S^1[2]$ is parametrized by a cylinder (see Example \ref{example: two-point-configuration})
\[
  S^1[2]=\{(e^{2\pi i\theta},u)| 0\leq \theta<1, 0\leq u\leq 1\}. 
\]
With this parametrization, the blow down map is 
\[
   \pi: S^1[2]\to (S^1)^2, \quad (e^{2\pi i \theta},u)\to (e^{2\pi i (\theta+u)}, e^{2\pi i \theta}).
\]

\begin{defn}
We define the function $P$ on $S^1[2]$ by the formula: 
\begin{align}
    P(\theta, u)= u-\frac{1}{2}. 
\end{align}
$P$ will be called the \emph{propagator}. 
\end{defn}

\begin{rmk}
$P$ is the first derivative of Green's function on $S^1$ with respect to the standard flat metric. In particular, it represents the propagator of topological quantum mechanics on $S^1$, hence the name. See Remark \ref{remark: propagator} in Appendix \ref{sec: configuration}. 
\end{rmk}

\begin{defn} We define the $\A^\bullet_{S^1[k]}$-linear operators
\[
  \pa_P, D:  \A^\bullet_{S^1[k]} \otimes_{\R} \A^\bullet_M(\widehat\Omega_{TM}^{-\bullet})^{\otimes k}\to \A^\bullet_{S^1[k]}\otimes_{\R}  \A^\bullet_M(\widehat\Omega_{TM}^{-\bullet})^{\otimes k}
\]
 by 
\[
   \pa_P(a_1\otimes \cdots \otimes a_k):=\frac{1}{2} \sum_{1\leq \alpha\neq \beta\leq k} \pi_{\alpha\beta}^*(P)\otimes_{\R} \bracket{\omega^{ij}(x) a_1\otimes \cdots \L_{\pa_{y^i}} a_\alpha \otimes \cdots \L_{\pa_{y^j}} a_\beta\otimes \cdots \otimes a_k}, \text{ and}
\]
\[
   D (a_1\otimes \cdots \otimes a_k):=\sum_{\alpha}\pm d\theta^\alpha\otimes_{\R} \bracket{a_1\otimes \cdots \otimes  d_{TM} a_\alpha\otimes \cdots \otimes a_k}.  
\]
 Here $a_i \in \A^\bullet_M(\widehat\Omega_{TM}^{-\bullet})$.  $\pi_{\alpha\beta}: S^1[m]\to S^1[2]$ is the forgetful map to the two points indexed by $\alpha, \beta$. $\theta^\alpha\in [0,1)$ is the parameter on the $S^1$ indexed by $1\leq \alpha\leq k$ and $d\theta^\alpha$ is viewed as a 1-form on $S^1[k]$ via the pull-back $\pi_\alpha: S^1[m]\to S^1$.  Again, $\pm$ is the appropriate Koszul sign.
\end{defn}

Let $d_{S^1}$ denote the de Rham differential on $\A^\bullet_{S^1[k]}$.

\begin{lem}\label{lem-operators}
As operators on $\A^\bullet_{S^1[k]} \otimes_{\R} \A^\bullet_M(\widehat\Omega_{TM}^{-\bullet})^{\otimes k}$,
\[
  [d_{S^1}, \pa_P]= [\Delta, D] \quad \text{and} \quad  [d_{S^1}, D]=[\pa_P, \Delta]=0.  
\]
\end{lem}
\begin{proof} The lemma follows from a direct computation after observing that
\[
   d_{S^1}\pi_{\alpha\beta}^*(P)=d\theta^\alpha-d\theta^\beta. 
\]
\end{proof}

\begin{defn}\label{defn-gamma-infty} Let $\gamma \in \A^\bullet_M(\W)$ (see Definition \ref{defn-AW}). We define $\gamma_\infty\in \A^\bullet_M(\widehat\Omega_{TM}^{-\bullet})[[\hbar]]$ by
\begin{align}\label{gamma-infty}
  e^{\gamma_\infty/\hbar}:=\Mult \int_{S^1[*]} e^{\hbar \pa_P+D} e^{\otimes \gamma/\hbar},
\end{align}
where
$
  e^{\otimes \gamma/\hbar}:=\sum_{k\geq 0} \frac{1}{k! \hbar^k} \gamma^{\otimes k},  \gamma^{\otimes k}\in \A^\bullet_M(\widehat\Omega_{TM}^{-\bullet})^{\otimes k}.   
$
\end{defn}
\begin{notn}\label{notn: configuration-integral-degree}
In the above definition and later discussions, we adopt the following convention: 
$$
\int_{S^1[*]}: \A^\bullet_{S^1[k]}\otimes_{\R} \A^\bullet_M(\widehat\Omega_{TM}^{-\bullet})^{\otimes k}\to  \A^\bullet_M(\widehat\Omega_{TM}^{-\bullet})^{\otimes k}
$$
is the integration $\int_{S^1[k]}$ over the appropriate configuration space $S^1[k]$. 
\end{notn}

In particular, 
\begin{align}\label{feynman-formula}
e^{\gamma_\infty/\hbar}=\sum_k \Mult \int_{S^1[k]} e^{\hbar \pa_P+D} \frac{1}{k! \hbar^k} \gamma^{\otimes k}=\sum_k \pm \Mult \int_{S^1[k]} d\theta^1\wedge\cdots \wedge d\theta^k e^{\hbar \pa_P} \frac{1}{k! \hbar^k} (d_{TM} \gamma)^{\otimes k} . 
\end{align}

\begin{rmk} 
The above formula has a combinatorial description (Feynman diagram expansion)
\begin{align}\label{graph-formula}
  \gamma_\infty=\sum_{\cG} \frac{W_\cG(P, \gamma)}{|Aut(\cG)|},
\end{align}
where the summation is over all connected graphs. $|\text{Aut}(\cG)|$ is the cardinality of the automorphism group of $\cG$. $W_\cG(P, \gamma)$ is a Feynman graph integral with $P$ as propagator labeling the edges and $d_{TM}\gamma$ labeling the vertices. This in particular implies that such defined $\gamma_\infty$ is an element of $ \A^\bullet_M(\widehat\Omega_{TM}^{-\bullet})[[\hbar]]$. We refer to Appendix \ref{Sec: Feynman} for details on such Feynman graph formula. The origin of such interpretation will be explained in Section \ref{sect:3}.
 
\end{rmk}

\begin{lem}\label{lem-BV-tansfer} Let $\gamma \in \A^\bullet_M(\W)$ and $\gamma_\infty\in \A^\bullet_M(\widehat\Omega_{TM})[[\hbar]]$ as in Definition \ref{defn-gamma-infty}. Then
\begin{align}\label{lem-BV-transfer}
\hbar \Delta  e^{\gamma_\infty/\hbar}=\Mult \int_{S^1[*]} e^{\hbar \pa_P+D} \frac{1}{2\hbar^2}[\gamma, \gamma]_\star\otimes e^{\otimes \gamma/\hbar}.
\end{align}
Here we recall that $[-,-]_\star$ is the commutator with respect to the fiberwise Moyal product. 
\end{lem}
\begin{proof} By Lemma \ref{lem-operators}, $\hbar\Delta -d_{S^1}$ commutes with $\hbar \pa_P+D$. 
\begin{align*}
\hbar \Delta  e^{\gamma_\infty/\hbar}&=\Mult \int_{S^1[*]} \hbar \Delta e^{\hbar \pa_P+D} e^{\otimes \gamma/\hbar}\\[1ex]
&=\Mult \int_{S^1[*]} d_{S^1} e^{\hbar \pa_P+D} e^{\otimes \gamma/\hbar}+\Mult \int_{S^1[*]} (\hbar\Delta -d_{S^1})e^{\hbar \pa_P+D} e^{\otimes \gamma/\hbar}\\[1ex]
&=\Mult \int_{\pa S^1[*]}  e^{\hbar \pa_P+D} e^{\otimes \gamma/\hbar}+\Mult \int_{S^1[*]} e^{\hbar \pa_P+D} (\hbar\Delta -d_{S^1}) e^{\otimes \gamma/\hbar}.
\end{align*}
The second term vanishes since $(\hbar\Delta -d_{S^1}) e^{\otimes \gamma/\hbar}=0$.

The boundary of $S^1[k]$ is given explicitly by (see Example \ref{example: boudary-circle-configuration}):
\[
   \pa S^1[k]=\bigcup_{I\subset\{1,\cdots, k\}, |I|\geq 2} \pi^{-1}(D_I),
\]
where $D_I\subset (S^1)^k$ is the small diagonal where points indexed by $I$  coincide. 

We first observe that subsets $I$ with $|I|>2$ don't contribute to the above boundary integral. This follows by type reason: in the boundary integral over $\pa S^1[k]$,  the integrand must  consist of  $k-1$ factors of $d\theta^\alpha$'s, and let us denote this product of $d\theta^\alpha$'s by  $\Theta$. It then follows that
\[
 \Theta|_{\pi^{-1}(D_I)}=\pi^*(\Theta|_{D_I})=0
\]
since $D_I$ is of codimension greater than 1 for $|I|>2$ and hence $\Theta|_{D_I}=0$. 

For $|I|=2$, say $I=\{\alpha, \beta\}$, we have the forgetful map
\[
  \pi_{\alpha\beta}: S^1[k]\to S^1[2], \quad \pi^{-1}(D_I)=\pi_{\alpha\beta}^{-1}(\pa S^1[2]). 
\]
The propagator $\pi_{\alpha\beta}^*P|_{\pi^{-1}(D_I)}=\pm 1/2$ on the two disconnected components of $\pi_{\alpha\beta}^{-1}(\pa S^1[2])$. Observe
\[
\left. \exp\bracket{\frac{\hbar}{2}\omega^{ij} \frac{\pa}{\pa y^i_\alpha} \frac{\pa}{\pa y^j_\beta}}a(y_\alpha) b(y_\beta)\right |_{y_\alpha=y_\beta=y}=(a\star b)(y), 
\]
and
\[
\quad \left. \exp\bracket{-\frac{\hbar}{2}\omega^{ij}\frac{\pa}{\pa y^i_\alpha}\frac{\pa}{\pa y^j_\beta}}a(y_\alpha) b(y_\beta)\right |_{y_\alpha=y_\beta=y}=(b\star a)(y). 
\]
Taking into account the different orientations of the two disconnected components of $\pi_{\alpha\beta}^{-1}(\pa S^1[2])$ and the permutation symmetry of $\Mult$, we have 
\begin{align*}
\Mult \int_{\pa S^1[*]}e^{\hbar \pa_P+D} e^{\otimes \gamma/\hbar}&=\Mult \int_{S^1[*-1]}e^{\hbar \pa_P+D}\frac{1}{2\hbar^2}[\gamma,\gamma]_\star \otimes e^{\otimes \gamma/\hbar}\\[1ex]
&=\Mult \int_{S^1[*]}e^{\hbar \pa_P+D}\frac{1}{2\hbar^2}[\gamma,\gamma]_\star \otimes e^{\otimes \gamma/\hbar}.
\end{align*}
\end{proof}

\begin{cor} Let $\gamma \in \A^\bullet_M(\W)$ and $\gamma_\infty\in \A^\bullet_M(\widehat\Omega^{-\bullet}_{TM})[[\hbar]]$ as in Definition \ref{defn-gamma-infty}. Then
\begin{align}\label{lem-homotopy-transfer}
(\nabla+\hbar \Delta+\hbar^{-1}d_{TM} R_\nabla)  e^{\gamma_\infty/\hbar}=\Mult \int_{S^1[*]} e^{\hbar \pa_P+D} \left (\frac{1}{\hbar}\nabla\gamma+ \frac{1}{2\hbar^2}[\gamma, \gamma]_\star+ \frac{1}{\hbar} R_\nabla \right ) \otimes e^{\otimes \gamma/\hbar}.
\end{align}
\end{cor}
\begin{proof} Since $\nabla$ is a symplectic connection, 
\[
\nabla e^{\gamma_\infty/\hbar}=\Mult \int_{S^1[*]} e^{\hbar \pa_P+D} \frac{1}{\hbar}\nabla\gamma \otimes e^{\otimes \gamma/\hbar}.
\]
On the other hand, 
\[
\Mult \int_{S^1[*]} e^{\hbar \pa_P+D} R_\nabla \otimes e^{\otimes \gamma/\hbar}=\Mult \int_{S^1[*]} e^{\hbar \pa_P+D} (d\theta^1 d_{TM} R_\nabla) \otimes e^{\otimes \gamma/\hbar}.
\]
Observe that $d_{TM} R_\nabla$ is linear in $y^i$ and $dy^i$, hence $\pa_P$ can be applied at most once to $d_{TM}R_\nabla$. If $\pa_P$ is applied exactly once to $d_{TM}R_\nabla$, then the integration over $\theta^1$ vanishes since
\[
  \int_{S^1}d\theta^1 \pi_{1\beta}^*(P)=\int_0^1 du (u-1/2)=0, \quad \beta>1. 
\]
Therefore ,
\begin{align*}
\MoveEqLeft[8] \Mult \int_{S^1[*]}e^{\hbar \pa_P+D} (d\theta^1 d_{TM}R_\nabla)\otimes e^{\otimes \gamma/\hbar}\\[1ex]
&=\Mult \int_{S^1[*]} (d\theta^1 d_{TM}R_\nabla)\otimes e^{\hbar \pa_P+D} e^{\otimes \gamma/\hbar}\\[1ex]
&= \bracket{\int_{S^1}d\theta} ( d_{TM}R_\nabla) \Mult \int_{S^1[*]}e^{\hbar \pa_P+D} e^{\otimes \gamma/\hbar}= d_{TM}R_\nabla\ e^{\gamma_\infty/\hbar}. 
\end{align*}
The corollary follows from the above identities and Lemma \ref{lem-BV-tansfer}. 
\end{proof}

\begin{proof}[Proof of Theorem \ref{main-thm}] By equation \eqref{lem-homotopy-transfer}, 
\begin{align*}
(\nabla+\hbar \Delta+\hbar^{-1}d_{TM} R_\nabla)  e^{\gamma_\infty/\hbar}&=\Mult \int_{S^1[*]} e^{\hbar \pa_P+D} \left(\frac{1}{\hbar}\nabla\gamma+ \frac{1}{2\hbar^2}[\gamma, \gamma]_\star+ \frac{1}{\hbar} R_\nabla \right) \otimes e^{\otimes \gamma/\hbar}\\[1ex]
&=\Mult \int_{S^1[*]} e^{\hbar \pa_P+D} \omega_\hbar/\hbar \otimes e^{\otimes \gamma/\hbar}=0
\end{align*}
The last equality follows for type reason: $\omega_\hbar/\hbar\in\A_M^2[[\hbar]]$ contains no factor of $d\theta$ and  $D(\omega_\hbar/\hbar)=0$. (Recall the convention of integrals on configuration spaces Notation \ref{notn: configuration-integral-degree}.) 
\end{proof}

\subsubsection{Local-to-global morphism}

\begin{defn}\label{defn-local-to-globle}
We define the \emph{local-to-global} morphism $[-]_\infty$ to be the following
\begin{align*}
 [-]_\infty:  {\A^\bullet_M(\W)}&\to{\A^\bullet_{TM}(\widehat\Omega_{TM}^{-\bullet})}[[\hbar]]\\
 O &\mapsto [O]_\infty
\end{align*}
by the equation
\begin{align}\label{observable-map}
[O]_\infty e^{\gamma_\infty/\hbar}:= \Mult \int_{S^1[*]} e^{\hbar \pa_P+D} (O d\theta^1)\otimes e^{\otimes \gamma/\hbar}.
\end{align}
\end{defn}

The name "local-to-global`` alludes to the fact that the map represents a factorization map from local observables on the interval to global observables on the circle (see section \ref{section-observable}). Our convention ensures that
\[
  [1]_\infty=\int_{S^1}d\theta=1. 
\]

\begin{thm}\label{thm-ob-cochain}
The map $[-]_\infty$ is a cochain map from the complex $\bracket{\A^\bullet_M(\W), \nabla+\hbar^{-1}\bbracket{\gamma,-}_{\star}}$ to the complex $\bracket{\A^\bullet_M(\widehat\Omega_{TM}^{-\bullet})[[\hbar]], \nabla+\hbar \Delta+\{\gamma_\infty,-\}_{\Delta}}$. In other words, 
\begin{align}\label{thm-local-to-global}
\bbracket{\nabla O+\hbar^{-1}\bbracket{\gamma,O}_{\star}}_\infty=\nabla [O]_\infty+\hbar \Delta [O]_\infty+\{\gamma_\infty,[O]_\infty\}_{\Delta}, \quad \forall O\in \A^\bullet_M(\W). 
\end{align}
\end{thm}

\begin{proof} Let 
\[
  \tilde\gamma=\gamma+\xi O d\theta , \quad \tilde \gamma_\infty=\gamma_\infty+\xi O_\infty, 
\]
where $\xi$ is a formal variable and $\xi^2=0$. Then 
\[
    e^{\tilde \gamma_\infty/\hbar}=\Mult \int_{S^1[*]}e^{\hbar \pa_P+D} e^{\otimes \tilde \gamma/\hbar}.
\]
By the same reasoning, \eqref{lem-homotopy-transfer} also holds for $\tilde \gamma$. Therefore, we find
\begin{align*}
  (\nabla+\hbar \Delta+\hbar^{-1}d_{TM} R_\nabla)e^{\tilde \gamma_\infty/\hbar}=\Mult \int_{S^1[*]}e^{\hbar \pa_P+D} \left(\hbar^{-1}\nabla \tilde \gamma+\frac{1}{2 \hbar^2}[\tilde \gamma, \tilde \gamma]_\star +\hbar^{-1}R_\nabla\right)\otimes e^{\otimes \tilde \gamma/\hbar}. 
\end{align*}
The theorem then follows by comparing the $\xi$-linear term on both sides.
\end{proof}

\subsection{The trace map}\label{section-trace} Recall that the symbol map gives an isomorphism 
\[
  \sigma: H^0(\A^\bullet_M(\W),  \nabla+\hbar^{-1}\bbracket{\gamma,-}_{\star})\to \cinfty(M)[[\hbar]]. 
\]
\begin{defn}\label{defn:trace} We define the trace map as follows
\begin{align}\label{defn-trace}
  \Tr: \cinfty(M)[[\hbar]]\to \R((\hbar)), \quad  \Tr (f)=  \int_M\int_{\gamma_\infty}  [\sigma^{-1}(f)]_\infty.
\end{align}
\end{defn}

\begin{rmk}Formula \eqref{defn-trace} also appears in \cite{FFS}, wherein the same propagator has been used. 
\end{rmk}

The following proposition explains the name "trace``. 

\begin{prop} Let $\star$ be the star product constructed from an Abelian connection (see Section \ref{sec:DQ}). Then 
\begin{align}\label{prop-trace}
\Tr(f\star g)=\Tr(g\star f), \forall f, g\in \cinfty(M).
\end{align} 
\end{prop}
\begin{proof} Introduce two formal variables $\xi_1, \xi_2$ with $\xi_1^2=\xi_2^2=0$. The analogue of  \eqref{lem-homotopy-transfer}  with $\tilde \gamma=\gamma+\xi_1 \sigma^{-1}(f)d\theta+\xi_2 \sigma^{-1}(g)$ gives 
\begin{align*}
  (\nabla+\hbar \Delta+\hbar^{-1}d_{TM} R_\nabla)e^{\tilde \gamma_\infty/\hbar}&=\Mult \int_{S^1[*]}e^{\hbar \pa_P+D} \left(\hbar^{-1}\nabla \tilde \gamma+\frac{1}{2 \hbar^2}[\tilde \gamma, \tilde \gamma]_\star +\hbar^{-1}R_\nabla\right)\otimes e^{\otimes \tilde \gamma/\hbar}\\[1ex]
  &= \Mult \int_{S^1[*]}e^{\hbar \pa_P+D} \left( \frac{\xi_1\xi_2}{\hbar^2}[\sigma^{(-1)}(f), \sigma^{(-1)}(g)]_\star d\theta^1 \right)\otimes e^{\otimes \tilde \gamma/\hbar}\\[1ex]
  &=\frac{\xi_1\xi_2}{\hbar^2}\bbracket{[\sigma^{-1}(f), \sigma^{-1}(g)]_{\star}}_\infty e^{\gamma_\infty/\hbar}.
\end{align*}

It follows that $\bbracket{[\sigma^{-1}(f), \sigma^{-1}g]_{\star}}_\infty$ is $(\nabla+\hbar \Delta+\{\gamma_\infty,-\}_\Delta)$-exact, hence 
\[
  \Tr ([f,g]_\star)=\int_{M}\int_{\gamma_\infty} [[\sigma^{-1}(f), \sigma^{-1}(g)]_{\star}]_\infty=0. 
\]
\end{proof}

\begin{prop} The leading $\hbar$-order of the trace map is given as follows:
\begin{align}\label{prop-normalization}
\Tr(f)=\frac{(-1)^n}{\hbar^n}\left(\int_M f \frac{\omega^n}{n!}+ O(\hbar) \right), \; \forall f\in \cinfty(M).
\end{align}
\end{prop}
\begin{proof}[Sketch of proof] This proposition follows from a direct computation. We sketch a quick proof in terms of the Feynman diagram expression \eqref{graph-formula}. The $\hbar$-leading term of $\int_{\gamma_\infty}[\sigma^{-1}(f)]_\infty=\int_{Ber} [\sigma^{-1}(f)]_\infty e^{\gamma_\infty/\hbar}$ comes from tree diagrams in equation \eqref{graph-formula}. Since $\int_{Ber}$ sets all bosonic variables to be zero, the only contribution to the leading $\hbar$-order comes from the leading term $\omega_{ij}dy^i dx^j$ of $d_{TM}\gamma$ which gives 
\[
  \int_M  \int_{Ber} f e^{\omega_{ij}dy^i dx^j/\hbar}=\frac{(-1)^n}{\hbar ^n}\int_M f \frac{\omega^n}{n!}. 
\]
\end{proof}

Trace maps in the context of deformation quantization have been constructed via different methods by Fedosov \cite{Fed} and Nest-Tsygan  \cite{Nest-Tsygan}. Further, it is proved in \cite{Nest-Tsygan} that there exists a unique $\R[[\hbar]]$-linear evaluation map defined on $\cinfty(M)[[\hbar]]$ satisfying equation \eqref{prop-trace} and \eqref{prop-normalization}. Formula \eqref{defn-trace} then gives a rather explicit formula for the trace map via BV integration.  It can be viewed as lifting Fedosov's work \cite{Fed}  to the chain level and is essentially equivalent to the explicit formula in \cite{FFS} via Hochschild chains. 

\subsection{The algebraic index theorem}\label{section-index} The trace map contains highly nontrivial information about the geometry of the underlying manifold $M$. For instance, consider the algebraic index theorem \cites{Fed, Nest-Tsygan}, which states that 
\[
  \Tr (1) = \int_M e^{-\omega_\hbar/\hbar} \hat A(M). 
\]
Here $\omega_\hbar$ represents the class of the deformation quantization (see Theorem \ref{thm: existence-uniqueness-Fedesov-equation}). As an application of our formula \eqref{defn-trace}, we will give an alternate proof of the algebraic index theorem by adapting an equivariant localization computation to BV integration. 

Consider the $S^1$-action on $(S^1)^k$ by 
\[
    e^{2\pi i \theta}: (e^{2\pi i \theta_1}, \cdots, e^{2\pi i \theta_k})\to  (e^{2\pi i (\theta_1+\theta)}, \cdots, e^{2\pi i (\theta_k+\theta)}). 
\]
It induces a free $S^1$-action on $S^1[k]$ such that $\pi: S^1[k]\to (S^1)^k$ is equivariant. Let us denote its infinitesimal vector field on $S^1[k]$ by $\pa_\theta$ and let $\iota_{\pa_\theta}$ and $\L_{\pa_\theta}$ be the associated contraction and Lie derivative operators respectively.

\begin{defn}\label{defn-extended-symbol}
We extend the symbol map $\sigma$ on the Weyl bundle to the BV bundle by
\begin{align}
  \sigma:  \A^\bullet_M(\widehat\Omega_{TM}^{-\bullet})\to \A^\bullet_M, \quad y^i, dy^i\mapsto 0. 
\end{align}
\end{defn}

\subsubsection{Equivariant trace map} 
We first extend our previous constructions to an $S^1$-equivariant version. 

\begin{defn} \label{defn-S1-complex} We define two $S^1$-equivariantly extended complexes by  
\begin{align*}
\A^\bullet_M(\W)^{S^1}:=\bracket{\A^\bullet_M(\W)[u,u^{-1},d\theta], \nabla+\frac{1}{\hbar}[\gamma,-]_\star-u\iota_{\pa_\theta}}, 
\end{align*}
\begin{align*}
\A^\bullet_M(\widehat\Omega_{TM}^{-\bullet})^{S^1}[[\hbar]]:=\bracket{\A^\bullet_M(\widehat\Omega_{TM}^{-\bullet})[u,u^{-1}], \nabla+\hbar \Delta+\{\gamma_\infty,-\}_{\Delta}+u d_{TM} }. 
\end{align*}
In the first formula, $d\theta$ is an odd variable of degree $1$ representing a generator of $H^1(S^1)$ and $\iota_{\pa_\theta}$ is simply the contraction map
$$
  \iota_{\pa_\theta}: d\theta a\to a, \quad a \in \A^\bullet_M(\W)[u,u^{-1}]. 
$$
Finally, $u$ is a formal variable of cohomological degree $2$. It is direct to check that $\A^\bullet_M(\W)^{S^1}$ and $\A^\bullet_M(\widehat\Omega_{TM}^{-\bullet})^{S^1}[[\hbar]]$ are differential complexes. 
\end{defn}

\begin{rmk}\label{rmk-S1} The motivation for the preceding definition is as follows.
\begin{itemize}
\item The Cartan model of the equivariant de Rham complex for an $S^1$-action on a manifold $M$ is given by \cite{GS}
$$
    \bracket{ (\A_M)^{S^1}[u], \quad d+ u\iota_{X}},
$$
where $(\A_M)^{S^1}$ are $S^1$-invariant forms, and $X$ is the vector field on $M$ generated by the $S^1$-action. In the case of $S^1$ acting on itself, we get the complex $\bracket{\C[u,d\theta], u\iota_{\pa_\theta}}$ whose further localization for $u$ gives $\bracket{\C[u,u^{-1},d\theta], u\iota_{\pa_\theta}}$. Then $\A_M(\W)^{S^1}$ is the tensor product of complexes $\A_M(\W)\otimes \C[u,u^{-1}, d\theta]$. This can be viewed as $S^1$-equivariant complex (further inverting $u$) on the subspace of maps
$$
  S^1 \to M
$$
corresponding to constant maps. This is the locus upon which our quantum field theory in Section \ref{sect:3} is modeled. 

\item The complex $(\Omega_{TM}^{-\bullet}[[\hbar]],\hbar \Delta)$ models a family of hochschild chain complexes of Weyl algebras (see Remark \ref{rmk-Hochschild}), and $d_{TM}$ plays the role of Connes's cyclic B-operator. Then $\A^\bullet_M(\widehat\Omega_{TM}^{-\bullet})^{S^1}[[\hbar]]$ can be viewed as modelling the de Rham complex of Getzler's connection \cite{Getzler-GM} on periodic cyclic homology. 
\end{itemize}

\end{rmk}

\begin{defn}
We define the $S^1$-equivariant version of the local-to-global morphism (Definition \ref{defn-local-to-globle}) by
\begin{align}
  [-]^{S^1}_\infty: \A^\bullet_M(\W)^{S^1}\to \A^\bullet_M(\widehat\Omega_{TM}^{-\bullet})^{S^1}[[\hbar]],
   \quad [O]_\infty^{S^1}e^{\gamma_\infty/\hbar} := \Mult \int_{S^1[*]} e^{\hbar \pa_P+D} O \otimes e^{\otimes  \gamma/\hbar}.
\end{align}
\end{defn}

\begin{rmk}\label{rmk-compare-equivariant-observable}
To compare with equation \eqref{observable-map}, for $O\in \A^\bullet_M(\W)$, we have
\[
   [O d\theta]^{S^1}_\infty=[O]_\infty. 
\]
\end{rmk}

\begin{lem}  $[-]^{S^1}_\infty$ is a cochain map with respect to the differential complexes as in Definition \ref{defn-S1-complex}. 
\end{lem}
\begin{proof} We need to prove that 
\[
(\nabla+\hbar\Delta+\hbar^{-1}d_{TM}R_\nabla+u d_{TM}) [O]_\infty^{S^1}e^{\gamma_\infty/\hbar} = \Mult \int_{S^1[*]} e^{\hbar \pa_P+D} \bracket{\nabla O+[\gamma/\hbar, O]_\star-u\iota_{\pa_\theta}O} \otimes e^{\otimes  \gamma/\hbar}.
\]
The analogue of Theorem \ref{thm-ob-cochain} implies that
\[
(\nabla+\hbar\Delta+\hbar^{-1}d_{TM}R_\nabla) [O]_\infty^{S^1}e^{\gamma_\infty/\hbar} = \Mult \int_{S^1[*]} e^{\hbar \pa_P+D} (\nabla O+[\gamma/\hbar,O]_\star) \otimes e^{\otimes  \gamma/\hbar}.
\]
On the other hand, 
\begin{align*}
 d_{TM} [O]_\infty^{S^1}e^{\gamma_\infty/\hbar}& = \Mult \int_{S^1[*]} d_{TM} e^{\hbar \pa_P+D} O\otimes e^{\otimes \gamma/\hbar}\\[1ex]
 &=\Mult \int_{S^1[*]} (d_{TM}-\iota_{\pa_\theta}) e^{\hbar \pa_P+D} O\otimes e^{\otimes \gamma/\hbar}\\[1ex]
 &=\Mult \int_{S^1[*]} e^{\hbar \pa_P+D}(-\iota_{\pa_\theta}) O\otimes e^{\otimes \gamma/\hbar}
\end{align*}
where we have used $d_{TM}=[\iota_{\pa_\theta},D]$ and $[d_{TM}, \hbar\pa_P+D]=0$. 
\end{proof}

\begin{defn} We define the $S^1$-equivariant BV integration map by
\[
  \int^{S^1}_{\gamma_\infty}: \A^\bullet_M(\widehat\Omega^{-\bullet}_{TM})^{S^1}[[\hbar]]\to \A^\bullet_M((\hbar))[u,u^{-1}], \quad a\mapsto \sigma\bracket{u^ne^{\hbar \iota_\Pi/u} a e^{\gamma_\infty/\hbar}}. 
\]
Here $\sigma$ is the extended symbol map as in Definition \ref{defn-extended-symbol}. 
\end{defn}

The map $\int^{S^1}_{\gamma_\infty}$ extends $\int_{\gamma_\infty}$ as follows. 

\begin{lem}\label{lem-S1-BV-extension} Let $a\in \A^\bullet_M(\widehat\Omega^{-\bullet}_{TM})$. Viewed as an element of $\A^\bullet_M(\widehat\Omega^{-\bullet}_{TM})^{S^1}$, we have
$$
\lim_{u\to 0} \int_{\gamma_\infty}^{S^1} a= \int_{\gamma_\infty}a. 
$$
\end{lem}
\begin{proof} When $a$ doesn't contains $u$, 
\begin{align*}
\lim_{u\to 0} \int_{\gamma_\infty}^{S^1} a=\frac{1}{n!}\sigma\bracket{\iota_{\Pi}^n a e^{\gamma_\infty/\hbar}}.
\end{align*}
Recall Definition \ref{defn:BVint}, the right hand side is precisely $ \int_{\gamma_\infty}a$. 
\end{proof}

The equivariant analogue of Lemma \ref{lem-cochain-BV}  is the following. 
\begin{lem} $\int^{S^1}_{\gamma_\infty}$ is a cochain map. 
\end{lem}
\begin{proof} 
\begin{align*}
 \int^{S^1}_{\gamma_\infty}(\nabla+\hbar \Delta+\{\gamma_\infty,-\}_\Delta+ud_{TM})a= \sigma\bracket{u^ne^{\hbar \iota_\Pi/\hbar}  (\nabla+\hbar\Delta+\hbar^{-1}d_{TM}R_\nabla+ud_{TM}) a e^{\gamma_\infty/\hbar}}.
\end{align*}
As in the proof of Lemma \ref{lem-cochain-BV}, 
\[
 \sigma\bracket{u^ne^{\hbar \iota_\Pi/\hbar}  (\nabla+\hbar^{-1}d_{TM}R_\nabla) a e^{\gamma_\infty/\hbar}}=d_M  \sigma\bracket{u^ne^{\hbar \iota_\Pi/u} a e^{\gamma_\infty/\hbar}}.
\]
On the other hand, 
\[
 \sigma\bracket{u^ne^{\hbar \iota_\Pi/u}  (\hbar \Delta+ud_{TM}) a e^{\gamma_\infty/\hbar}}= \sigma\bracket{ ud_{TM} u^ne^{\hbar \iota_\Pi/u}  a e^{\gamma_\infty/\hbar}}=0,
\]
where we have used $\Delta=[d_{TM}, \iota_{\Pi}]$ which implies 
$
\hbar \Delta+ud_{TM}=e^{-\hbar \iota_\Pi/u}ud_{TM}e^{\hbar \iota_\Pi/u}.
$
\end{proof}

\begin{cor}\label{prop-equivariant-trace}  The composition of  the equivariant version of local-to-global map and BV integration
\[
   \int_{\gamma_\infty}^{S^1} [-]^{S^1}_\infty: \A^\bullet_M(\W)^{S^1}\to \A^\bullet_M((\hbar))[u,u^{-1}]
\]
is a cochain map. Here the differential on $\A_M((\hbar))[u,u^{-1}]$ is simply the de Rham differetial ($u$ and $\hbar$ linearly extended). 
\end{cor}

\begin{defn} We define the $S^1$-equivariant trace map by
 \begin{align*}
   \Tr^{S^1}:  \A^\bullet_M(\W)^{S^1}  \to \R((\hbar))[u,u^{-1}], \quad
        O \mapsto  \int_M \int^{S^1}_{\gamma_\infty} [O]^{S^1}_\infty. 
 \end{align*}
\end{defn}

The relation between $\Tr$ and $\Tr^{S^1}$ is explained by the following proposition. 
\begin{prop} If $f\in \cinfty(M)$, then
\begin{align}\label{equivariant-limit}
   \Tr(f)=\Tr^{S^1}(d\theta \sigma^{-1}(f)). 
\end{align}
\end{prop}
\begin{proof} By Remark \ref{rmk-compare-equivariant-observable} and Lemma \ref{lem-S1-BV-extension}, 
\[
    \Tr(f)= \lim_{u\to 0}   \Tr^{S^1}(d\theta \sigma^{-1}(f)). 
\]
On the other hand, it is easy to see that $\Tr^{S^1}(d\theta \sigma^{-1}(f))$ is of cohomological degree $0$. Since $\deg u=2$, we conclude that
\[
 \Tr^{S^1}(d\theta \sigma^{-1}(f))=\lim_{u\to 0}   \Tr^{S^1}(d\theta \sigma^{-1}(f)). 
\]
\end{proof}

\subsubsection{The algebraic index} We are now ready to compute the algebraic index $\Tr(1)$. By equation \eqref{equivariant-limit}, it is equivalent to compute the equivariant version $\Tr^{S^1}(d\theta)$. The idea of the computation comes from equivariant localization: we show that in the $S^1$-equivariant complex, the variation of $\hbar$ leads to exact terms that do not contribute to $\Tr^{S^1}$. This is the analogue of the rigidity property of the trace of periodic cyclic cycles that is used in \cite{Nest-Tsygan} to compute the algebraic index in terms of cyclic cochains. Our  formalism can be viewed as an explicit de Rham model of \cite{Nest-Tsygan}, where $\Tr^{S^1}$ can be computed via semi-classical approximation in terms of one-loop Feynman diagrams related to $\hat A(M)$. This relates Fedosov's quantization, topological quantum mechanics and index theorem in a manifest way. The rest of this subsection is devoted to carrying out the details. 

Consider the following Euler vector field (which defines the weight decomposition of the Weyl bundle)
\[
  E=\sum y^i \frac{\pa}{\pa y^i}+ 2\hbar \frac{\pa}{\pa \hbar}. 
\]
Let $\iota_E$ be the contraction and $\L_E$ be the Lie derivative with respect to $E$. These can be extended as derivations to define natural operators 
\[
  \iota_E, \L_E: \A^\bullet_M(\widehat\Omega_{TM}^{-\bullet})^{\otimes k}\to \A^\bullet_M(\widehat\Omega_{TM}^{-\bullet})^{\otimes k}
\]
which we denote by the same symbols. 

\begin{lem}
\begin{align}\label{LE-gamma}
   \nabla \L_E(\gamma/\hbar)+[\gamma/\hbar, \L_E(\gamma/\hbar)]_\star=2\hbar \frac{\pa}{\pa \hbar}(\omega_\hbar/\hbar). 
\end{align}
\end{lem}
\begin{proof} Applying $\L_E$  to $\hbar^{-1}$\eqref{Fedosov-eqn} we obtain equation \eqref{LE-gamma}. 

\end{proof}

We now consider the $\hbar$-dependence of the expression
\[
e^{\omega_\hbar/{u \hbar}} \int^{S^1}_{\gamma_\infty} [d\theta]^{S^1}_\infty.
\]
This expression is not independent of $\hbar$. However as we next show, any $\hbar$-dependent terms are $d_M$-exact and hence will not contribute upon integrating over the manifold $M$.

\begin{cor} The expression
\begin{align}\label{rescaling-eqn}
  \hbar \pa_\hbar \bracket{ e^{\omega_\hbar/u\hbar} \int^{S^1}_{\gamma_\infty} [d\theta]^{S^1}_\infty}
\end{align}
is $d_M-$exact.
\end{cor}

\begin{proof} Observe that $e^{\hbar \iota_{\Pi}/u}$ and $e^{\hbar \pa_P+D}$ are invariant under the Euler vector field $E$. Hence
\begin{align*}
  2\hbar \frac{\pa}{\pa\hbar} \bracket{\int^{S^1}_{\gamma_\infty} [d\theta]^{S^1}_\infty}&=\sigma\bracket{ \L_E u^ne^{\hbar \iota_\Pi/u} \int_{S^1[*]} e^{\hbar \pa_P+D}  e^{\otimes \gamma/\hbar}}={\int^{S^1}_{\gamma_\infty} [\L_E(\gamma/\hbar)]^{S^1}_\infty}\\[1ex]
  &=-u^{-1} {\int^{S^1}_{\gamma_\infty} [(\nabla+[\gamma/\hbar,-]_\star-u\iota_{\pa_\theta}) d\theta \L_E(\gamma/\hbar)]^{S^1}_\infty}\\[1ex]
 &\qquad \quad +u^{-1} {\int^{S^1}_{\gamma_\infty} [(\nabla+[\gamma/\hbar,-]_\star) d\theta \L_E(\gamma/\hbar)]^{S^1}_\infty}\\[1ex]
  &=-u^{-1}{\int^{S^1}_{\gamma_\infty} [d\theta \L_E(\omega_\hbar/\hbar)]^{S^1}_\infty}+ d_M\text{-exact}\\[1ex]
  &=-\bracket{2\hbar \frac{\pa}{\pa \hbar}(\omega_\hbar/u\hbar)}{\int^{S^1}_{\gamma_\infty} [d\theta]^{S^1}_\infty}+ d_M\text{-exact},
\end{align*}
where on the second line we have used Corollary \ref{prop-equivariant-trace}, and on the third line we have used equation \eqref{LE-gamma}. The corollary follows immediately. 
\end{proof}

Given a differential form $A=\sum_{p\ \text{even}} A_p,\;  A_p\in \A^P_M$, of even degree, we let $A_u \in \A^\bullet_M [u^{-1}, u]$ denote 
\[
A_u :=\sum u^{-p/2}A_p.
\]

\begin{lem}[Semi-classical limit]\label{Semi-classical limit}
\[
\int_{\gamma_\infty}^{S^1}[d\theta]^{S^1}_\infty=u^n e^{-\omega_\hbar/u\hbar}(\hat A(M)_u+O(\hbar)).
\]
Here $\hat A(M)$ is the $\hat A$-genus of $M$. 
\end{lem} 
\begin{proof} This follows from a Feynman diagram computation. Let us write
\[
\int_{\gamma_\infty}^{S^1}[d\theta]^{S^1}_\infty=\sigma\bracket{u^n e^{\hbar \iota_\Pi/u} \int_{S^1[*]}e^{\hbar \pa_P+D} e^{\otimes \gamma/\hbar} }=\sigma\bracket{u^n  \int_{S^1[*]}e^{\hbar (\pa_P+\iota_{\Pi}/u)} e^{D}e^{\otimes \gamma/\hbar} }.
\]
Similar to equation \eqref{graph-formula}, the above formula can be written as sum of Feynman diagrams with vertex $d_{TM}\gamma$, propagator $P+ \iota_\Pi/u$, and integrated over configuration space. Let us formally write 
\[
\int_{S^1[*]}e^{\hbar (\pa_P+\iota_{\Pi}/u)} e^{D}e^{\otimes \gamma/\hbar}=\exp\bracket{\hbar^{-1}\sum_{\cG:\text{connected}} \frac{W_\cG(P+\iota_\Pi/u, d_{TM}\gamma)}{|Aut(\cG)|}}. 
\]

By Lemma \ref{lemma-curvature}, 
\[
\gamma=\omega_{ij}y^i dx^j+\frac{1}{8}R_{(ijk)l}y^iy^jy^kdx^l+ O(y^4)+O(\hbar). 
\]

It follows that the tree diagrams which survive under the symbol map $\sigma$ are given by two vertices of $\omega^{ij}dy^i dx^j$ connected by the propagator $\iota_{\iota_\Pi}/u$:
\[
\frac{\omega}{u}=\figbox{0.22}{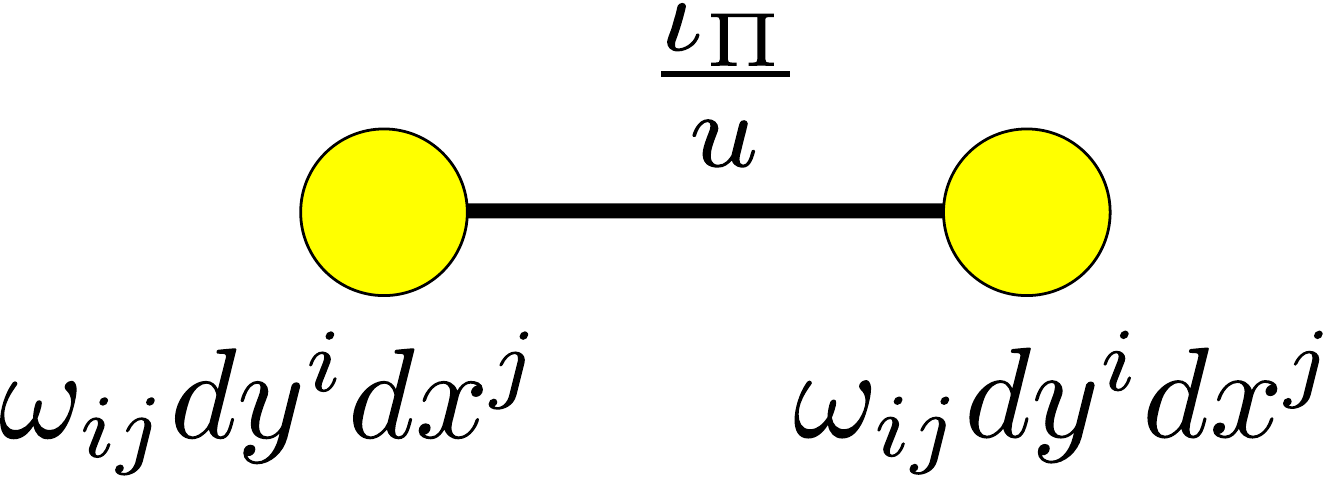}
\]
These terms together contribute a factor of
\[
   e^{-\omega/\hbar u}. 
\]

There are two types of one-loop diagrams that survive under the symbol map $\sigma$. The first type has two vertices, one of which is a tree-level vertex given by $\omega^{ij}dy^i dx^j$ while the other vertex is a one-loop vertex given by the $\hbar$-order in $\gamma$. The two vertices are connected by the propagator $\iota_{\iota_\Pi}/u$:
\[
\frac{\omega_1}{u}=\figbox{0.22}{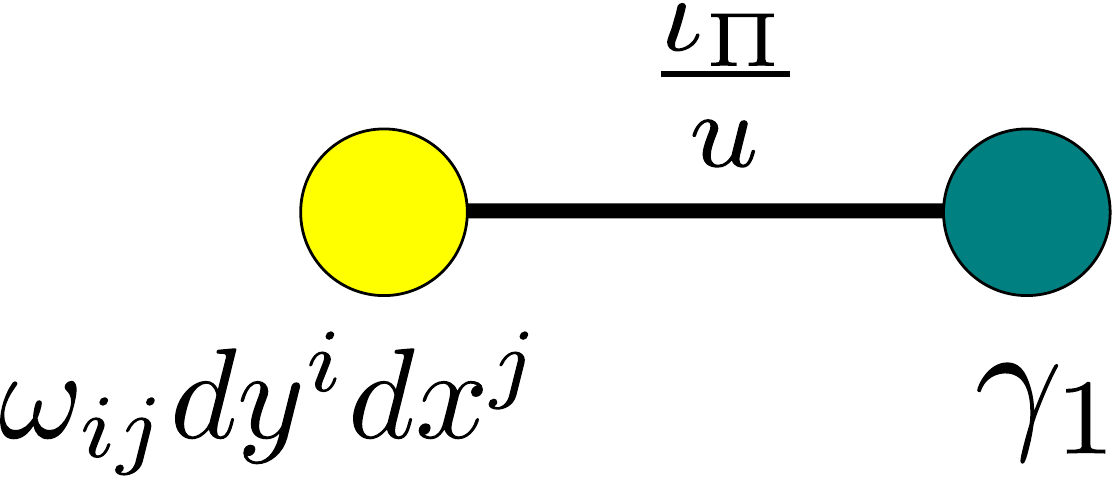}
\]
By Fedosov's equation \eqref{Fedosov-eqn}, such diagrams contribute 
\[
   e^{\omega_1/u}, 
\]
where recall that $\omega_{\hbar}=-\omega+\omega_1\hbar+\cdots$. 

The second type of diagram is a wheel of $k$-vertices decorated by $d_{TM} \bracket{\frac{1}{8}R_{(ijk)l}y^iy^jy^kdx^l}$ and internal edges decorated by the propagator $P$. Then, each vertex of the wheel is connected to another vertex of type $\omega_{ij}dy^i dx^j$ by propagator $\iota_\Pi/u$: 
\[
\figbox{0.21}{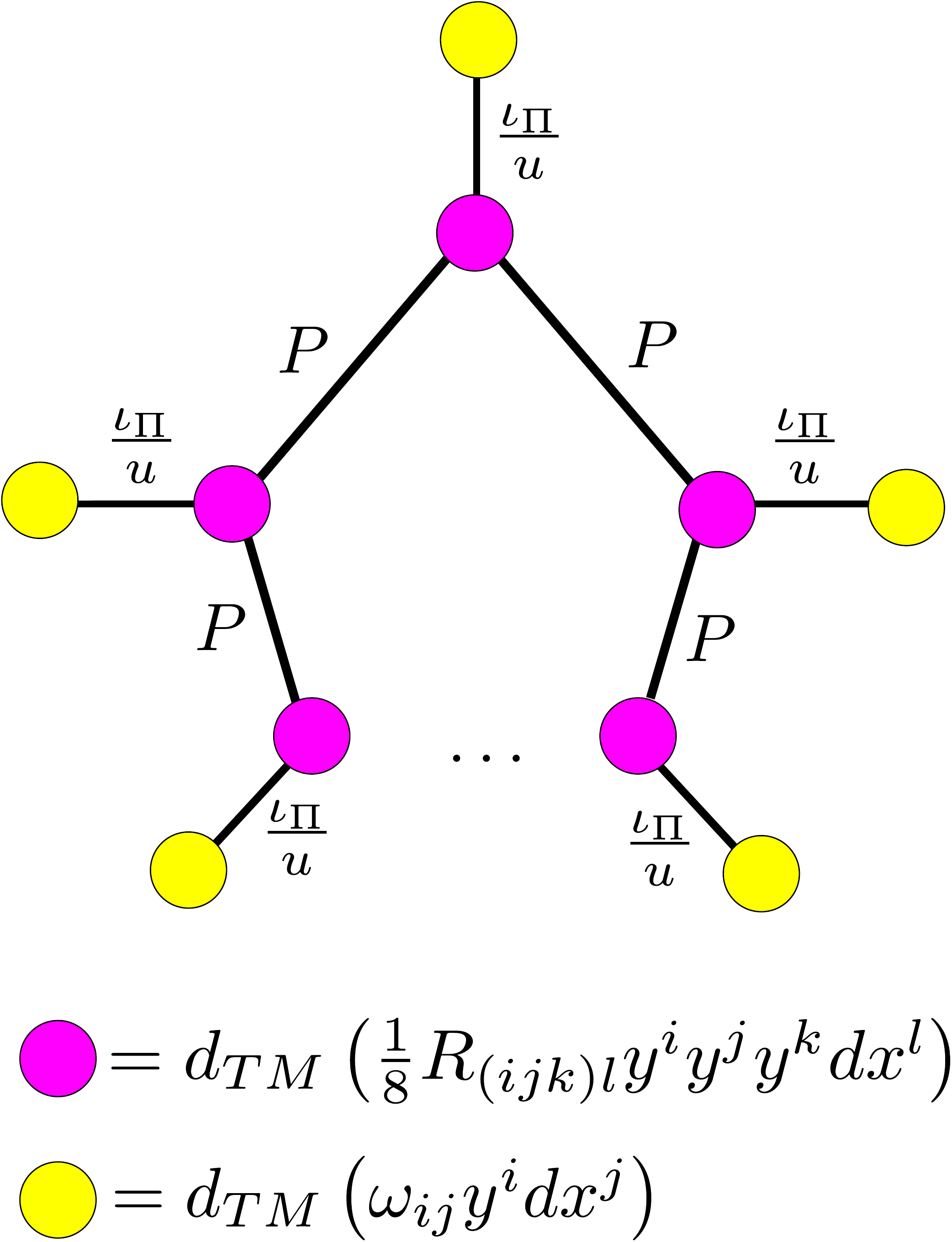}
\]
For the computation of one-loop graphs, we need the following lemma:
\begin{lem}
The following graph represents the functional $R_\nabla/u$. Here the vertex of the same color represents the same functional as in the above one-loop graph.
$$
\figbox{0.21}{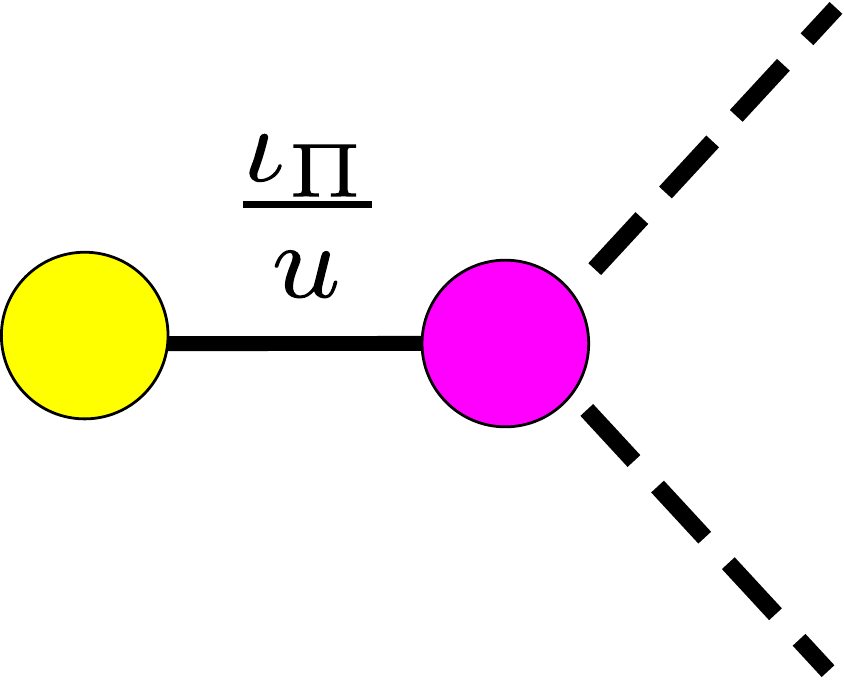}
$$
\end{lem}
\begin{proof}
It is straightforward to see that contracting a functional with the yellow vertex by the internal edge $\iota_\Pi/u$ has the effect of replacing $dy^i$ by $dx^i$. Thus the above graph represents the functional 
\begin{align*}
\frac{3}{8}R_{(ijk)l}y^iy^jdx^kdx^l&=\frac{3}{16}\left(R_{(ijk)l}y^iy^jdx^kdx^l+R_{(ijl)k}y^iy^jdx^ldx^k\right)\\
&=\frac{3}{16}\left(R_{(ijk)l}-R_{(ijl)k}\right)y^iy^jdx^kdx^l.
\end{align*}
The following computation finishes the proof of the lemma.
\begin{align*}
3\left(R_{(ijk)l}-R_{(ijl)k}\right)&=\left(R_{ijkl}+R_{jkil}+R_{kijl}\right)-\left(R_{ijlk}+R_{jlik}+R_{lijk}\right)\\
&=2R_{ijkl}+\left(R_{jkil}-R_{jlik}\right)+\left(R_{kijl}-R_{lijk}\right)\\
&=2R_{ijkl}-R_{jilk}-R_{ijlk}\\
&=4R_{ijkl},
\end{align*}
where we have used the Bianchi identity in the third equality.
\end{proof}
Thus, by contracting ``mixed edges'' (those equipped with $\iota_\Pi/u$), the above one-loop diagram is equivalent to a wheel of $k$-vertices of type  $R_{\nabla}/u$  and internal edges equipped the with propagator $P$. Such a diagram contributes 
\[
\frac{u^{-k}}{k} \Tr ((\nabla^2)^k)   \int_{S^1[k]} \pi_{12}^*(P)\pi_{23}^*(P)\cdots \pi_{k1}^*(P)d\theta^1\wedge\cdots d\theta^k.
\]
Here $\nabla^2$ is the curvature of $\nabla$. The extra factor $\frac{1}{k}$ comes from the cyclic symmetry of the diagram. Since the propagator $P$ represents the kernel of the operator $\left (\frac{d}{d\theta}\right )^{-1}$ on the circle $S^1$, we have 	
\[
 \int_{S^1[k]} \pi_{12}^*(P)\pi_{23}^*(P)\cdots \pi_{k1}^*(P)d\theta^1\wedge\cdots d\theta^k=\Tr_{L^2(S^1)} \left (\frac{d}{d\theta}\right )^{-k} =\begin{cases} \frac{2\zeta(k)}{(2\pi i)^k} & k\text{ is even} \\ 0 & k\text{ odd}\end{cases}.
\]
Here $\zeta (k)$ is Riemann's zeta function. In the last equality,  we have used the fact that $\{e^{2\pi i m\theta}\}_{m\in \Z}$ forms an $L^2$-basis, hence
\[
\Tr_{L^2(S^1)} \left (\frac{d}{d\theta}\right )^{-k} =\sum_{m\in \Z} \frac{1}{(2\pi i m)^k}. 
\]

The total contribution from such one-loop diagrams is 
\[
  \sum_{k\geq 1} \frac{u^{-k/2}}{2k} \Tr \left ( \left (\nabla^2 \right)^{2k} \right) \frac{2\zeta(2k)}{(2\pi i)^{2k}}=\sum_{k\geq 1} \frac{2 (2k-1)!\zeta(2k)}{(2\pi i)^{2k}}u^{-k/2}ch_{2k}(M), 
\]
where $ch_*$ is the Chern-character. As explained in \cite{Owen-Ryan} (see also \cite{BHJ}), this is precisely 
\[
  \log \hat A(M)_u. 
\]
The diagrams beyond one-loop will contribute to higher $\hbar$-order, therefore 
\[
\int_{\gamma_\infty}^{S^1}[d\theta]^{S^1}_\infty= u^n e^{-\omega/u\hbar+\omega_1/u}(\hat A(M)_u+O(\hbar))=u^n e^{\omega_\hbar/u\hbar}(\hat A(M)_u+O(\hbar)). 
\]
\end{proof}

\begin{rmk} The one-loop computation is essentially the semi-classical analysis in \cite{Owen-Ryan}. 

\end{rmk}
 
By Lemma \ref{Semi-classical limit} and equation \eqref{rescaling-eqn}, we find 
\begin{align*}
\int^{S^1}_{\gamma_\infty}[d\theta]^{S^1}_\infty=u^n e^{-\omega_\hbar/u\hbar}(\hat A(M)_u+d_M\text{-exact}). 
\end{align*}

Therefore the algebraic index is computed as 
\[
  \Tr(1)=\Tr^{S^1}(d\theta)=u^n \int_M e^{-\omega_\hbar/u\hbar} (\hat A(M)_u+d_M\text{-exact})=\int_M e^{-\omega_\hbar/\hbar} \hat A(M),
\]
which is independent of $u$ as expected. 

\section{One-dimensional Chern-Simons theory}\label{sect:3}

In this section, we explain the quantum field theory aspects underlying the constructions of the previous section. In more detail, we construct a one-dimensional analogue of Chern-Simons theory which is of AKSZ type and describes the sigma model from $S^1$ to a symplectic manifold $(M,\omega)$ in a neighborhood of  the constant maps. Its \BV quantization leads to a rigorous quantum field theoretical interpretation of Section \ref{section-Fedosov-BV}. In particular, Fedosov's Abelian connections correspond to action functionals satisfying the quantum master equation and the assignment $\gamma \mapsto \gamma_\infty$ is realized via the homotopic renormalization group flow in the sense of \cite{Kevin-book}. It is remarkable that Fedosov's Abelian connection incorporates perturbative quantum corrections at all loops, as required by gauge consistency at the quantum level;  we prove this in Theorem \ref{thm: QME=Fedosov's equation}. 

\subsection{\BV formalism and homotopic renormalization}\label{sec: homotopic-BV}
In this subsection we collect basics on \BV formalism and explain Costello's homotopic renormalization theory of \BV quantization \cite{Kevin-book}. We follow the presentation in \cite{Si-vertex}. 

\subsubsection{\BV master equation}

\begin{defn} A differential \BV (BV) algebra is a triple $(\A, Q, \Delta)$ where
\begin{itemize}
\item $(\A, \Delta)$ is a BV algebra (see Definition \ref{defn-BV}), and
\item $Q: \A\to \A$ is a derivation of degree $1$ such that $Q^2=0$ and $[Q, \Delta]=0$. 
\end{itemize}

\end{defn}

The induced \BV bracket $\{-,-\}_{\Delta}$ will simply be written as $\{-,-\}$ in this section. 

\begin{defn}
Let $(\A, Q, \Delta)$ be a differential BV algebra. A degree $0$ element $I_0\in \A$ is said to satisfy the \emph{classical master equation} (CME) if 
$$
QI_0+\frac{1}{2}\fbracket{I_0,I_0}=0.
$$ 
A degree $0$ element $I\in \A[[\hbar]]$ is said to satisfy the \emph{quantum master equation} (QME) if 
$$
QI+\hbar \Delta I+\frac{1}{2}\fbracket{I,I}=0.
$$
Here $\hbar$ is a formal variable representing the quantum parameter. 
\end{defn}

 The ``second-order" property of $\Delta$ implies that the QME is equivalent to 
$$
   (Q+\hbar \Delta)e^{I/\hbar}=0. 
$$
If we decompose $I=\sum\limits_{g\geq 0}I_g\hbar^g$, then the $\hbar\to0$ limit of QME is precisely CME: 
$
QI_0+\frac{1}{2}\fbracket{I_0, I_0}=0.
$

\begin{rmk}
A solution $I_0$ of CME leads to a differential $Q+\{I_0,-\}$, which is usually called the BRST operator in physics. In practice, any action functional with a local gauge symmetry can be completed into one that satisfies CME.  Upon quantization,  the QME describes the quantum consistency of gauge transformations. This fact is at the heart of the \BV formalism.
\end{rmk}

\subsubsection{(-1)-shifted symplectic geometry} A general class of differential BV algebras arises from odd symplectic geometry. 

Let us start with a finite dimensional toy model. Let $(V,Q)$ be a finite dimensional dg vector space. The differential $Q: V\to V$ induces a differential on various tensors of $V, V^*$, still denoted by $Q$. Let 
$$
   \omega \in \wedge^2 V^*, \quad Q(\omega)=0, 
$$
be a $Q$-compatible symplectic structure such that $\deg(\omega)=-1$ (this explains the name $(-1)$-shifted symplectic). Then $\omega$ induces an isomorphism
$$
     V^* \simeq V[1].
$$
Let $K=\omega^{-1}\in \Sym^2(V)$ be the Poisson kernel of degree $1$ under the identification
\[     \wedge^2 V^* \simeq\  \Sym^2(V)[2],
\]
where we have used the canonical identification $\wedge^2(V[1])\simeq \Sym^2(V)[2]$. Finally, define 
$$
\widehat\OO(V):=\widehat{\Sym}(V^*)=\prod_{n}\Sym^n(V^*).
$$
Then $(\widehat\OO(V), Q)$ is a graded-commutative dga. Let $\Delta_K$ denote the second order operator $\pa_K$ (see Conventions). It is straight-forward to see that 
$(\widehat\OO(V), Q, \Delta_K)$ defines a differential BV algebra. The above construction can be summarized as 
$$
\mbox{$(-1)$-shifted dg symplectic}  \Longrightarrow \mbox{differential BV}.
$$

\begin{rmk}\label{rmk-degenerate-BV} 
Note that the relevant data is $K$ instead of $\omega$. In fact,  for any $Q$-closed element $K\in \Sym^2(V)$ of degree $1$, the triple $(\widehat\OO(V), Q, \Delta_K)$ defines a differential BV algebra. In particular, $K$ could be degenerate corresponding to the Poisson case. For example, such degenerate BV structure appears naturally in the topological string field theory \cite{BCOV} which is fully developed in \cite{Kevin-Si} to formulate the quantum B-model aspect of mirror symmetry. 
\end{rmk}

Let $P$ be a degree $0$ element of $\Sym^2(V)$, and consider the new Poisson kernel 
$$
  K_P:= K+Q(P). 
$$
We say that $K_P$ is homologous to $K$. We can form a new differential BV algebra $(\widehat\OO(V), Q, \Delta_{K_P})$ as in the preceding remark. The following simple algebraic lemma is central to Costello's homotopic renormalization theory \cite{Kevin-book}. 

\begin{lem}\label{lem-HRG} Let $\pa_P$ denote the second order operator on $\widehat\OO(V)$ via the contraction with $P$ (see Conventions). The following equation holds formally as operators on $\widehat\OO(V)[[\hbar]]$
$$
     \bracket{Q+\hbar \Delta_{K_{P}}}e^{\hbar \pa_{P}}=e^{\hbar \pa_{P}}\bracket{Q+\hbar \Delta_{K}},
$$
i.e., the following diagram commutes
$$
\xymatrix{
    \widehat\OO(V)[[\hbar]] \ar[rr]^{Q+\hbar \Delta_{K}} \ar[d]_{\exp\bracket{\hbar \pa_{P}}} && \widehat\OO(V)[[\hbar]] \ar[d]^{\exp\bracket{\hbar \pa_{P}}}\\
   \widehat \OO(V)[[\hbar]] \ar[rr]_{Q+\hbar \Delta_{K_{P}}} && \widehat\OO(V)[[\hbar]]
}
$$
\end{lem}
\begin{proof}[Sketch of proof] This follows from the observation that 
$
  \bbracket{Q, \pa_{P}}=\Delta_{K}-\Delta_{K_{P}}. 
$
\end{proof}

\begin{defn} Let 
$$
\OO^+(V):= \Sym^{\geq 3}(V^*)\oplus \hbar \widehat\OO(V)[[\hbar]]
$$
denote the subspace of $\widehat\OO(V)[[\hbar]]$ consisting of functions that are at least cubic modulo $\hbar$. 

\end{defn}

\begin{cor}\label{cor-HRG} Let $I\in \OO^+(V)$ be a solution of quantum master equation in $(\widehat\OO(V)[[\hbar]], Q, \Delta_K)$. Let $I_P$ be  defined by the equation
$$
  e^{I_P/\hbar}= e^{\hbar \pa_P} e^{I/\hbar}. 
$$ 
Then $I_P$ is a well-defined element of $\OO^+(V)$ that solves the quantum master equation in $(\widehat\OO(V)[[\hbar]], Q, \Delta_{K_P})$.
\end{cor}
\begin{proof}[Sketch of proof] The real content of the above formal definition of $I_P$ is that 
$$
  I_P= \sum_{\Gamma: \text{connected}} \frac{\hbar^{g(\Gamma)}}{|\text{Aut}(\Gamma)|} W_{\Gamma}(P, I) 
$$
where the summation is over all connected Feynman graphs with $P$ being the propagator and $I$ being the vertex. Here $g(\Gamma)$ denotes the genus of $\Gamma$, and $|\text{Aut}(\Gamma)|$ denotes the cardinality of the automorphism group of $\Gamma$. See Appendix \ref{Sec: Feynman}. 

$I$ being at least cubic modulo $\hbar$ implies that the above infinite graph sum leads to a well-defined element of $\OO^+(V)$. The statement of $I_P$ being a solution of the quantum master equation is a direct consequence of Lemma \ref{lem-HRG}. 

\end{proof}

\begin{defn} Given $P\in \Sym^2(V)$ of degree $0$,   we define the \emph{homotopic renormalization group flow} (HRG) operator 
\begin{align*}
  W(P, -): \OO^+(V)\to \OO^+(V), \quad 
           e^{W(P,I)/\hbar}=e^{\hbar \pa_P} e^{I/\hbar}. 
\end{align*}
This expression is understood in terms of Feynman graphs as in Corollary \ref{cor-HRG}. 
\end{defn}

By Lemma \ref{lem-HRG}, we see that HRG links QME for differential BV algebras whose Poisson kernels are homologous. 
\bigskip

\begin{rmk} In \cite{Kevin-book}, $W(P,-)$ is called the renormalization group flow operator. We adopt the notion homotopic renormalization group flow operator following \cite{Si-vertex} to distinguish it from another homotopy that comes from a rescaling symmetry of the underlying manifold (which is called local renormalization group flow in \cite{Kevin-book}). 

\end{rmk}

\subsubsection{Costello's homotopic renormalization}\label{Costello-renormalization}
In real life, quantum field theories almost always have infinite dimensional space of fields playing the role of $V$ in the toy model. Typically, the dg vector space $V$ is replaced by the space of smooth sections of a graded vector bundle $E$ on a manifold $X$:
$$
 V\rightsquigarrow  \E=\Gamma(X, E). 
$$
The differential $Q$ is a differential operator on $E$ that makes $(\E, Q)$ into an elliptic complex. For example, it could be the de Rham complex or the Dolbeaut complex. The $(-1)$-shifted symplectic pairing is local:
$$
   \omega(s_1, s_2)=\int_X \bracket{s_1, s_2}, \quad s_i\in \E,
$$
where $(-,-)$ is a degree $-1$ skew-symmetric pairing on $\E$ valued in the density line bundle on $X$. 

The space of fields $\E$ carries a natural topology, and $V^*$ is replaced by the continuous dual of $\E$ (see Conventions)
$$
  \E^*=\Hom(\E, \R),
$$
i.e., distributions. Then formal functions $\widehat\OO(\E)$ are defined in terms of the completed tensor product. Precisely, 
$$
(\E^*)^{\otimes k}=\E^*{\otimes}\cdots {\otimes} \E^*
$$
are distributions on the bundle $E \boxtimes \cdots \boxtimes E $ over $X\times \cdots \times X$.  $\Sym^k(\E^*)$ is defined similarly (paying attention to the sign of permutations). Then 
$$
   \widehat\OO(\E):=\prod_{k\geq 0} \Sym^k(\E^*)
$$
is the analogue of $\widehat\OO(V)$.

Since $\omega$ is given by an integration, the Poisson kernel $K=\omega^{-1}$ is given by $\delta$-function. In particular, $K$ is \emph{singular}. Therefore, the naive analogue of the finite dimensional model
$$
  \Delta_K: \widehat\OO(\E)\to \widehat\OO(\E)
$$
is ill-defined. This difficulty is called the ultra-violet problem in quantum field theory. 

In \cite{Kevin-book}, Costello uses a homotopy to link $K$ to a smooth kernel. For example, $K$ is homologous to the heat kernel $K_L$ associated to a generlized Laplacian constructed from the elliptic operator $Q$.  Then $K_L$ is smooth for $0<L<\infty$ and $K=K_0$. Since smooth functions can be paired with distributions, the BV operator (defined by the same formula as the finite dimensional case)
$$
\Delta_L:=  \Delta_{K_L}: \widehat\OO(\E)\to \widehat\OO(\E)
$$
is well-defined. In this way we obtain a family of differential BV algebras $\{\widehat\OO(\E), Q, \Delta_L\}_{L>0}$. A (perturbative) BV quantization of such quantum field theory is  a family of solutions of the QME in $(\widehat\OO(\E), Q, \Delta_L)$ linked by HRG. This is the formalism that we will use to quantize the one-dimensional Chern-Simons theory that we will next describe.

\subsection{AKSZ formalism}
Chern-Simons theory is usually formulated on three dimensional manifolds. In \cite{AKSZ}, it is realized that Chern-Simons theory can be formulated as a sigma model. This construction is internal to the \BV formalism and can be generalized in various ways, which is nowadays called the AKSZ-construction. In this section, we describe a prototypical model to motivate a version of one-dimensional sigma model 
$$
  S^1 \to (M, \omega)
$$
from a circle to a symplectic manifold in the vicinity of constant maps. This is rooted in the AKSZ-construction, and we will just call it one-dimensional Chern-Simons theory. It can be perturbatively quantized rigorously via Costello's homotopic renormalization theory. When the target $M=T^*X$ is the total space of a cotangent bundle, this is considered in \cite{Owen-Ryan}. 

\subsubsection{AKSZ-construction}
We briefly explain the AKSZ-construction to motivate our model. We refer to \cite{AKSZ} for precise definitions of the relevant notions in super geometry.  

Let $\Sigma$ be a supermanifold equipped with a volume form of degree $-k$, and let $X$ be a supermanifold equipped with a symplectic form of degree $k-1$.  The mapping space 
$$
  Map(\Sigma, X)
$$
carries a natural $(-1)$-shifted sympletic structure $\omega$ as follows.  Given $f: \Sigma\to X$, its tangent space at $f$ is 
$$
  T_fMap(\Sigma, X)=\Gamma(\Sigma, f^* T_X). 
$$
Then 
$$
  \omega(\alpha, \beta)=\int_{\Sigma} (\alpha, \beta)_X, \quad \alpha, \beta\in T_fMap(\Sigma, X)
$$
Here $(-,-)_X$ denotes the pairing on $T_X$ that arises from the symplectic structure on $X$, and $\int_\Sigma$ denotes the integration with respect to the volume form on $\Sigma$. 

The dg structures on $\Sigma$ and $X$ generate a BRST transformation, which allows us to construct a solution of classical master equation. Action functionals obtained in this way will be called of AKSZ type. 

Classical Chern-Simons theory on a three-dimensional manifold is of AKSZ type
$$
  (M, \Omega^\bullet_M) \to B\g.
$$
Here the source is the supermanifold whose underlying space and structure sheaf is a three dimensional manifold $M$ and the de Rham complex on $M$ respectively. The target space $B\g$ is the supermanifold with its underlying space a point, and its structure sheaf the \CE complex $\wedge^\bullet  \g^*=\OO(\g[1])$ of the Lie algebra $\g$. The integration $\int_M$ produces a volume form of degree $-3$ on $(M, \Omega^\bullet_M)$. The killing form on $\g$ produces a symplectic form of degree $2$ on the shifted space $\g[1]$. Then the de Rham differential on $M$ together with the \CE differential on $B\g$ produce the standard Chern-Simons action on $M$ valued in the Lie algebra $\g$. 

This construction can be generalized to other dimensions. We will  be interested in the one dimensional case, where the source supermanifold is the de Rham complex on the circle
$
  (S^1, \Omega^\bullet_{S^1}).
$ The integration $\int_{S^1}$ produces a volume form of degree $-1$. It follows that the AKSZ construction can be formulated on a target with symplectic form of degree $0$, i.e., an ordinary symplectic manifold. This explains the name for our one-dimensional Chern-Simons theory. 

\subsubsection{The prototype for our model}\label{sec-prototype}
We describe a prototype of our one-dimensional model on a linear symplectic target space. The case for a  symplectic manifold target will be a globalization/gluing of our prototype. 

 Let $V$ be a graded vector space  with a degree 0 symplectic pairing
$$
  (-,-): \wedge^2 V\to \R. 
$$ 
The space of fields in our prototype will be 
$$
   \E= \A^\bullet_{S^1}\otimes_{\R} V. 
$$
Let $S^1_{dR}$ denote the locally ringed space whose underlying topological space is the circle $S^1$ and whose structure sheaf is the de Rham complex on $S^1$. We can view  $\varphi\in \E$ as a $V$-valued function on $S^1_{dR}$, describing a map (denoted by $\hat \varphi$)
$$
 \hat \varphi: S^1_{dR} \to V.
$$

The differential $Q=d_{S^1}$ on $\E$ is the de Rham differential on $\A^\bullet_{S^1}$. The $(-1)$-shifted symplectic pairing is 
$$
  \omega(\varphi_1, \varphi_2):=\int_{S^1} (\varphi_1, \varphi_2), \quad \varphi_i \in \E. 
$$
Given $I\in \widehat\OO(V)$ of degree $k$, it induces an element $\rho(I)\in \widehat\OO(\E)$ of degree $k-1$ via
\begin{equation}\label{equation: the-map-rho}
   \rho(I)(\varphi):=\int_{S^1} \hat\varphi^*(I), \quad \forall \varphi \in \E. 
\end{equation}
Here, $\hat\varphi^*(I)\in \A^\bullet_{S^1}$ is the valuation of $I$ at $\varphi$ ($\A^\bullet_{S^1}$-linearly extended). 

We choose the standard flat metric on $S^1$ and use  heat kernel regularization to formulate the homotopic quantum master equations described in Section \ref{Costello-renormalization}. Let $h_t(\theta_1,\theta_2)$ be the associated heat kernel function on $S^1$, where $0\leq \theta_i<1$ are coordinates parametrizing $S^1$.  Explicitly, 
$$
h_t(\theta_1, \theta_2)=\frac{1}{\sqrt{4\pi t}}\sum_{m\in \Z}e^{-(\theta_1-\theta_2+m)^2/4t}.
$$

Let $\Pi\in \wedge^2 V$ be the Poisson tensor on $V$. Let
$$
K_L= h_L(\theta_1,\theta_2)(d\theta_1\otimes 1-1\otimes d\theta_2) \Pi \in \Sym^2(\E).
$$ 
Then $K_L$ is a smooth kernel for $L>0$ and $K_0$ is the singular Poisson kernel for the inverse of $\omega^{-1}$. Let
$$
  P_\epsilon^L=\int_\epsilon^L (d^*_{S^1}\otimes 1)K_L, \quad 0<\epsilon<L,
$$ 
where $d^{*}_{S^1}$ is the adjoint of $d_{S^1}$ with respect to the chosen flat metric. Then it is easy to see that $K_L$ is homologous to $K_\epsilon$ via $P_\epsilon^L$.  

\begin{thm}\label{thm-weyl} Given $I\in\OO^+(V)$ of degree $1$, the limit 
$$
   \rho(I)[L]:= \lim_{\epsilon \to 0} W(P_\epsilon^L, \rho(I))
$$
exists as an degree $0$ element of $\OO^+(\E)$. Moreover,
\begin{itemize}
\item The family $\{\rho(I)[L]\}_{L>0}$ satisfies the following homotopic renormalization group flow equation
$$
     \rho(I)[L_2]=W(P_{L_1}^{L_2}, \rho(I)[L_1]), \quad \forall\ 0<L_1<L_2. 
$$
\item
The family $\{\rho(I)[L]\}_{L>0}$ solves the  quantum master equation
$$
      (Q+\hbar \Delta_L)e^{\rho(I)[L]/\hbar}=0, \quad \forall\  L>0,
$$
if and only if 
$$
   \bbracket{I, I}_\star=0. 
$$
Here $\OO(V)[[\hbar]]$ inherits a natural Moyal product $\star$ from the linear symplectic form $(-,-)$, and the bracket $[-,-]_\star$ is the commutator with respect to the Moyal product. 
\end{itemize}
\end{thm}

This theorem is proved in the same way as Theorem \ref{thm: QME=Fedosov's equation} using configuration space techniques. It can be viewed as a specialization of Theorem \ref{thm: QME=Fedosov's equation} to the tangent space at a point in a symplectic manifold. We will not repeat the proof here and leave it to the interested reader. This theorem illustrates the relationship between BV quantization in one dimension and the Weyl quantization. It is not surprising from here that Fedosov's deformation quantization will appear naturally when we formulate the theory on a symplectic manifold. 

\subsection{One-dimensional Chern-Simons theory}\label{sect:model}
We are now ready to present our one-dimensional Chern-Simons theory, which can be viewed as a global version of our prototype (Section \ref{sec-prototype}) over a symplectic manifold. 

\begin{defn}
Let $\g_M[1]$ denote $\A_M^\bullet(TM)$ on $M$. The space of fields in our one-dimensional Chern-Simons theory is 
\[
\mathcal{E}:=\A^\bullet_{S^1}\otimes_{\R} \g_M[1].
\]
\end{defn}
\bigskip

\begin{rmk} Note that $\g_M$ defines a local $L_\infty$-algebra on $M$ which characterizes the smooth structure of $M$. This structure was used in \cite{Owen-Ryan} to describe a semi-classical limit theory, i.e., when $M=T^* X$. In that case, the classical action functional corresponds to the $L_\infty$-analogue of the Maurer-Cartan equation and the quantization terminates at the one-loop level. As we will see, when $M$ is an arbitrary symplectic manifold, the quantization requires contributions from all loops. Consequently, we take an approach different from that of \cite{Owen-Ryan}, using the geometry of the Weyl bundle to characterize the quantization. In particular, we will not use such an $L_\infty$ interpretation in our presentation. 
\end{rmk}

Let 
\[
\E^*=\Hom_{\A^\bullet_M}(\E, \A^\bullet_M)=\Hom(\A^\bullet_{S^1}, \R)\otimes_{\R} \A_M^\bullet(T^* M)
\] 
denote the $\A^\bullet_M$-linear dual of $\E$ as an $ \A_M^\bullet(T^* M)$-valued distribution on $S^1$. 

\begin{defn}
We define the functionals on $\E$ valued in  $\A^\bullet_M[[\hbar]]$ by
\[
\widehat\OO(\E):=\widehat{\Sym}(\E^*)[[\hbar]]:=\prod_{k\geq 0} \OO^{(k)}(\E):=\prod_{k\geq 0} \Sym^k_{\A^\bullet_M}(\E^*)[[\hbar]],
\]
where $\Sym^k_{\A^\bullet_M}(\E^*)$ is the (graded)-symmetric $\A^\bullet_M$-linear completed tensor product represented by distributions on $k$ products of $S^1$. This is similar to that in Section \ref{Costello-renormalization} except that our base coefficient ring is $\A^\bullet_M$. 

Further, let $\Ol(\E)\subset \widehat\OO(\E)$ denote the subspace of local functionals on $S^1$, i.e., those of the form 
$
  \int_{S^1}  \mathcal L
$
for some lagrangian density $\mathcal L$ on $S^1$. 
\end{defn}

There exists a natural generalization of the map $\rho$ from our prototype  (equation \eqref{equation: the-map-rho})
\[
   \rho: \A^\bullet_M(\W) \to \Ol(\E)[[\hbar]]
\]
defined as follows. Given 
$
I\in \Gamma(M,\Sym^k(T^* M)),
$
we can associate an ($\A^\bullet_M$-valued) functional on $\g_M[1]$ using the natural pairing between $TM$ and $T^* M$. This functional can be  further $\A^\bullet_M$-linearly extended to $\E$ via integration over $S^1$. Explicitly, 
\begin{align}\label{rho-map}
\rho(I):\mathcal{E}\rightarrow \A^\bullet_M,\quad
\alpha \mapsto\frac{1}{k!}\int_{S^1}I(\alpha,\cdots,\alpha).
\end{align}
 Since $\rho(I)$ requires at least one input to be a $1$-form on $S^1$, our convention is that
\[
  \rho(I)=0 \quad \text{if}\ I\in \A^\bullet_M[[\hbar]] .
\]
Moreover, $\A^\bullet_M[[\hbar]]$ can be viewed as the base (coefficient) ring of our theory. 

Given a symplectic connection $\nabla$ and  $I\in \A_M^1 (\W)$,
our action functional will be of the form 
\begin{equation}\label{eqn:action-functional}
S(\alpha):=\frac{1}{2}\int_{S^1}\omega(d_{S^1}\alpha+\nabla\alpha,\alpha)  +  \rho(I)(\alpha), \quad \alpha\in \E. 
\end{equation}
Here $d_{S^1}$ is the de Rham differential on $S^1$. We have used the symplectic form $\omega$ in the first term to pair factors in $TM$. This first term constitutes the free part of the theory, which defines a derivative 
\[
   Q= d_{S^1}+\nabla. 
\]
In contrast to the general setting as in \cite{Kevin-book}, $Q^2\neq 0$ but instead gives rise to the curvature of  $\nabla$. 

We will show that in order to quantize the resulting theory, $I$ must satisfy Fedosov's equation (\ref{Fedosov-eqn}) for Abelian connections on $\W$. This can be viewed as a family version of Theorem \ref{thm-weyl}. 

\subsubsection{Regularization} We now carry out the \BV quantization of our model as outlined in Section \ref{Costello-renormalization}. Such a quantization amounts to constructing an effective action which is compatible with homotopic renormalization group (HRG) flow and satisfies the quantum master equation (QME). 

Let us fix the standard flat metric on $S^1$ and let $d_{S^1}^*$ denote the adjoint of the de Rham differential $d_{S^1}$. Let $\mathbb{K}_t\in\mathcal{E}\otimes_{\A^\bullet_M}\mathcal{E}$ denote the kernel of the $\A^\bullet_M$-linear operator 
\[
e^{-tD}: \E\to \E, 
\]
where $D:=[d_{S^1},d^*_{S^1}]$ is the Hodge Laplacian. $\mathbb{K}_t$ contains an analytic part on $S^1$ and a combinatorial part on $\g_M[1]$. If we denote by $K_t$ the heat kernel on $\A^\bullet_{S^1}$, then
\[
   K_t(\theta_1, \theta_2)=\frac{1}{\sqrt{4\pi t}}\sum_{m\in \Z}e^{-(\theta_1-\theta_2+m)^2/4t} (d\theta_1\otimes 1-1\otimes d\theta_2) \in \A^1_{S^1\times S^1}, 
\]
where $\theta_1, \theta_2\in [0,1)$ are coordinates on  two copies of  $S^1$ respectively,  consequently
\[
   \mathbb{K}_t=K_t \cdot  \frac{1}{2}\omega^{ij}\pa_i\otimes \pa_j \in \Sym^2_{\A^\bullet_M}(\E).
\]
\bigskip
\begin{rmk}
We use the following convention in the definition of the heat kernel $K_t$:
$$
e^{-tD}(\alpha)=\int_{\theta_2\in S^1}K_t(\theta_1,\theta_2)\wedge\alpha.
$$
\end{rmk}

\begin{defn}For $L>0$, we define the scale $L$ BV-operator on $\widehat\OO(\mathcal{E})$
\[
\Delta_L:\widehat\OO(\mathcal{E})\rightarrow \widehat\OO(\mathcal{E})
\]
to be the second order operator given by contracting with the smooth kernel $\mathbb{K}_L$. 
\end{defn}

For $\Phi\in \OO^{(k)}(\E)$ and $\alpha_i\in \E$, let
\[
    \abracket{\Phi, \alpha_1\otimes\cdots\otimes \alpha_k}
\]
denote the natural evaluation map. Then $\Delta_L \Phi\in \OO^{(k-2)}(\E)$ is defined explicitly by
\[
   \abracket{\Delta_L \Phi, \alpha_1\otimes \cdots \otimes \alpha_{k-2}}=\abracket{\Phi, \mathbb{K}_L\otimes \alpha_1\otimes \cdots\otimes \alpha_{k-2}}. 
\]
Note that $\Delta_L$ is well-defined since $\mathbb{K}_L$ is smooth. 

\begin{defn}
The scale $L$ effective BV bracket of $\Phi_1, \Phi_2 \in \widehat\OO(\E)$ is defined by:
\begin{equation}\label{eqn:BV-bracket}
\fbracket{\Phi_1,\Phi_2}_L:=\Delta_L(\Phi_1\cdot \Phi_2)-\Delta_L(\Phi_1)\cdot \Phi_2-(-1)^{|\Phi_1|}\Phi_1\cdot\Delta_L(\Phi_2). 
\end{equation}
\end{defn}

\begin{defn}
We define the effective propagator  by
$
\mathbb{P}_\epsilon^L=\int_{\epsilon}^L (d^*_{S^1}\otimes 1) \mathbb{K}_t dt.
$ As before, $\mathbb{P}_\epsilon^L$  defines a second order operator on $\OO(\E)$ denoted by ${\pa_{\mathbb{P}_\epsilon^L}}$. 
\end{defn}

The effective propagator also contains two parts: an analytic part $P_\epsilon^L=\int_\epsilon^L (d^*_{S^1}\otimes 1)K_t dt$ and a combinatorial part on $\g_M[1]$, such that
\[
   \mathbb{P}_\epsilon^L=P_\epsilon^L \cdot \frac{1}{2}\omega^{ij}\pa_i\otimes \pa_j \in \Sym^2_{\A^\bullet_M}(\E).
\]

The full propagator $\mathbb{P}_0^\infty$ has a simple form as computed in Proposition \ref{prop:propagator-step-function} of the Appendix.

\subsubsection{Quantum master equation} Following \cite{Kevin-book}, we describe the quantum master equation as the quantum gauge consistency
condition. We need a slight modification since $Q$ contains a connection and $Q^2\neq 0$ represents its curvature. We follow the strategy of \cite{Qin-Si}.

As we explained in Section \ref{Costello-renormalization}, the basic idea of \cite{Kevin-book} is to start with the ill-defined BV operator $Q+\hbar \Delta_0$ at $L=0$ and then use the (singular) propagator (and associated HRG operator) as a homotopy to get a well-defined effective one at $L>0$
\[
    Q+\hbar \Delta_L= e^{\hbar {\pa_{\mathbb{P}_0^L}}}(Q+\hbar \Delta_0) e^{-\hbar {\pa_{\mathbb{P}_0^L}}}.
\]
In our case, we can modify the singular BV operator as 
\[
   Q+\hbar \Delta_0+ \hbar^{-1}\rho(R_{\nabla}),
\]
where $R_{\nabla}$ is the curvature \eqref{curvature} and the last term represents multiplication by the quadratic functional $\rho(R_\nabla)$. Formally, 
\[
   [\Delta_0, \rho(R_\nabla)]=-\nabla^2,
\]
i.e., the curvature. This implies that  $Q+\hbar \Delta_0+ \hbar^{-1}\rho(R_{\nabla})$ would be a square zero operator if it were well-defined. To get a well-defined effective BV operator, we use the same HRG operator and formally define 
\[
  Q_L+\hbar \Delta_L+\hbar^{-1}\rho(R_{\nabla}):=  e^{\hbar {\pa_{\mathbb{P}_0^L}}}(Q+\hbar \Delta_0+\hbar^{-1}\rho(R_{\nabla})) e^{-\hbar {\pa_{\mathbb{P}_0^L}}}.
\]
It is not hard to see that $Q_L$ is the derivation on $\OO(\E)$ induced from the following operator on $\E$
\begin{equation}\label{QL}
Q_L:=Q+\bbracket{{\pa_{\mathbb{P}_0^L}}, \rho(R_\nabla)}=Q - \nabla^2\int_0^L d_{S^1}^* e^{-tD}dt.
\end{equation}

The following Lemma is a direct consequence of the formal consideration above.

\begin{lem}\label{lem:sqzero}
As an operator on $\widehat\OO(\E)$, we have
\[
  \bracket{  Q_L+\hbar \Delta_L+\hbar^{-1}\rho(R_{\nabla})}^2=0. 
\]
Moreover, it is compatible with the homotopic renormalization group flow in the following sense
\[
  \bracket{Q_L+\hbar \Delta_L+\hbar^{-1}\rho(R_{\nabla})}e^{\hbar {\pa_{\mathbb{P}_\epsilon^L}}} =  e^{\hbar {\pa_{ \mathbb{P}_\epsilon^L}}} \bracket{Q_\epsilon+\hbar \Delta_\epsilon+\hbar^{-1}\rho(R_{\nabla})}.
\]
\end{lem}

\begin{defn}
A family of functionals $F[L]\in \OO^+(\E)$ parametrized by $L>0$ is said to satisfy the quantum master equation (QME) if for any $L>0$,
\[
   Q_L+\hbar\Delta_L+\{F[L], -\}_L
\]
defines a square-zero operator (i.e., a differential) on $\widehat\OO(\E)[[\hbar]]$.  
\end{defn}

\begin{lem}\label{lem:qmerewrite}
$F[L]$ satisfies the QME if and only if
\[
    Q_L F[L]+\hbar \Delta_L F[L]+\frac{1}{2}\{F[L], F[L]\}_L+\rho(R_\nabla)=0
\]
modulo constant terms valued in the coefficient ring $\A^\bullet_M[[\hbar]]$.
\end{lem}

\begin{proof}Observing that
\[
  Q_L^2=-\fbracket{\rho(R_\nabla),-}_L, \quad \bbracket{Q_L, \Delta_L}=0.
\]
We have
\begin{align*}
\bracket{ Q_L+\hbar\Delta_L+\{F[L], -\}_L}^2&=Q_L^2- \fbracket{ Q_L F[L]+\hbar \Delta_L F[L]+\frac{1}{2}\{F[L], F[L]\}_L, -}_L\\[1ex]
&=-\fbracket{ \rho(R_\nabla)+Q_L F[L]+\hbar \Delta_L F[L]+\frac{1}{2}\{F[L], F[L]\}_L, -}_L.
\end{align*}
Then we observe that the coefficient ring $\A^\bullet_M[[\hbar]]$ is the center of $\OO(\E)[[\hbar]]$ with respect to the BV bracket.
 \end{proof}
\begin{rmk} Equivalently, $F[L]$ satisfies the QME if and only if 
\[
 (Q_L+\hbar \Delta_L+\hbar^{-1}\rho(R_\nabla))e^{F[L]/\hbar}=A e^{F[L]/\hbar}
\]
for some $A\in \A^1_M[[\hbar]]$, see {\cite{Qin-Si}}. The $A$ here is the constant term in Lemma \ref{lem:qmerewrite}, which is $d_M$-closed. 

\end{rmk}

As in \cite{Kevin-book}, the quantum master equation is compatible with the homotopic renormalization group equation in the following sense:  if $F[\epsilon]$ satisfies the
QME, then $F[L]=W(\mathbb{P}_\epsilon^L, F[\epsilon])$ also satisfies the QME. This can also be seen directly by combining Lemma \ref{lem:sqzero} with Lemma \ref{lem:qmerewrite}.

The locality axiom in \cite{Kevin-book} says that $F[L]$ has an asymptotic expansion (in $L$) via local functionals as $L\to 0$ if $F[L]$ is constructed from a local functional by the addition of local counter terms. Our next goal is to establish a connection between the solutions of the quantum master equation $\{F[L]\}$ and solutions of Fedosov's equation (\ref{Fedosov-eqn}) governing Abelian connections. 

\begin{thm}\label{thm: QME=Fedosov's equation}
Let $I\in \A^1_M(\W)$ be a one form valued in the Weyl bundle over $M$. Let $\rho(I)$ be the associated $\A_M^{\bullet}[[\hbar]]$-valued local functional on $\E$ defined in equation \eqref{rho-map}. 
\begin{itemize}
\item For any $L>0$, the limit \[
   \lim\limits_{\epsilon\to 0} W({\mathbb{P}}_\epsilon^L, \rho(I)) \in \OO^+(\E)
\]
exists as a functional on $\E$, which will be denoted by $\rho(I)[L]$. The family $\{\rho(I)[L]\}_{L>0}$ satisfies the following homotopic renormalization group flow (HRG) equation
$$
   \rho(I)[L_2]=W({\mathbb{P}}_{L_1}^{L_2}, \rho(I)[L_1]), \; \text{ for } \; 0< L_1\leq L_2. 
$$
\item The family of functionals $\{\rho(I)[L]\}_{L>0}$ satisfy the quantum master equation if and only if $I$ satisfies Fedosov's equation for Abelian connection as in Theorem \ref{thm: existence-uniqueness-Fedesov-equation}. 
\end{itemize}
\end{thm}

\subsubsection{The BV bundle revisited}\label{sec-BV-revisited} We are now ready to explain the construction in Section \ref{section-Fedosov-BV}.

Let $\H$ denote the space of harmonic fields
\[
\H := \H (S^1) \otimes \g_M[1]\subset \E.
\]
Here $\H(S^1):=\R[d\theta]$ is the space of Harmonic forms on $S^1$. When we consider the scale $L=\infty$, the corresponding BV kernel $\mathbb{K}_\infty$ at $L=\infty$ lies in 
$$
  \mathbb{K}_\infty\in \Sym^2(\H)\subset \Sym^2(\E). 
$$
In this way we obtain a BV kernel on the subspace $\H$.  Let
\[
\H^*=\Hom_{\A^\bullet_M}(\H, \A^\bullet_M)=\Hom(\H({S^1}), \R)\otimes_{\R} \A^\bullet_M(T^* M),
\] 
and 
\[
\widehat\OO(\H):=\widehat{\Sym}(\H^*):=\prod_{k\geq 0} \Sym^k_{\A^\bullet_M}(\H^*).
\]

The embedding $i_\H: \H\subset \E$ induces a natural map
$$
 i_\H^*: \widehat\OO(\E)[[\hbar]]\to \widehat\OO(\H)[[\hbar]]. 
$$

Let $I\in \A^1_M(\W)$ be a solution of Fedosov's equation (\ref{Fedosov-eqn}). Theorem \ref{thm: QME=Fedosov's equation} implies that $\rho(I)[\infty]$ satisfies the quantum master equation at $L=\infty$. Since the BV kernel $K_\infty\in \Sym^2(\H)$,  the restriction $i_\H^*(\rho(I)[\infty])$ satisfies a version of quantum master equation formulated in $\widehat\OO(\H)[[\hbar]]$. 
 
After unwinding various definitions, it is easy to see that 
$$\widehat\OO(\H) = \A_M(\widehat\Omega^{-\bullet}_{TM})
$$ exactly corresponds to our BV bundle in Section \ref{sec-BV-bundle}. Similarly, the quantum master equation for $i_\H^*(\rho(I)[\infty])$  corresponds to the quantum master equation in Definition \ref{defn-QME}. This explains Theorem \ref{main-thm}: the transformation 
$$
\gamma\mapsto \gamma_{\infty}
$$
in Theorem \ref{main-thm} is precisely the homotopic renormalization group flow
$$
\rho(I)[L=0]\mapsto \rho(I)[L=\infty]. 
$$
See also Theorem \ref{thm-harmonic-observable}.

\subsection{Configuration spaces and proof of Theorem \ref{thm: QME=Fedosov's equation}}\label{section-configuration}
This subsection is devoted to proving Theorem \ref{thm: QME=Fedosov's equation}; again, we will use compactified configuration spaces. The idea of using configuration spaces to quantize a classical field theory was introduced by Kontsevich in the case of Chern-Simons theory \cites{Kontsevich, Kontsevich-notes} and around the same time further developed by Axelrod and Singer in \cite{AS2}. They are the real analogue of the algebraic construction by Fulton and MacPherson \cite{Fulton-MacPherson}.

\subsubsection{HRG flow}\label{subsection:RG flow}
Utilizing compactified configuration spaces, we define the HRG flow of functionals.  By Lemma \ref{lem: lifting-propagator}, the kernel $\mathbb{P}_0^L$ as a function on $(S^1)^2$ extends to a smooth function on $S^1[2]$.  

\begin{defn}
Define $\widetilde{\mathbb{P}}_0^L$ to be the smooth lifting of $\mathbb{P}_0^L$ to $S^1[2]$ by
\[
\widetilde{\mathbb{P}}_0^L=\widetilde{P}_0^L\cdot  \frac{1}{2}\omega^{ij}\pa_i\otimes \pa_j. 
\]
\end{defn}
The following lemma allows us to exclude graphs with tadpoles  (i.e., internal edges connecting a vertex to itself) in the HRG flow of our one-dimensional Chern-Simons theory.
\begin{lem}\label{lem:self-loop-vanishing}
Let $\cG$ be a graph that contains a tadpole. Then the corresponding graph weight vanishes.
\end{lem}

\[
\figbox{0.23}{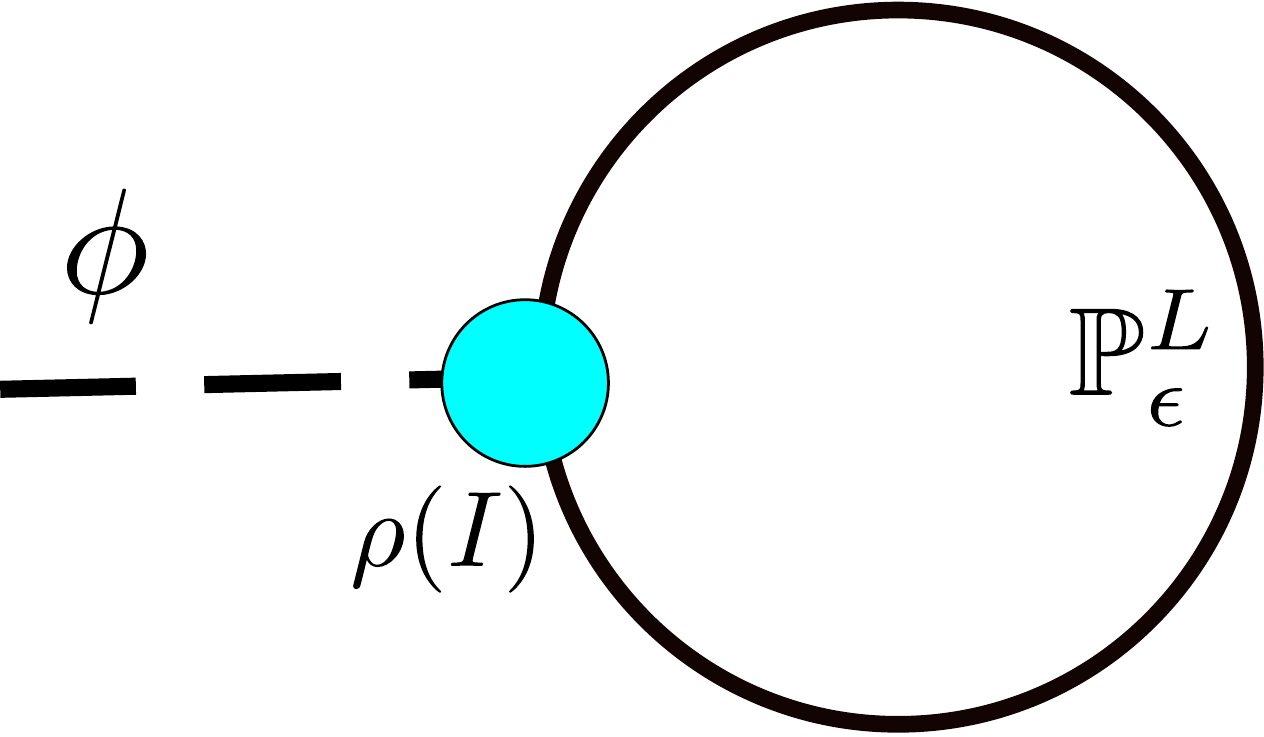}
\]

 \begin{proof} This follows as $P_\epsilon^L(\theta_1, \theta_2)=0$ when $\theta_1=\theta_2$ and the fact that the lagrangian density of $\rho(I)$ attached to the vertex does not contain derivatives. 
 \end{proof}

Let $\cG$ be  a graph without tadpoles, we denote by $V(\cG)$ and $E(\cG)$ the set of vertices and edges of $\cG$ respectively (see Appendix \ref{Sec: Feynman}). For each edge $e\in E(\cG)$, let $v(e)$ denote the
2-point set of the vertices incident to $e$. There are natural maps
\begin{align*}
&\pi_e: S^1[V(\cG)]\rightarrow S^1[v(e)], \\
&\pi_v: S^1[V(\cG)]\rightarrow S^1,
\end{align*}
for all internal edges $e$ and vertices $v$ in $\cG$. We can now use the projections $\pi_e$ to pull back $\widetilde{\mathbb{P}}_0^L$ and get a smooth function on the compactified configuration space $S^1[V(\cG)]$. Also, we can use $\pi_v$ to pull back the inputs of the Feynman graphs on the tails. Altogether we can use the compactified configuration spaces to define graph weights without worrying about singularities of the propagators. We let
\[
W_{\cG}(\widetilde{\mathbb{P}}_0^L, \rho(I)) \in \widehat\OO(\E)
\]
denote the corresponding Feynman graph integral on $S^1[V(\cG)]$. 

\begin{lem} For any $I\in \A^\bullet_M(\W)$, the limit
\[
   \lim\limits_{\epsilon\to 0} W({\mathbb{P}}_\epsilon^L, \rho(I)) \in \widehat\OO(\E)[[\hbar]]
\]
exists as a functional on $\E$, which will be denoted by $\rho(I)[L]$.  The limit value is given by 
\[
\rho(I)[L]=\sum_{\cG}\frac{\hbar^{g(\cG)}}{|\text{Aut}(\cG)|}W_\cG(\widetilde{\mathbb{P}}_0^L, \rho(I)),
\]
where the sum is over connected graphs. Moreover, $\rho(I)[L]$ satisfies the homotopic renormalization group flow (HRG) equation
$$
   \rho(I)[L_2]=W({\mathbb{P}}_{L_1}^{L_2}, \rho(I)[L_1]), \; \text{ for } \; 0< L_1\leq L_2. 
$$
\end{lem}
\begin{proof} The existence of the $\epsilon\to 0$ limit follows from the fact that the propagator $\mathbb{P}_0^L$ extends to a smooth function on $S^1[2]$, hence $W({\mathbb{P}}_\epsilon^L, \rho(I))$ (and also its limit as $\epsilon\to 0$) becomes an integration of smooth forms over a compact manifold (with corners). The fact about HRG equation is a rather formal consequence of the semi-group property: 
$$
   W({\mathbb{P}}_{\epsilon}^{L_2}, -)=W({\mathbb{P}}_{L_1}^{L_2}, W({\mathbb{P}}_{\epsilon}^{L_1}, -)).
$$
\end{proof}

This proves the first statement in Theorem \ref{thm: QME=Fedosov's equation}. 

\subsubsection{The QME revisited}\label{subsection-QME}
The next goal is to describe the quantum master equation for $\rho(I)$ in terms of configuration spaces. Since the QME is compatible with renormalization group flow, we will consider the $L\to 0$ limit of 
\[
    Q_L \rho(I)[L]+\hbar \Delta_L \rho(I)[L]+\frac{1}{2}\{\rho(I)[L], \rho(I)[L]\}_L+\rho(R_\nabla).
\]
It is in this limit that Fedosov's equation (\ref{Fedosov-eqn}) appears.  

\begin{lem}\label{lem:differential-of-propagator}
On the space $S^1[2]$ we have 
\[
d(\widetilde{\mathbb{P}}_0^L)=-\mathbb{K}_L.
\]
\end{lem}

\begin{proof}
Since the propagator $\widetilde{\mathbb{P}}_0^L$ is defined by lifting $\mathbb{P}_0^L$ smoothly from $S^1\times S^1\setminus\Delta$ to $S^1[2]$, we only need to show the equality on $S^1\times S^1\setminus\Delta$. But this is clear since $d({\mathbb{P}}_0^L)=\mathbb{K}_0-\mathbb{K}_L$ and the support of $\mathbb{K}_0$ is contained in the diagonal $\Delta$.
\end{proof}

The functional $\rho(I)[L]$ consists of graph weights $W_\cG (\widetilde{\mathbb{P}}_0^L, \rho(I))$. The de Rham differential $d_{S^1}
\rho(I)[L]$ will contribute an integration over the boundary of the configuration space. To illustrate this, let us consider a  graph weight with
external inputs $\phi_1, \cdots, \phi_k$ and propagator $\widetilde{\mathbb{P}}_0^L$: 
\[
    \varpi_\cG (\phi_1, \cdots, \phi_k)=\int_{S^1[V(\cG)]}\prod_{e\in E(\cG)}\pi_e^*(\widetilde{\mathbb{P}}_0^L)\prod_{i=1}^k\phi_i.
\]

\begin{rmk}
The term  $\prod_{i=1}^k\phi_i$ in the above equation is a little vague. Actually we are taking the sum of all possible ways of assigning $\phi_i$'s
to the tails in $\cG$. 
\end{rmk}

Then,
\begin{align*}
\MoveEqLeft[2] d_{S^1}\varpi_\cG(\phi_1,\cdots,\phi_k)\\[1ex]
=&\sum_{i=1}^k \varpi_\cG(\phi_1,\cdots,d_{S^1}(\phi_i),\cdots,\phi_k)\\[1ex]
=& \int_{S^1[V(\cG)]}d\left(\prod_{e\in E(\cG)}\pi_e^*(\widetilde{\mathbb{P}}_0^L)\prod_{i=1}^k\phi_i\right)-\int_{S^1[V(\cG)]}\sum_{e_0\in E(\cG)}d(\pi_{e_0}^*(\widetilde{\mathbb{P}}_0^L))\prod_{e\in E(\cG)\setminus e_0}\pi_e^*(\widetilde{\mathbb{P}}_0^L)\prod_{i=1}^k\phi_i   \\[1ex]
 \overset{(1)}{=}  &\int_{\partial S^1[V(\cG)]}\prod_{e\in E(\cG)}\pi_e^*(\widetilde{\mathbb{P}}_0^L)\prod_{i=1}^k\phi_i-\int_{S^1[V(\cG)]}\sum_{e_0\in E(\cG)}\pi_{e_0}^*(\mathbb{K}_L)\prod_{e\in E(\cG)\setminus e_0}\pi_e^*(\widetilde{\mathbb{P}}_0^L)\prod_{i=1}^k\phi_i,
\end{align*}
where we have used Stokes' theorem and Lemma \ref{lem:differential-of-propagator} in the last equality (1). The edge $e_0$ in the above equation 
could be separating or not. Namely, deleting the edge $e_0$ results in either two connected graphs or one connected graphs as shown in the following pictures:
\begin{align*}
&\text{separating:}\hspace{27mm} \figbox{0.2}{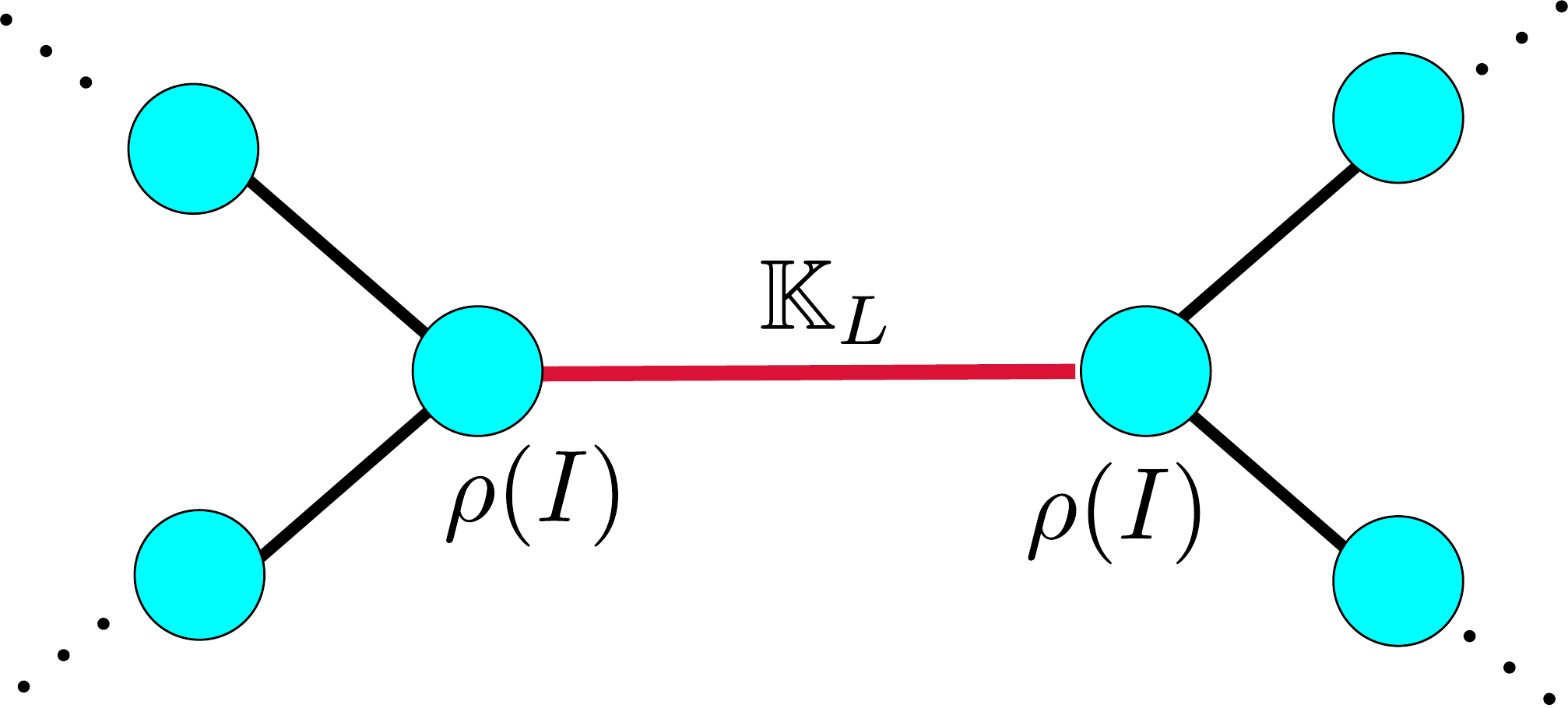}\\
&\text{non-separating:}\hspace{20mm} \figbox{0.19}{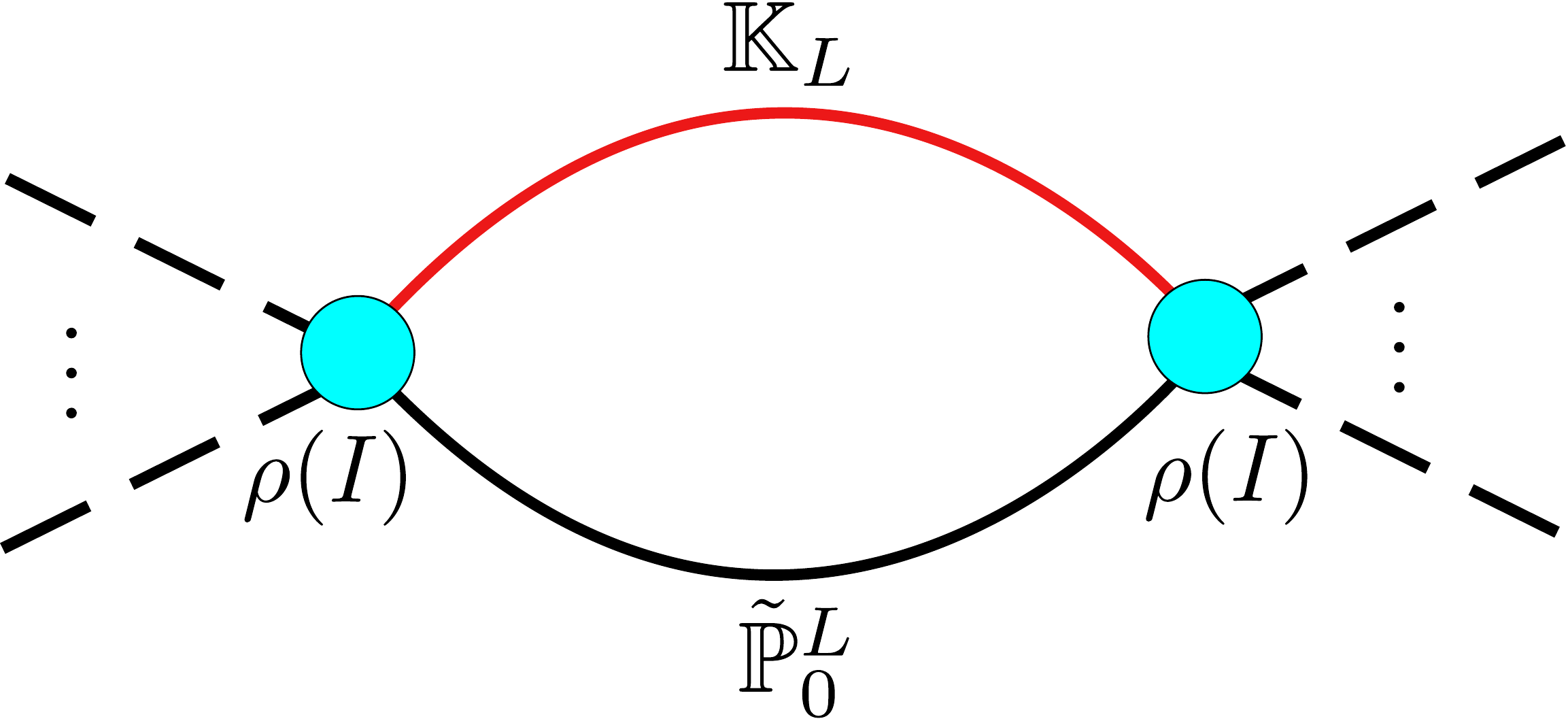}
\end{align*}

We apply this formula to $d_{S^1} \rho(I)[L]$. The sum of  the terms corresponding to non-separating edges over all  graphs $\cG$ will cancel with the term $\hbar\Delta_L\rho(I)[L]$ in the quantum master equation and the sum of all the other terms (corresponding to separating edges) will cancel with those terms in $\{\rho(I)[L],\rho(I)[L]\}_L$ .

Let us now consider the integrals over the codimension $1$ strata of $S^1[V]$.  As explained in the Appendix (in particular, Example \ref{example: boudary-circle-configuration}), these strata correspond bijectively to subsets $I\subset V$ with cardinality $\geq 2$ and we denote them by $S^1(I)=\pi^{-1}D(I)$. 

\begin{lem} For $I\neq V$ and $\cG$ connected,
\[
\lim_{L\to 0} \int_{S^1(I)}\prod_{e\in E(\cG)}\pi_e^*(\widetilde{\mathbb{P}}_0^L)\prod_{i=1}^k\phi_i=0.
\]
\end{lem}
\begin{proof} Assume $I\neq V$. Since $\cG$ is connected, there exists an internel edge $e\in E(\cG)$ connecting a vertex in $I$ to a vertex in $V\backslash I$. For $-\frac{1}{ 2}<\theta_1-\theta_2<\frac{1}{2}$,
\begin{align*}
  P_0^L(\theta_1, \theta_2)&=\int_0^L \frac{dt}{\sqrt{4\pi t}} \sum_{n\in \Z} \frac{(\theta_1-\theta_2+n)}{4t}e^{-\frac{(\theta_1-\theta_2+n)^2}{4t}}\\[1ex]
  &=\int_0^L \frac{dt}{\sqrt{4\pi t}} \frac{(\theta_1-\theta_2)}{4t}e^{-\frac{(\theta_1-\theta_2)^2}{4t}}+\int_0^L \frac{dt}{\sqrt{4\pi t}} \sum_{n\neq 0} \frac{(\theta_1-\theta_2+n)}{4t}e^{-\frac{(\theta_1-\theta_2+n)^2}{4t}}.
\end{align*}
The second term is uniformly bounded for $L\leq 1$ and vanishes as $L\to 0$. The first term equals 
\begin{align*}
  \pm \frac{1}{8\sqrt{\pi}} \int_\frac{2\sqrt{L}}{|\theta_1-\theta_2|}^\infty dt e^{-t^2}
\end{align*}
which is also uniformly bounded for $L\leq 1$. It follows that $\widetilde P_0^L$ is uniformly bounded as $L\to 0$. Since $\lim_{L\to 0}\widetilde P_0^L=0$ outside $\pa S^1[2]$, the lemma follows from the dominated convergence theorem. 

\end{proof}

\begin{lem}
Let $|I|\geq 3$, then the integral over the boundary stratum corresponding to $I$ vanishes.
\end{lem}

\begin{proof} This is similar to the argument in Lemma \ref{lem-BV-transfer}. When $|I|\geq 3$, the corresponding boundary stratum is a sphere bundle with the dimension of the fibers at least 1. On one hand, the propagators
are only zero forms and on the other hand all the inputs attached to the tails becomes smooth forms on the base of this fiber bundle. Thus the
integral vanishes for type/dimension reason.  
\end{proof}

In the $L\to 0$ limit, it remains to consider the cases with $|V(\cG)|=|I|=2$, i.e., those two-vertex graphs. The corresponding boundary integral gives rise to a local functional on $\E$ denoted by $O_2$ which admits the following description.

\begin{prop}
The boundary integral corresponding to all the two-vertex diagrams is the functional 
\[
O_2=\frac{1}{2\hbar}\rho([I,I]_\star). 
\]
\end{prop}

\begin{proof} 
The boundary of $S^1[2]$ consists of two copies of $S^1$, on which the values of $\widetilde P_0^L$ are $\pm 1/2$.
Consider the two vertex diagram with $k$ loops, hence $k+1$ propagators. The following depicts the integral over the boundary of
configuration
space. 
\[
\figbox{0.23}{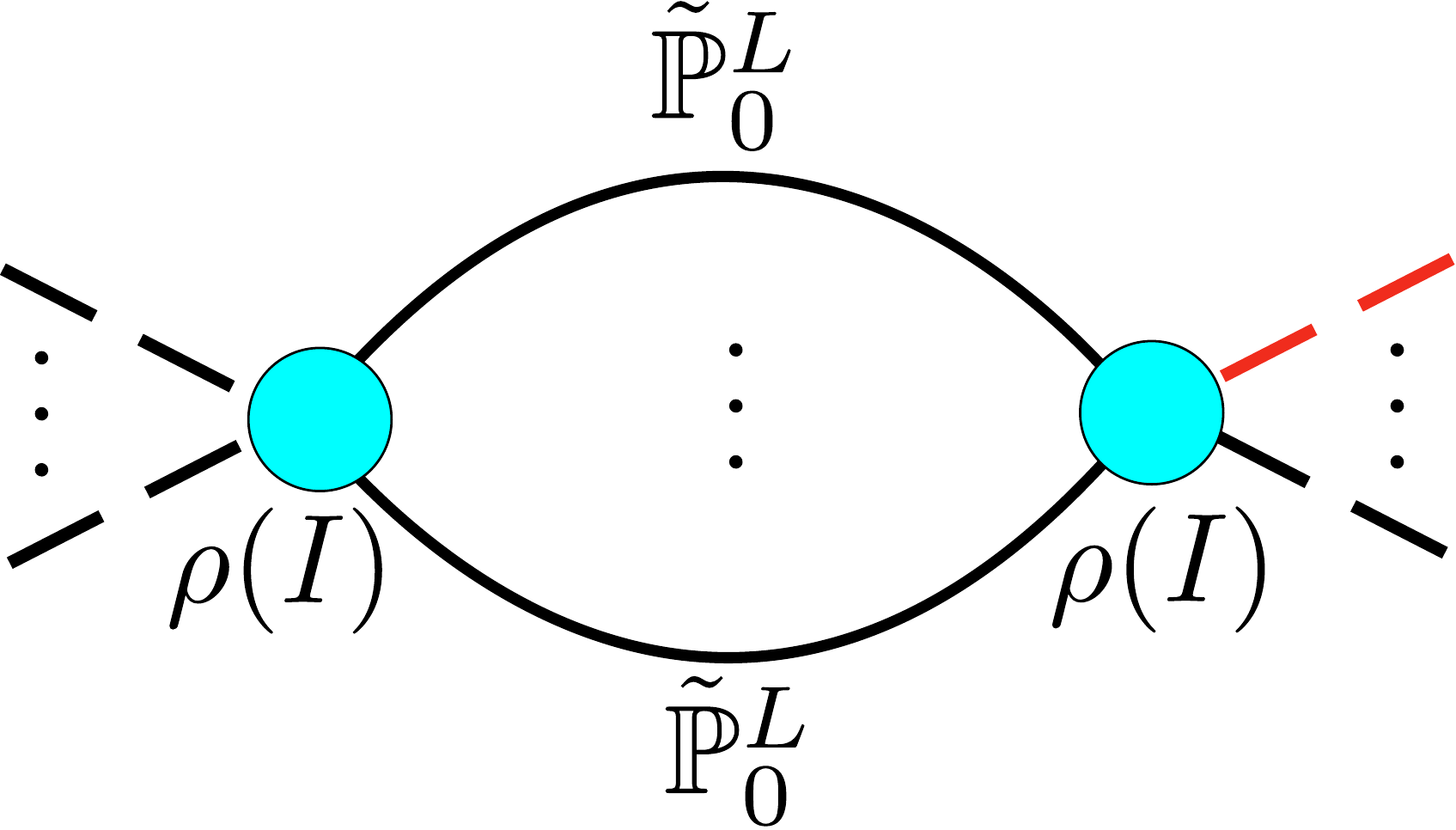}=\frac{1}{(k+1)!}\cdot \rho\left(\figbox{0.23}{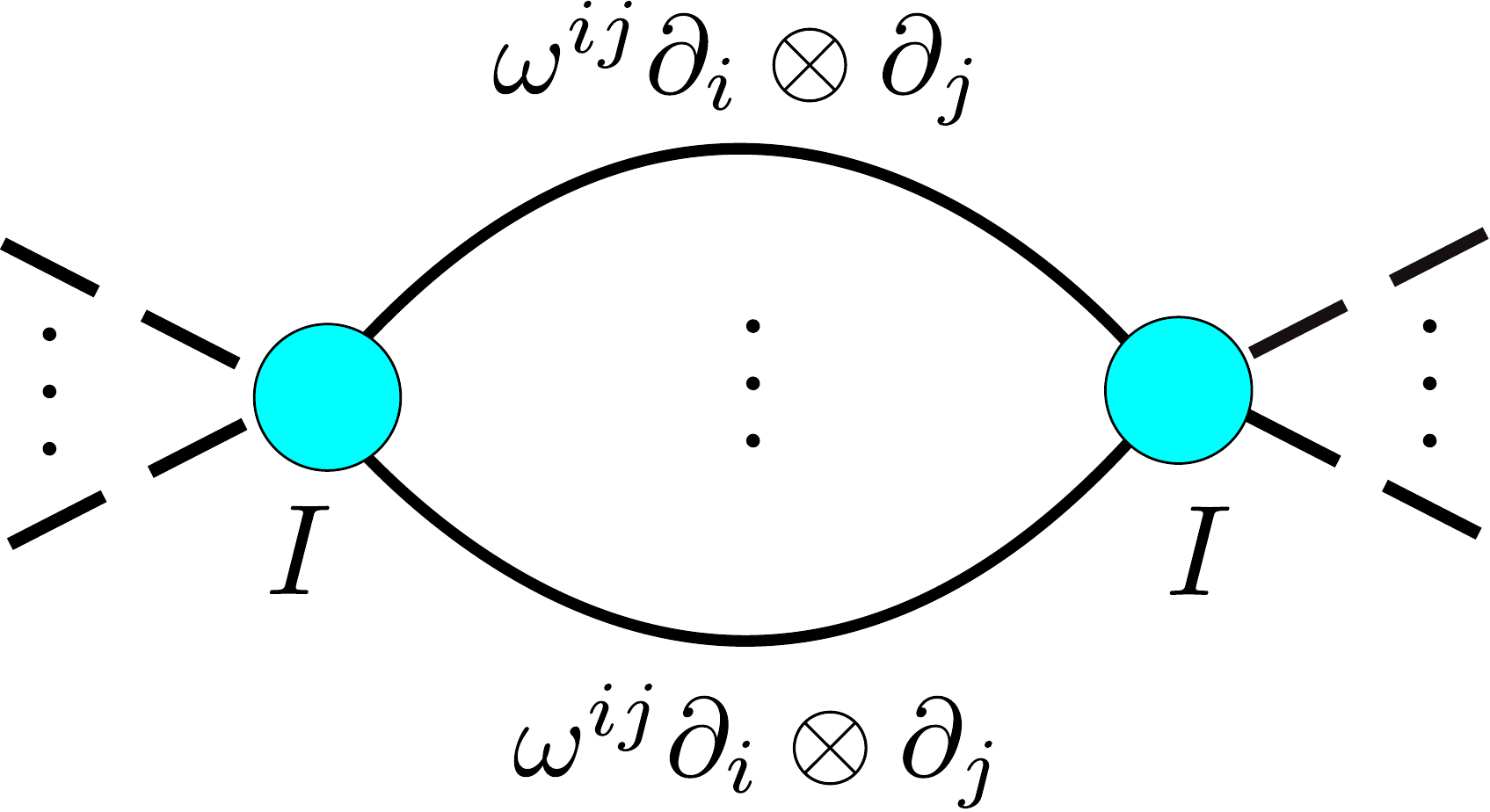}\right)
\]
By type reason, the left picture above contains exactly one input of a $1$-form on $S^1$ which labels the red tail. The
analytic part of $\widetilde{\mathbb{P}}_0^L$ will contribute $\frac{1}{(k+1)!}$, which together with the combinatorial factor $\omega^{ij}\pa_i\otimes
\pa_j$ exactly gives the formula of the Moyal product. Finally, the constant term in $I * I$, which is annihilated by $\rho$, doesn't appear in $O_2$ for type reason (since $O_2$
contains at least one input of $1$-form on $S^1$).  
\end{proof}

We are now in a position to prove the second statement of Theorem \ref{thm: QME=Fedosov's equation}, relating the QME to Fedosov's equation. Let $\gamma$ be a solution of Fedosov's equation of Abelian connection. From the above computation, we have 
\begin{align*}
\MoveEqLeft[8] \lim_{L\to 0} \left ( Q_L \rho(I)[L]+\hbar \Delta_L \rho(I)[L]+\frac{1}{2}\{\rho(I)[L], \rho(I)[L]\}_L+\rho(R_\nabla)\right )\\
  = & \rho(\nabla(I))+O_2+\rho(R_\nabla)=\rho\bracket{\nabla I+\frac{1}{2\hbar}[I, I]_*+R_\nabla}=\rho(\omega_\hbar)=0,
\end{align*}
since $\omega_\hbar\in \A^2_M[[\hbar]]$ lies in the coefficient ring. It follows from the compatibility of homotopic renormalization group flow equation with quantum master equation that $\rho(I)[L]$ solves the quantum master equation at any scale $L>0$. This argument can be reversed. We have completed the proof of Theorem \ref{thm: QME=Fedosov's equation}.

\subsection{Observable theory}\label{section-observable}
In this subsection we explain the significance of the local-to-global (cochain) morphism in Definition \ref{defn-local-to-globle}. This is related to the factorization algebra of quantum observables in our one-dimensional Chern-Simons theory. This observable theory in the \BV formalism is fully developed in \cite{Kevin-Owen}. 

\subsubsection{Local observables}
Let us first recall the definition of local quantum observables.

\begin{defn}
A local quantum observable $O$ supported on an open set $U\subset S^1$ is an assignment 
\[
L\mapsto O[L]\in\widehat\OO(\E)[[\hbar]], \ L>0
\]
such that 
\begin{enumerate}
\item The following homotopic RG flow equation is satisfied:
\[
\rho(I)[L]+\xi\cdot O[L]=W({\mathbb{P}}_\epsilon^L,\rho(I)[\epsilon]+\xi\cdot O[\epsilon]), \quad \forall 0<\epsilon<L. 
\]
Here $\xi$ is a formal variable with $\xi^2=0$.
\item As $L\rightarrow 0$, the observable becomes supported on $U$. 
\end{enumerate}

\begin{rmk}\label{rmk-support} The meaning of the support condition $(2)$ is quite involved and we choose not to state it precisely but refer to \cite{Kevin-Owen} for a thorough discussion. Instead, we sketch the basic idea of \cite{Kevin-Owen}. Naively, the BV kernel $K_L$ is getting close to the $\delta$-function as $L\to 0$, which can be used to control the support to be compatible with the homotopic RG flow equation. To make this naive thought precise, we have to replace the heat kernel regularization by regularizations of the $\delta$-function whose supports are arbitrary close to that of the $\delta$-function, i.e. the diagonal.  Then the support condition $(2)$ can be precisely formulated. The heat kernel regularization is related to an arbitrary regularization via the homotopic renormalization group flow. 
\end{rmk}

We will let $\text{Obs}^q(U)$ denote the cochain complex of local quantum observables on $U$, where we equip it with the quantum differential
\[
Q_L+ \{\rho(I)[L],-\}_L+\hbar \Delta_L. 
\]
\end{defn}

\begin{thm}\label{thm-local-obs} Let $U\subset S^1$ be a small open interval, then the complex $\text{Obs}^q(U)$ of local quantum observables is quasi-isomorphic to $(\A^\bullet_M(\W), \nabla+\frac{1}{\hbar}[\gamma, -]_\star)$. 
\end{thm}

\begin{proof}  We consider the spectral sequence of the  filtration $F^k$ on $\widehat\OO(\E)[[\hbar]]$ defined by 
$$
F^k= \sum_{p+q\geq k}\A^p_M \hbar^q  \widehat\OO(\E)[[\hbar]]. 
$$
The differential of the graded complex is given by $d_{S^1}$.

Let $\E(U)=\A^\bullet_{U}\otimes_{\R} \g_M[1]$ be the fields supported on $U$. The $E_1$-page is given by the $d_{S^1}$-cohomology and 
\[
  H^*(\E(U)^*, d_{S^1})\iso \delta_p \otimes (\g_M[1])^*,
\]
where $p\in U$ is an arbitrary point and $\delta_p$ represents the $\delta$-function distribution at $p$
\[
 \delta_p:  f\to f_0(p), \quad f=f_0+f_1, \quad f_i\in \A^i_U.  
\]
Since $\delta_p$ is concentrated in degree $0$, it follows easily that the cochain complex of local quantum observables is quasi-isomorphic to the $E_1$-page of this spectral sequence. At the $E_1$-page, 
\[
  \text{Obs}(U) \iso \Sym_{\A^\bullet_M}( \delta_p \otimes (\g_M[1])^*)\iso \Sym_{\A^\bullet_M}( \g_M[1])^*)=\A^\bullet_M(\W). 
\]
We still need to analyze the differential.

Let $O\in \A^\bullet_M(\W)$ and write 
\[
   \delta_p(O)\in \Sym_{\A^\bullet_M}( \delta_p \otimes (\g_M[1])^*)
\]
for the $\A^\bullet_M$-valued distribution on the product of $S^1$'s under the above identification. We obtain an observable $O[L]\in \text{Obs}^q(U)$ by the formula 
\[
\rho(I)[L]+\xi\cdot O[L]=\lim_{\epsilon\to 0}W(\widetilde{\mathbb{P}}_\epsilon^L,\rho(I)+\xi\cdot \delta_p(O)), \quad \xi^2=0,
\]
where the limit is well-defined as before. We only need to show that the map
\[
  \A^\bullet_M(\W)\to \text{Obs}(U), \quad O\to O[L], 
\]
is a cochain morphism. This is similar to the argument in section \ref{subsection-QME} and we only sketch it here. 

Let
\[
  \pi: S^1[k]\to (S^1)^k
\]
be the blow down map. $O[L]$ is obtained via a graph integral on the subspace $\pi^{-1}(p\times (S^1)^{k-1})\subset S^1[k]$.  When we apply the  quantum differential, we end up with two types of boundary integrals over the configuration space. The first one is over
\[
  p\times \pa S^1[k-1],
\]
which vanishes by the quantum master equation applied to $\rho(\gamma)$ as in Theorem \ref{thm: QME=Fedosov's equation}. The second one is over
\[
  \pi^{-1}(p\times p\times (S^1)^{k-2}),
\]
which again consists of two connected components. This boundary integral in the $L\to 0$ limit gives 
\[
   \{\rho(\gamma), \delta_p(O)\}_{L\to 0}=\delta_p([\gamma/\hbar, O]_\star).
\]
It follows that
\[
\lim_{L\to 0}Q_L O[L]+ \{\rho(I)[L],O[L]\}_L+\hbar \Delta_LO[L]=\nabla \delta_p(O)+ \delta_p ([\gamma, O]_\star)=\delta_p(\nabla O+[\gamma/\hbar, O]_\star). 
\]
Running the renormalization group flow backward, we find
\[
Q_L O[L]+ \{\rho(I)[L],O[L]\}_L+\hbar \Delta_LO[L]=(\nabla O+[\gamma/\hbar, O]_\star)[L]
\]
as expected. 
\end{proof}

\begin{rmk} Let $U_1, U_2$ be two disjoint open intervals and $U_1, U_2\subset U$ contained in a bigger open interval $U$. One important feature of observables is the existence of a factorization product of cochain complexes 
\[
   \text{Obs}^q(U_1)\otimes \text{Obs}^q(U_2)\to \text{Obs}^q(U). 
\]
It can be shown via similar arguments on configuration space that under the equivalence in Theorem \ref{thm-local-obs}, this product is identified with the Moyal product on $\A^\bullet_M(\W)$. Alternatively, one can identify the factorization algebra, $\text{Obs}^q$, via a localization/descent construction as detailed in \cite{GGW}.
\end{rmk}

\subsubsection{Global observable}\label{sec:global-observable} Consider the cochain complex of global quantum observables at scale $L$ 
\[
\text{Obs}^q (S^1)[L] := \left ( \widehat\OO (\mathcal{E})[[\hbar]], Q_L + \{\rho(I)[L], - \}_L + \hbar \Delta_L \right ).
\]
A global quantum observable is an assignment $L \mapsto O[L] \in \text{Obs}^q (S^1)[L]$ compatible with HRG flow (as explained in the previous section).

HRG flow defines a homotopy equivalence between the complexes of global quantum observables at different scales (this follows from compatibility between the QME and RG flow which is Lemma \ref{lem:sqzero} , see \cite{Kevin-book} for further details) and for simplicity we consider the complex of observables at scale $L= \infty$; we will denote this complex by $\text{Obs}^q$.  It follows that (up to an explicit homotopy equivalence) we have the quasi-isomorphism 
\[
\text{Obs}^q \simeq \left ( \widehat\OO (\H) , Q_\infty +\{  i_\H^*\rho(I)[\infty]  , - \}_\infty + \hbar \Delta_\infty \right ),
\]
where $\H$ denotes the space of harmonic fields $
\H := \H^\ast (S^1) \otimes \g_M[1]
$ and we have kept the same notations as in Section \ref{sec-BV-revisited}. Recall that $\g_M[1]=\A^\bullet_M(T_M)$ implies the identification \begin{align}\label{nifty-obs}
\widehat\OO(\H) = \A^\bullet_M(\widehat\Omega^{-\bullet}_{TM}). 
\end{align}

\begin{thm}\label{thm-harmonic-observable} Let $\gamma$ be a solution of Fedosov's equation \eqref{Fedosov-eqn} and let $\gamma_\infty$ be as in Definition \ref{defn-gamma-infty}. Then under the identification \eqref{nifty-obs}, we have
\[
  i_\H^* \rho(\gamma)[\infty]= \gamma_\infty. 
\]
Moreover, we have the quasi-isomorphic complexes
\[
\text{Obs}^q \simeq (\A^\bullet_M(\widehat\Omega_{TM}^{-\bullet})[[\hbar]], \nabla+\hbar \Delta+ \{\gamma_\infty,-\}_\Delta). 
\]
\end{thm}

\begin{proof}
Recall that we can express RG flow for a functional $F$ by
\[
e^{W(\widetilde{\mathbb{P}}_\epsilon^L, F)/\hbar} = e^{\hbar \partial_{\widetilde{\mathbb{P}}_\epsilon^L}} e^{F/\hbar}.
\]
So we need to show that (upon restricting to $\H$) we have
\[
e^{\gamma_\infty / \hbar} = e^{W(\widetilde{\mathbb{P}}_0^\infty, \rho(\gamma))/\hbar} = e^{\hbar \partial_{\widetilde{\mathbb{P}}_0^\infty}} e^{\rho (\gamma)/\hbar}.
\]
But this is a direct consequence of formula \eqref{feynman-formula}. In fact, the propagator  $\pa_{\widetilde{\mathbb P}_0^\infty}$ coincides with $\pa_P$ in equation \eqref{feynman-formula}. The vertex $\rho(\gamma)$ coincides with $d_{TM}\gamma$ upon restricting to $\H$, where $y^i$ represents a functional with harmonic $0$-form input on $S^1$ and $dy^i=d_\W y^i$ represents a functional with harmonic $1$-form input on $S^1$. This implies 
\[
  i_\H^* \rho(\gamma)[\infty]= \gamma_\infty. 
\]
To show that the differentials also coincide, we simply need to observe that 
\[
Q_\infty|_{\H}=\nabla \quad \text{ and }  \quad \Delta_\infty=\Delta
\]
when restricted to $\H$ under the identification  \eqref{nifty-obs}.

\end{proof}

\subsubsection{Local to global factorization map}
Now we give the quantum field theoreical interpretation of Section \ref{section-Fedosov-BV}. This will be sketchy and idea-based to avoid introducing too much technicality from \cite{Kevin-Owen}.

There is a natural map of cochain complexes 
\[
   \text{Obs}^q(U)[L] \to \text{Obs}^q(S^1)[L], 
\]
which is the factorization map from local observables to global observables. The HRG flow gives a homotopy between different scales $L$. Our definition of 
$
  [-]_\infty
$
in equation \eqref{observable-map} is precisely the cochain map
\[
   \text{Obs}^q(U)[L=0]  \to \text{Obs}^q(S^1)[L=\infty] = \text{Obs}^q. 
\]

Moreover, the evaluation map of Section \ref{sect:integration} can be thought of as the expectation map
\[
\langle - \rangle : H^0 (\text{Obs}^q) \to \R ((\hbar)).
\]
The trace map $\Tr : C^\infty (M) [[\hbar]] \to \R((\hbar))$  of Section \ref{section-trace} is then constructed as follows. Let $f \in C^\infty (M)[[\hbar]]$ and let $O_f \in \text{Obs}^q (U)$ be a cochain representative, then
\[
\Tr (f) = \int_M \int_{\gamma_\infty} [ \sigma^{-1} (f) ]_\infty = \left \langle [O_f]_\infty \right \rangle.
\] 
The fact that the map $\Tr$ vanishes on $\star$-commutators is just a consequence of the local quantum observables forming a translation invariant factorization algebra on $S^1$.
Concretely, let $U_1 \subset S^1$ be an open interval on the circle and for $\vartheta \in S^1$, let $T_\vartheta U_1 \subset S^1$ be the translation of $U_1$ by $\vartheta$, then we have an isomorphism
\[
H^\ast (\text{Obs}^q (U_1)) \cong H^\ast (\text{Obs}^q (T_\vartheta U_1)).
\]
Let $\{O_1 (L)\} \in H^\ast (\text{Obs}^q (U_1))$ and $\{O_2 (L) \} \in H^\ast (\text{Obs}^q (U_2))$ be local quantum observables on disjoint intervals on $S^1$. Then we have
\[
\langle O_1 \star O_2 \rangle = \langle O_2 \star O_1 \rangle .
\]
To explicitly see the equivalence of the two point correlation functions we simply choose $\mu, \eta, \zeta \in S^1$ and consider the resulting isomorphisms induced by translation as described by the following figure.

\[
  \figbox{0.24}{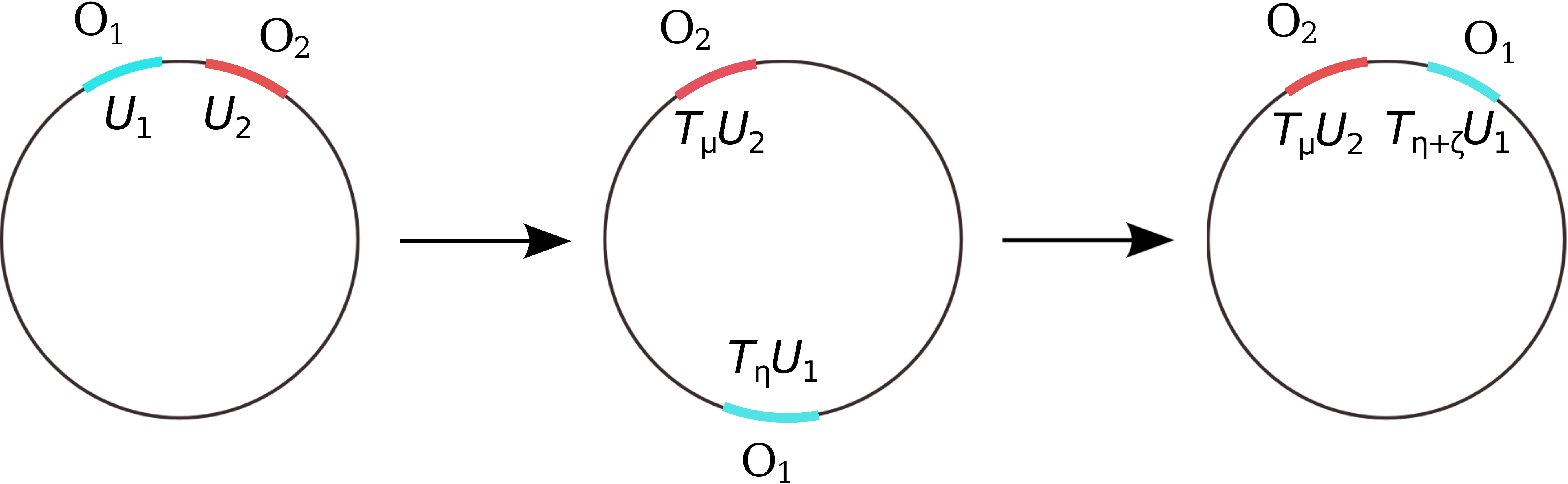}
\]

\appendix
\renewcommand*{\thesection}{\Alph{section}}

\section{Feynman diagrams}\label{Sec: Feynman}
In this section we describe the relevant Feynman diagrammatics  that are used in this paper. 
For a full exposition, see for example  \cite{QFT-graph}. 

\begin{defn}
A graph $\cG$ consists of the following data:
\begin{enumerate}
 \item A finite set of vertices $V(\cG)$;
 \item A finite set of half-edges $H(\cG)$;
 \item An involution $\sigma: H(\cG)\rightarrow H(\cG)$. The set of fixed points of this map is denoted by $T(\cG)$ and is
called the set of tails of $\cG$. The set of two-element orbits is denoted by $E(\cG)$ and is called the set of internal edges of
$\cG$;
 \item A map $\pi:H(\cG)\rightarrow V(\cG)$ sending a half-edge to the vertex to which it is attached;
 \item A map $g:V(\cG)\rightarrow \mathbb{Z}_{\geqslant 0}$ assigning a genus to each vertex.
\end{enumerate}
\end{defn}
It is clear how to construct a topological space $|\cG|$ from the above abstract data. A graph $\cG$ is called $connected$ if
$|\cG|$ is connected. The genus of the graph $\cG$ is defined to be 
\[
g(\cG):=b_1(|\cG|)+\sum_{v\in V(\cG)}g(v),
\] 
where
$b_1(|\cG|)$ denotes the first Betti number of $|\cG|$.

 Let $\E$ be a graded vector space over a base (coefficient) ring $R$. Let $\E^*$ be its $R$-linear dual. When $\E$ carries a topology, $\E^*$ denotes the continuous linear dual (see Conventions).  Let  
 $$
 \widehat\OO(\E):=\prod_{k\geq 0} \OO^{(k)}(\E), \quad \OO^{(k)}(\E)=\Sym^k_{R}(\E),
 $$ 
 denote the space of formal functions on $\E$. Let
\[
\OO^+(\E) \subset\widehat\OO(\mathcal{E})[[\hbar]]
\] 
be the subspace consisting of those
functions  which are at least cubic modulo $\hbar$ and the nilpotent ideal $\mathcal{I}$ in the base ring
$R$.  Let $F\in \OO^+(\E)[[\hbar]]$,  which can be expanded as
\[
F=\sum_{g,k\geq 0}\hbar^g F_{g}^{(k)}, \quad F_{g}^{(k)}\in \OO^{(k)}(\E).
\]
We view each $F_{g}^{(k)}$ as an
$S_k$-invariant linear map
\[
F_{g}^{(k)}: \mathcal{E}^{\otimes k}\rightarrow R.
\]

Fix a $P\in \Sym^2(\E)$ be an element of degree $0$, which will be called the \emph{propagator}. 
With $F$ and $P$, we will describe the {\it (Feynman) graph weights}
\[
W_\cG(\mathbb{P}_\epsilon^L,F)\in \OO^+(\E)
\] 
for any connected graph $\cG$. We label each vertex $v$ in $\cG$ of genus $g(v)$ and valency $k$ by
$F^{(k)}_{g(v)}$. This defines an assignment
\[
F(v):\mathcal{E}^{\otimes H(v)}\rightarrow R,
\]
where $H(v)$ is the set of half-edges of $\cG$ which are incident to $v$.
Next, we label each internal edge $e$ by the propagator 
\[
P_e=P\in\mathcal{E}^{\otimes H(e)},
\]
where $H(e)\subset H(\cG)$ is the two-element set consisting of the half-edges forming $e$. We can then contract
\[
\otimes_{v\in V(\cG)}F(v): \mathcal{E}^{H(\cG)}\rightarrow R
\]
with 
\[
\otimes_{e\in E(\cG)}P_e\in\mathcal{E}^{H(\cG)\setminus T(\cG)}
\] 
to yield an $R$-linear map
\[
W_\cG(P,F) : \mathcal{E}^{\otimes T(\cG)}\rightarrow R.
\]

\begin{defn}
We define the homotopic renormalization group flow operator (HRG) with respect to the propagator $P$ 
\[
   W(P, -): \OO^+(\E)\to \OO^+(\E), 
\]
by
\begin{equation}\label{RG-flow}
W(P, F):=\sum_{\cG}\frac{\hbar^{g(\cG)}}{\lvert \text{Aut}(\cG)\rvert}W_\cG(P, F),
\end{equation}
where the sum is over all connected graphs, and $\text{Aut}(\cG)$ is the automorphism group of $\cG$ as labelled graphs. It is straight-forward to check that $W(P,-)$ is well-defined on $\OO^+(\E)[[\hbar]]$. 

\end{defn}

Equivalently, it is useful to describe the HRG operator formally via the simple equation 
\[
e^{W(P, F)/\hbar}=e^{\hbar \partial_{P}} e^{F/\hbar}.
\]
Here $\pa_P$ is the second order operator on $\OO(\E)$ by contracting with $P$ (See Conventions). This formula is equivalent to the combinatorial one \eqref{RG-flow} by an explicit expansion of the contractions. See \cite{Kevin-book} for a thorough discussion on (homotopic) renormalization group flow operator.

\section{Configuration spaces}\label{sec: configuration}
We will briefly  recall the construction and basic facts  of  compactified configuration spaces. We refer the reader to \cites{AS2, Fulton-MacPherson} for further details. Let $M$ be a smooth manifold and let $V:=\{1,\cdots,n\}$ for some integer $n\geq 2$.  The compactified configuration space $M[V]$ is  a smooth manifold with corners as we now recall. 
Let $S\subset V$ be any subset with $|S|\geq 2$. We denote by $M^S$ the set of all maps from $S$ to $M$ and by $\Delta_S \subset M^S$ the small diagonal. The real oriented blow up of $M^S$ along $\Delta_S$, denoted as $\text{Bl}(M^S,\Delta_S)$, is a manifold with boundary whose interior is diffeomorphic to $M^S\setminus\Delta_S$ and whose boundary is diffeomorphic to the unit sphere bundle associated to the normal bundle of $\Delta_S$ inside $M^S$.

\begin{defn}\label{defn-configuration}
Let $V$ be as above and let $M_0^V$ be the configuration space of $n=|V|$ pairwise different points in $M$,
\[
M_0^V:=\{(x_1,\cdots, x_n)\in M^{V}: x_i\not=x_j\ \text{for}\ i\not=j\}.
\]
There is the following embedding:
\[
M_0^V\hookrightarrow M^V\times\prod_{|S|\geq 2}\text{Bl}(M^S,\Delta_S).
\]
The space $M[V]$ is defined as the closure of the above embedding.
\end{defn}

\begin{rmk}
We will denote $M[V]$ by $M[n]$ for $V=\{1,\cdots, n\}$.
\end{rmk}
It follows from the definition of $M[V]$ that for every subset $S\subset V$, there is a natural projection
\[
\pi_S: M[V]\rightarrow M[S].
\]
and
\[
\pi:  M[n]\to M^n. 
\]
\begin{eg}\label{example: two-point-configuration}
The space $S^1[2]$ is a cylinder $S^1\times [0,1]$, which can be explicitly constructed by cutting along the diagonal in $S^1\times S^1$. Its boundary is given by $\partial S^1[2]=S^1\times S^0$, i.e.,  two copies of $S^1$.  More explicitly, let us fix a parametrization  
\[
  S^1=\{e^{2\pi i \theta}|0\leq \theta <1\}.
\]
Then $S^1[2]$ is parametrized by a cylinder 
\[
  S^1[2]=\{(e^{2\pi i\theta},u)| 0\leq \theta<1, 0\leq u\leq 1\}. 
\]
Further, 
\[
   \pi: S^1[2]\to (S^1)^2, \quad (e^{2\pi i \theta},u)\mapsto (e^{2\pi i (\theta+u)}, e^{2\pi i \theta}).
\]
\end{eg}
We now consider the lifting of propagators to compactifications of configuration spaces.
\begin{prop}\label{prop:propagator-step-function}
Let $S^1$ be the interval $[0,1]$ with $0$ and $1$ identified. The analytic part of the  propagator $P:={P}_0^\infty(\theta_1,\theta_2)$ is the following periodic  function of $\theta_1-\theta_2\in\R\backslash\mathbb{Z}$
\begin{equation}\label{eqn:propagator-formula} 
P_0^\infty(\theta_1, \theta_2)=\theta_1-\theta_2-\frac{1}{2},\quad \text{if}\hspace{2mm}  0< \theta_1-\theta_2<1. 
\end{equation}
\end{prop}

\begin{proof}
Recall that the heat kernel on $S^1=\R/\Z$ is given explicitly by
\begin{align*}
K_t(\theta_1,\theta_2)&=\frac{1}{\sqrt{4\pi t}}\sum_{n\in\Z}e^{-\frac{(\theta_1-\theta_2+n)^2}{4t}}(d\theta_1\otimes 1-1\otimes d\theta_2)\\[1ex]
&=\sum_{n\in\Z}e^{-4\pi^2n^2t}\cdot e^{2\pi in(\theta_1-\theta_2)}(d\theta_1\otimes 1-1\otimes d\theta_2),
\end{align*}
where in the last identity we have used the Poisson summation formula. Hence,
\begin{align*}
{P}_0^\infty(\theta_1,\theta_2)&=\int_0^\infty d^*_{\theta^1}(K_t(\theta_1,\theta_2)) dt\\[1ex]
&=\int_0^\infty\left(\sum_{n\in\Z}(2\pi in)\cdot e^{-4\pi^2n^2t}\cdot e^{2\pi in(\theta_1-\theta_2)}\right)dt\\[1ex]
&=\sum_{n\in\Z\setminus \{0\}}\frac{i}{2\pi n}e^{2\pi in(\theta_1-\theta_2)}.
\end{align*}
It is easy to check that the Fourier coefficients of the function in equation (\ref{eqn:propagator-formula}) are the same as above.
\end{proof}

Since $\mathbb{P}_0^L=\mathbb{P}_0^\infty-\mathbb{P}_L^\infty$ and $\mathbb{P}_L^\infty$ is a smooth kernel, the following lemma is then immediate.

\begin{lem}\label{lem: lifting-propagator}
For all $0 <L \leq \infty$, the propagator ${P}_0^L$ can be lifted to a smooth function on $S^1[2]$. Let us denote this lift by $\widetilde{{P}}_0^L$. In particular, the restriction of $\widetilde{{P}}_0^L$ to the boundary components of $\pa S^1[2]$ are the constant functions $\frac{1}{2}$ and $-\frac{1}{2}$ respectively. 
\end{lem}

\begin{rmk}\label{remark: propagator}
 In particular, we will let $P$ denote the lifting of the $\infty$-scale propagator on $S^1[2]$. More precisely, in terms of the parametrization of $S^1[2]$ in Example \ref{example: two-point-configuration}: 
 \begin{equation}\label{eqn: infinity-propagator}
 P(\theta, u)=u-\frac{1}{2}.
 \end{equation}
\end{rmk}

Let us  recall the stratification of the configuration spaces $M[V]$. The compactified configuration space $M[V]$ can be written as the disjoint union of open strata
\[
M[V]=\bigcup_{\mathcal{S}}M(\mathcal{S})^0.
\]
Here $\mathcal{S}$ denotes a collection of subsets  of the index set $V$ such that
\begin{enumerate}
\item $\mathcal{S}$ is nested: if  $S_1,S_2$ belong to $\mathcal{S}$, then either they are disjoint or one contains the other;
\item Every subset $S \in \mathcal{S}$ is of cardinality $\geq 2$.
\end{enumerate}

A useful fact is that the open stratum $M(\mathcal{S})^0$ (and the corresponding closed stratum $M(\mathcal{S})$) is of codimension $|\mathcal{S}|$. For the consideration of the boundary integrals, we only need those strata of codimension $1$ which correspond to collections $\mathcal{S}=\{S\}$ consisting of a {\em single} subset $S\subset V$ with cardinality $\geq 2$.  We will denote such a codimension $1$ closed stratum simply by $M(S)$. Recall that when $S=V$, $M (S)^0$ can be described as the sphere bundle of the normal bundle of the small diagonal $\Delta_V\subset (M)^V$. 

\begin{eg}\label{example: boudary-circle-configuration}
  The codimension 1 strata of $S^1[m]$ can be described as follows. Let $I\subset \{1,\cdots, m\}$ be a subset. Let
\[
D_I=\{\{\theta_1, \cdots, \theta_m\}\in (S^1)^m| \theta_\alpha=\theta_\beta\ \text{for}\ \alpha, \beta\in I\}\subset (S^1)^m 
\]
be the associated partial diagonal. Then 
\[
   \pa S^1[m]=\bigcup_{|I|\geq 2} \pi^{-1}(D_I). 
\] 
\end{eg}

\section*{Acknowledgements} 

The authors gratefully acknowledge support from the Simons Center for Geometry and Physics, Stony Brook University, especially for the workshop ``Homological methods in quantum field theory'' at which some of the research for this paper was performed. R Grady and S Li  would like to thank the Perimeter Institute for Theoretical Physics for its hospitality and working conditions during the writing of parts of the current paper. Research at Perimeter Institute is supported by the Government of Canada through Industry Canada and by the Province of Ontario through the Ministry of Economic Development and Innovation.

Additionally, the authors would like to thank Kevin Costello, Owen Gwilliam, and Steve Rosenberg for feedback on various versions of the present work. We are also very grateful to the referees for many useful clarifications and suggestions.

\end{document}